\documentclass[letterpaper, 10pt,reqno]{amsart}

\usepackage[portrait,margin=3cm]{geometry}
\usepackage{mathrsfs}
\usepackage{amssymb,amsthm,amsfonts,amsbsy,amstext,latexsym}
\usepackage[reqno]{amsmath}

\usepackage{caption}
\usepackage{subcaption}

\usepackage{bbm}
\usepackage{comment}
\usepackage{algorithm}
\usepackage{algpseudocode}
\usepackage{csvsimple}

\usepackage{soul}
\usepackage{listings,lscape, epstopdf, units, textcomp, fancyhdr, url}

\usepackage{color}
\usepackage[dvipsnames]{xcolor}
\usepackage{graphicx}
\usepackage{tikz}
  \usetikzlibrary{calc}
  \usetikzlibrary{shapes.geometric}
  \tikzstyle{vertex}=[circle,fill=orange!60,minimum size=10pt,inner sep=0pt]
  \tikzstyle{tedge} = [draw,ultra thick,->,>=stealth, orange]
  \tikzstyle{esq}=[circle,fill=white,minimum size=10pt,inner sep=0pt]
  \tikzstyle{up}=[<-,>=stealth]

\usepackage[textsize=small]{todonotes}
\usepackage{enumitem}

\usepackage{mathtools}
\mathtoolsset{showonlyrefs}

\usepackage{xifthen}

\definecolor{MyDarkblue}{rgb}{0,0.08,0.50}
\definecolor{Brickred}{rgb}{0.65,0.08,0}
\usepackage{hyperref,upref}
\hypersetup{
	colorlinks=true,       
	linkcolor=MyDarkblue,  
	citecolor=Brickred,    
	filecolor=red,         
	urlcolor=cyan          
}

\newtheorem*{theorem*}{Theorem}
\newtheorem{theorem}{Theorem}[section]
\newtheorem{lemma}[theorem]{Lemma}
\newtheorem{proposition}[theorem]{Proposition}
\newtheorem{corollary}[theorem]{Corollary}

\theoremstyle{definition}

\newtheorem{remark}[theorem]{Remark}

\newcommand{\invisible}[2]{%
    \ifthenelse{\isempty{#1}}
    {}
    {#2}
}

\newcommand*{\bfi}{\mathbf{i}}
\newcommand*{\bfI}{\mathbf{I}}
\newcommand*{\bfj}{\mathbf{j}}

\newcommand*{\smlPrec}{\scalebox{.65}{$\prec$}}
\newcommand*{\smlSucceq}{\scalebox{.65}{$\succeq$}}

\renewcommand{\emptyset}{\varnothing}

\newcommand{\eps}{\varepsilon}

\title{Mixing time and isoperimetry in random geometric graphs}

\author[Marcos Kiwi]{Marcos Kiwi}
\address[Marcos Kiwi]{Univ.\ Chile, Santiago, Chile}
\email{mk@dim.uchile.cl}
\thanks{Marcos Kiwi has been supported by Centro de Modelamiento Matemático (CMM) BASAL fund FB210005 for center of excellence from ANID-Chile.}
\author[Carlos Martinez]{Carlos Martinez}
\address[Carlos Martinez]{Institute for Applied Mathematics, University of Bonn, Endenicher Allee 60, 53115 Bonn, Germany}
\email{cmartine@uni-bonn.de}
\thanks{Carlos Martinez has been supported by the Deutsche Forschungsgemeinschaft (DFG, German Research Foundation) –
Projektnummer 552316285}

\author[Dieter Mitsche]{Dieter Mitsche}
\address[Dieter Mitsche]{IMC, Pontificia Universidad Católica, Santiago, Chile}
\email{dieter.mitsche@uc.cl}
\thanks{Dieter Mitsche has been supported by Fondecyt grant 1220174.}

\begin{document}

\maketitle

\begin{abstract}
    In this paper we study the mixing time of the simple random walk on the giant component of supercritical $d$-dimensional random geometric graphs generated by the unit intensity Poisson Point Process in a $d$-dimensional cube of volume $n$. With $r_g$ denoting the threshold for having a giant component, we show that for every $\eps > 0$ and any $r \ge (1+\eps)r_g$, the mixing time of the giant component is with high probability $\Theta(n^{2/d}/r^{2})$, thereby closing a gap in the literature. The main tool is an isoperimetric inequality which holds, w.h.p., for any large enough vertex set, a result which we believe is of independent interest. Our analysis also implies that the relaxation time is of the same order.
\end{abstract}
\noindent

MSC Class: 05C80, 60D05, 05C81

\section{Introduction}\label{sec:intro}
The mixing time - the time for a walk to approach its stationary distribution - is a fundamental parameter in the study of random walks, and it has been widely investigated across various classes of random graph models.
A prominent example is the infinite component supercritical percolation on $\mathbb{Z}^d$, where Benjamini and Mossel~\cite{benjamini2003mixing} analyzed the simple random walk on the giant component of the graph restricted to a finite cube of volume $\Theta(n^d)$ and showed that its mixing time is $\Theta(n^2)$.
The mixing time of the same walk was also studied in long-range percolation on $\mathbb{Z}^d$: the case $d=1$ was first analyzed by Benjamini, Berger and Yadin in~\cite{BenjaminiLRP}, and it was later generalized to higher dimensions by Crawford and Sly~\cite{crawford2012simple}.
In yet another random graph class, the mixing time of the simple random walk in the giant component in the supercritical $d$-dimensional percolated hypercube with parameter $p=c/d$ and~$c>1$ was very recently shown to be $\Theta(d^2)$ by Anastos, Diskin, Lichev and Zhukovskii~\cite{LyubenNew}. 

The class of random geometric graphs (RGGs) that we study also falls within the category of spatial random graph classes. RGGs serve as a model for spatial networks where vertices correspond to points of a spatial point process and edges are defined via proximity rules (for a survey of these graphs see the monograph of Penrose~\cite{Pen03}). 

Above the threshold for connectivity, RGGs are relatively homogeneous, and for such graphs the mixing time of random walks is understood (see the papers of Avin and Ercal~\cite{covertimergg} and Boyd et al.~\cite{boyd2005mixing}, see also below for details on their results). In a sparser regime below the connectivity threshold,  the graph structure is inhomogeneous, making the analysis of random walks considerably more challenging:  the mathematical difficulties stem both from the spatial randomness intrinsic to RGGs as from the variability in vertex degrees. Determining the mixing time in this case had remained an open problem; our results resolve it by successfully addressing all of the aforementioned challenges.
We determine the order of the mixing time throughout the entire supercritical regime. In the process, we also establish the order of the relaxation time and prove isoperimetric inequalities that hold with high probability for all sufficiently large vertex sets. This turns out to be harder than in supercritical percolated $\mathbb{Z}^d$, as we will explain in the outline of the proof below.

\medskip
\textbf{Main results.}
We next introduce the essential definitions required to state our main results. For a given connected graph $G=(V,E)$, we  consider the \emph{continuous-time simple random walk $(X_t)_{t\ge 0}$} with rate 1. 
Let $P$ be the transition matrix of the discrete simple random walk in $G$, that is, denoting~$\deg(x)$ the degree of a vertex~$x$, let $P$ be defined by $P(x,y)=1/\deg(x)$ if $\{x,y\}\in E$, and $0$ otherwise. 
Also, let $P_n(x,y)$ be the probability that the random walk with transition matrix $P$ is at $y$ after $n$ steps, starting from $x$.
The \emph{heat kernel} $H_t$ of the continuous time random walk is  then defined as 
\[
H_t(x,y)= \sum_{n\ge 0} \frac{t^{n}e^{-t}}{n!} P_{n}(x,y) 
\]
which represents the probability that, starting from $x$, the continuous walk is at $y$ at time $t$. This walk admits a stationary distribution $\pi$, which is a probability vector on $V$ satisfying $\pi=H_t\pi$ for all~$t\ge 0$. It is well known that $\pi(x)=\frac{\deg(x)}{2|E|}$. In case the corresponding discrete random walk has a stationary distribution, the latter coincides with the stationary distribution of the continuous-time random walk.

Given two measures $\mu$ and $\nu$ over the discrete space of vertices $V$, we define the \emph{total variation distance} as 
\[
\|\nu-\mu\|_{TV} = \frac{1}{2}\sum_{x\in V} |\nu(x)-\mu(x)|.
\]
The \emph{(continuous-time) mixing time} \( \tau_{\text{mix}}(G,\varepsilon) \) on $G$
is defined as:
\[
\tau_{\textup{mix}}(G,\varepsilon) = \inf \Big\{ t \geq 0 : \max_{x \in V} \| H_t(x, \cdot) - \pi \|_{\textup{TV}} \leq \varepsilon \Big\}.
\]
In words, this is the minimum time \( t \) such that, starting from any state, the distribution of the Markov chain is within \( \varepsilon \) total variation distance of the stationary distribution. One often  chooses $\eps=\frac14$, but  when the value of $\eps$ is not relevant (but fixed), we omit $\eps$.

Another quantity of interest associated to a random walk is its \emph{relaxation time} $\tau_{\textup{rel}}$ defined as $\tau_{\textup{rel}}(G) = (1-\lambda_2)^{-1}$
where $\lambda_2$ denotes the second-largest eigenvalue (in absolute value) of the transition matrix $P$.

The model of \emph{random geometric graphs (RGGs)} is defined as follows:
For $r>0$ and $n\in\mathbb{N}$ we denote by $\mathcal{G}_n^{r,d}$ (or rather $\mathcal{G}_n$ when the dependence on $r$ and $d$ is clear)  the RGG of radius $r$ inside the $d$-dimensional cube $\Lambda_n = \left[ \frac{-n^{1/d}}{2}, \frac{n^{1/d}}{2}\right)^d$ of volume $n$ whose vertex set $\mathcal{V}_n$ is the set of points of a Poisson Point Process (PPP) of unit intensity in $\Lambda_n$ and edge set 
\begin{equation}
    \mathcal{E}_n^r = \{\{x,y\}:x,y \in \mathcal{V}_n, \lVert x-y \rVert_2 \leq r\}.
\end{equation} 
We define the infinite counterpart $\mathcal{G}^{r,d}$ similarly, but replace $\Lambda_n$ by $\mathbb{R}^d$ with one artificial vertex added at the origin, and define then  $r_g=r_g(d)$ as the infimum of $r \ge 0$ such that 
\[
\mathbb{P}(\mbox{the origin is in an  infinite component of $\mathcal{G}^{r,d}$)} > 0.
\]
It is well 
known that $0<r_g<\infty$ for every $d\ge 2$, although its exact value is not known (see~\cite{Pen03} for estimates in $d=2$). When restricting $\mathbb{R}^d$ to $\Lambda_n$, denote  by $\mathcal{L}_n$ the largest  component of $\mathcal{G}_n$ (ties broken arbitrarily). When $r > r_g$, we will refer to this component also as the \emph{giant component} or simply giant.

We say that a sequence of events $(\mathcal{A}_n)_n$ holds with high probability, or simply w.h.p., if $\mathbb{P}(\mathcal{A}_n) \to 1$ as $n \to \infty$.

\smallskip
Our main result concerns the mixing time and the relaxation time of the random walk on  the giant component of RGGs in $d$ dimensions. 
\begin{theorem}\label{thm:main_thm_bounds_over_mixing_rgg}
Let $d \in \mathbb{N}$, let $\varepsilon > 0$  and $r\ge (1+\varepsilon)r_g$.
Then, w.h.p. 
\[
\tau_{\textup{rel}}(\mathcal{L}_n)=\Theta(\tau_{\textup{mix}}(\mathcal{L}_n)) = \Theta\left(\frac{n^{2/d}}{r^{2}}\right).
\]
\end{theorem}
While the result is not surprising as it agrees for constant radius $r$ with the value obtained on the giant component of supercritical percolation (on a restriction to a finite large cube of $\mathbb{Z}^d$), 
our proof, especially of the upper bound on the mixing time for $d \ge 2$, is (perhaps also surprisingly) delicate. Theorem \ref{thm:main_thm_bounds_over_mixing_rgg} shows that the mixing time is throughout the whole regime the same function of the radius (at least the leading order). This is in stark contrast to the cover time of $\mathcal{L}_n$, which is the average number of steps that the walk takes to visit every vertex of the graph. Denoting by $r_c$ the connectivity threshold, the second and third author recenty proved that the latter is of order $\Theta(n\log n)$ for $r<(1-\eps)r_c$ and of order $\Theta(n\log^2n)$ when $r>(1+\eps)r_c$, thus showing a phase transition at $r_c$ (see~\cite{martinez2025jump}). 

Central to our proof for the $d\ge 2$ case is the next stated result, which includes an isoperimetric inequality that is of interest in its own right.
Isoperimetric inequalities are closely connected to several properties of random walks.  In particular, they can be used to prove bounds on the heat kernel decay (see for instance Nash's inequality in Theorem 3.3.11 of~\cite{saloff2006lectures}). Related results for RGGs (in particular, on the Cheeger constant) so far were only given above the connectivity threshold, see~\cite{MullerPenrose2020}. 

\begin{theorem}\label{thm:isoperimetric_inequality_rgg}
    Let $\varepsilon,\delta>0$, $d \ge 2$, and $r\ge (1+\varepsilon)r_g$. Then, there are constants $C_1 > 0$ and~$c_1(C_1) > 0$ such that w.h.p.~the following holds: For every connected set $A$ of vertices of $\mathcal{L}_n$ for which $C_1(r^{2d+1}(\log n)^{\frac{d}{d-1}})^d\le |A|\le(1-\delta)|\mathcal{L}_n|$, the number of edges between $A$ and~$A^c$ is at least~$c_1|A|^{(d-1)/d}r^{d+1}$.

    Moreover, there exists $C' > 0$, such that if $r^d \ge C'\log n$, for any $c_2 > 0$, the conclusion holds under the weaker hypothesis $|A| \ge c_2r^d$.
\end{theorem}

The order of the lower bound stated on the size of the edge-boundary in Theorem~\ref{thm:isoperimetric_inequality_rgg} corresponds to the ``average'' number of edges connecting the vertices inside a ball to those outside. The requirement over the minimum size of the set $|A|$ is natural: due to randomness and spatial irregularities, one could in principle find sets connected to the giant component by only a single edge. Our result shows that this does not occur for large enough sets. We believe that for the regime $r^d=O(\log n)$, the order of the lower bound on $|A|$ is not optimal (although it suffices for our applications). In contrast, when~$r^d=\Omega(\log n)$, the order of the lower bound on $|A|$ becomes optimal. 

As an immediate corollary of Theorem~\ref{thm:main_thm_bounds_over_mixing_rgg} we obtain that the mixing time does not exhibit \emph{cutoff}: recall that a sequence of random walks exhibits cutoff if for all $\eps > 0$, 
$$
\lim_{|G| \to \infty}
\frac{t_{\textup{mix}}(G, \eps)}{t_{\textup{mix}}(G, 1-\eps)}=1.
$$ (in words, the total variation distance jumps rapidly from $1$ to $0$ in a very tiny interval). We obtain a corollary that, as in the case of the percolated grid, there is no cutoff on the giant component of RGGs. 

\begin{corollary}
    Let $d \in \mathbb{N}$, $\varepsilon > 0$  and $r\ge (1+\varepsilon)r_g$. The continuous-time simple random walk on~$\mathcal{L}_n$ does not exhibit cutoff.
\end{corollary}
\begin{proof}
This follows immediately from Theorem~\ref{thm:main_thm_bounds_over_mixing_rgg} together with the well known fact (see for example Lemma 5 of~\cite{Justin}) that any sequence $(\tau_{\textup{mix}}(\mathcal{L}_n))_n$ that exhibits cutoff must satisfy the condition that~$\tau_{\textup{rel}}$ is asymptotically of larger order than $ \tau_{\textup{mix}}.$ 
\end{proof}
We conclude this section by emphasizing a key distinguishing feature of our analysis of the mixing time of the simple random walk on random geometric graphs—namely, that unlike existing results, it applies to the full supercritical regime and remains tight under the standard parameter choice in the definition of mixing time.

\medskip\noindent%
\textbf{Proof strategy.} Most of this manuscript consists in showing the upper bound on the mixing time of Theorem~\ref{thm:main_thm_bounds_over_mixing_rgg}. At the heart of the proof for $d \ge 2$ (for $d=1$ the proof is considerably simpler) is our second main result (the isoperimetric inequality of Theorem~\ref{thm:isoperimetric_inequality_rgg}) that shows that for all large enough subsets $A$ of the giant component's vertex set (of size which is at most $(1-\delta)$-fraction of the vertices of the giant) the number of edges between~$A$ and $A^c$ is $\Omega(|A|^{(d-1)/d} r^{d+1})$. The upper bound then follows by an average conductance method pioneered by Lovász and Kannan~\cite{lovasz1999faster}. To prove this isoperimetric inequality, we divide $\Lambda_n$ into tiles of side length proportional to the radius $r$, distinguishing cases when~$A$ is shattered around many tiles and the case when~$A$ is relatively concentrated in a few tiles. To this end, the analysis is divided into three regimes according to the value of the radius.

Above the connectivity threshold, that is, for $r^d=\Omega(\log n)$, the proof is a relatively easy exercise, as the number of vertices per cell is well concentrated. 

For $r \ge r'$ with $r'$ being a large enough constant and $r^d=O(\log n)$, the proof is much more delicate: one important new ingredient is to show that the total degree of vertices incident to $A$, for $A$ relatively large, is $O(|A|r^d)$. A simple counting of lattice animals is not enough, and we develop a container method, which to the best of our knowledge is new in a geometric setup: we need to encapsulate lattice animals inside much larger tiles in order to obtain the desired bound. The analogous result is immediate on percolated $\mathbb{Z}^d$ (since the degree of every vertex is trivially bounded by $2d$) as well as in the dense case of RGGs; here we need to group cells according to their density and count ways for choosing containers of cells. Next, in a series of lemmas we have to overcome several geometric obstacles particular to  RGGs and distinguish the analysis of sets $A$ according to the fact whether $A$ is contained in relatively few tiles with many vertices belonging to $A$ or whether $A$ is rather spread out. Especially the first type of sets requires delicate geometric arguments in order to identify sufficiently many adjacent pairs of tiles containing the desired number of edges between $A$ and $A^c$. 

Finally, for $(1+\eps)r_g \le r \le r'$, we need to expand the ideas of the previous case by an additional renormalization scheme. More in detail, we define tiles of two different sizes (of length roughly of the order of the radius) and define them as good if inside them there is one unique large component inside the tile. We show that the set of good tiles then dominates a supercritical percolation process, for which we then show the desired number of edges between $A$ and $A^c$ exist w.h.p. 

Since for any $\eps > 0$, by Theorem 12.5 of~\cite{levinperes}, we have
$(\tau_{\textup{rel}}(G)-1) \log \left( \frac{1}{2\varepsilon} \right) \le t_{\mathrm{mix}}(G,\varepsilon)$, it remains to show a lower bound on $\tau_{\textup{rel}}(G)$. This proof is comparably considerably simpler and in spirit of the corresponding lower bound proven by Benjamini and Mossel~\cite{benjamini2003mixing} on the mixing time of a simple random walk on supercritical percolation clusters inside cubes in $\mathbb{Z}^d$: we show that a constant fraction of vertices is close to the origin, having thus a small graph distance to the closest vertex to the origin, whereas a constant fraction of vertices is far from the origin. Hence, the expected distance has some non-trivial variance, and by Lemma 2.2 of~\cite{benjamini2003mixing} the result follows.

\smallskip\noindent%
\textbf{Further related work.}  
 Random geometric graphs  were introduced by Gilbert~\cite{Gil61} as a model for telecommunication networks. Since then, they have attracted considerable attention, both from the theoretical side (we refer here only to the already mentioned monograph by Penrose~\cite{Pen03}) and from the perspective of applications (see for example~\cite{Aky02, Nek07, FLMPSSSvL19}). 
On the other hand, the mixing time of random walks has been studied in many different contexts (see the book of Levin and Peres~\cite{levinperes} for a comprehensive treatment). 
Preliminary results on the mixing time of a random walker on RGGs were given in the already mentioned papers of Boyd et al.~\cite{boyd2005mixing} and Avin and Ercal~\cite{covertimergg}, where it was shown that w.h.p.\ $\tau_{\textup{mix}}(\mathcal{L}_n)=O(r^{-2}\log n)$ for $r$ being asymptotically larger than the connectivity threshold (a constant factor larger than the connectivity threshold, respectively). In fact,~\cite{boyd2005mixing} gave also a corresponding lower bound of the same order, but they considered the non-standard choice~$\eps=n^{-\alpha}$ (for some $\alpha > 0$) appearing in the definition of $\tau_{\textup{mix}}$: for such a choice of the parameter $\varepsilon$ known lower bounds for mixing times immediately imply that the upper bound in~\cite{boyd2005mixing} is tight (for the standard constant $\varepsilon$ case, tightness of the bound remained open -- our main result closes precisely this gap). Their result was later also generalized to geographical threshold graphs~\cite{beveridge2011mixing} (see the paper for the relevant definition). Furthermore, the mixing time of exponential random graphs was studied in~\cite{Bhamidi}. Bounds on the mixing time via the spectral gap (that is, $1-\lambda_2$) were also proven for long-range percolation  and random hyperbolic graphs (see~\cite{BenjaminiLRP}, the already mentioned paper of~\cite{crawford2012simple}, or also~\cite{KM18}).

Besides the papers already mentioned in the introduction, the mixing time was also studied in non-geometric random graph classes, and in particular, for the simple random walk in the giant component of supercritical $\mathcal{G}(n,p)$ graphs: first, in the strictly supercritical regime, that is, in the case when $p=c/n$ for fixed $c > 1$, Fountoulakis and Reed in~\cite{fountoulakis2008evolution} and independently Benjamini, Kozma and Wormald in~\cite{benjamini2014mixing} showed that the mixing time of the giant component is of order $\log^2 n$. For critical random graphs, that is, $\mathcal{G}(n,p)$ graphs with $p=(1+\lambda n^{-1/3})/n$ with $\lambda \in \mathbb{R}$, 
Nachmias and Peres proved in~\cite{Nachmias} that throughout this critical window the
mixing time of the largest component is w.h.p.\ of order $n$. Finally, in the near-critical case, that is, for $p=(1+\eps)/n$ with $\lambda=\eps^3 n \to \infty$ and $\lambda=o(n)$, Ding, Lubetzky and Peres showed in~\cite{ding2012mixing} that the mixing time on the largest component is with high probability of order $(n/\lambda)\log^2 \lambda$. In fact, the authors therein also showed that this is w.h.p.\ the order of the largest
mixing time over all components, both in the slightly supercritical
and in the slightly subcritical regime (that is, with $p=(1-\eps)/n$, with $\lambda=\eps^3 n$ and $\lambda=o(n)$).

Further results on random walks on supercritical percolation on $\mathbb{Z}^d$, including heat kernel decays, were given by Mathieu and Remy ~\cite{mathieu2004isoperimetry} and by Barlow~\cite{barlow_percolationHQ}. Quenched invariance principles for the simple random walk on percolation clusters in $\mathbb{Z}^d$ were given by Berger and Biskup in~\cite{berger2007quenched}, showing convergence for almost every percolation configuration to Brownian motion.
In long-range percolation,  scaling limits of the corresponding random walk were then (not for all ranges of parameters) given by Crawford and Sly in~\cite{crawford_sly_scaling_limitlrp}. The picture was more recently then completed by Berger and Tokushige~\cite{berger2024scaling}. Random walks were also analyzed in correlated percolation models such as the vacant set in random interlacements by Procaccia, Rosenthal and Sapozhnikov in~\cite{procaccia2016quenched} where the authors gave quenched invariance principles as well.

\medskip\noindent%
\textbf{Conventions and notation.}
In this section we summarize the, mostly standard, conventions we will adhere to. 
All graphs considered in this paper are assumed to be finite and simple (that is, undirected graphs without loops or multiple edges).
Two vertices $u, v \in G$ are said to be \emph{adjacent} or \emph{neighbors} if~$\{u, v\} \in E(G)$ and $u, v$ are called \emph{endvertices} of $\{u,v\}$. 
Two edges of $G$ are called \emph{disjoint} if they have no endvertices in common.

The graph distance in $G$ (or simply distance) between two vertices~$u, v$ is denoted $\mbox{d}_G(u,v)$ (or simply~$\mbox{d}(u,v)$) and defined as the minimum non-negative integer $\ell$ for which there is a sequence $v_0,...,v_\ell$ of vertices of $G$
such that $u=v_0$, $v=v_\ell$, and $v_{i-1}$ is adjacent to $v_{i}$ for all~$i\in [\ell]$, and~$+\infty$ if no such sequence exists.
A subset of vertices $U \subseteq G$ is said to be a \emph{component of~$G$}
(or just \emph{component} when~$G$ is clear from context) if the distance in $G$ between any pair of vertices in $U$ is finite and $U$ is an inclusion-maximal set. Two vertices (respectively, a set of vertices) of a graph~$G$ are said to be \emph{connected} if they belong to the same component of $G$.

The \emph{$k$-transitive closure} of a graph $G$, denoted $G^k$, is a graph over the same vertex set of $G$ where an edge exists between two distinct vertices $u,v$ if 
$\mbox{d}_G(u,v)$ is at most $k$. 
We say that $L$ is a \emph{$k$-lattice animal} in $G$ if $L$ is a connected set of vertices of $G^k$. When $k=1$, we simply call $L$ a lattice animal.

We use standard asymptotic notation. Specifically, if $(a_n)_{n}$, $(b_n)_{n}$ are sequences of real numbers, we write
$a_n = O(b_n)$ if for some positive constant $C > 0$ and non-negative integer $n_0$ it holds that~$|a_n| \le C|b_n|$
for all $n > n_0$. Also, we write $a_n = \Omega(b_n)$ if $b_n = O(a_n)$, and $a_n=\Theta(b_n)$ if~$a_n = O(b_n)$ and
$a_n=\Omega(b_n)$. We write $a_n = o(b_n)$ if~$a_n/b_n\to 0$ as $n\to\infty$ and $a_n=\omega(b_n)$ if~$b_n=o(a_n)$. 
Finally, for any event $\mathcal{E}$, we use $\mathbf{1}_{\mathcal{E}}$ to denote the indicator function that is equal to 1 when $\mathcal{E}$ holds and $0$ otherwise.


\medskip\noindent%
\textbf{Organization.}
In Section~\ref{sec: Prelim}, we state preliminary results on Markov chains, on Bernoulli percolation as well as on RGGs, and we also prove auxiliary lemmas that will be used later on. In particular, we extend known results on isoperimetric inequalities for percolation clusters in cubes, and we give an upper bound on the total degree of every sufficiently large connected set in a supercritical RGG. The main section of this paper is Section~\ref{sec:Upper}, where we prove the upper bound on the mixing time for dimension $d \ge 2$  by establishing the mentioned isoperimetric inequalities. Section~\ref{sec:lowerBnd} then complements the previous section with a matching lower bound on the relaxation time of the same order. In Section~\ref{sec:1D}, we address the (arguably much simpler) case $d=1$.
%

\section{Preliminaries and auxiliary lemmas}\label{sec: Prelim}
In this section we state results that we rely on as well as prove auxiliary results that will come in handy later on. 
\subsection{Markov chains and probability}
We start with some basic results on Markov chains.
We consider the continuous-time \emph{simple random walk} with constant rate over a connected graph $G=(V,E)$ that, at each step, waits for an interval of time whose length follows an exponential distribution.
Recall that  $\pi(x)=\mbox{deg}(x)/(2|E|)$ for all $x\in V$ is the stationary distribution of the  random walk. For $A\subseteq V$, let $\pi(A)=\sum_{x\in A}\pi(x)$.
A key tool in our investigation will be the \emph{conductance function} which is defined as
\[
\varphi(t) = \min_{\{A\subseteq V \colon 0<\pi(A)\leq t\}} \frac{Q(A,A^c)}{\pi(A)\pi(A^c)}, 
\qquad \text{where} \qquad 
Q(A,A^c) = \frac{|E(A,A^c)|}{2|E|}.
\]
The value of $\varphi(t)$ at $t=1/2$ is the (classical) so called \emph{Cheeger constant}. The following upper bound on the mixing time in terms of the conductance function and Cheeger's constant first appeared in~\cite[Theorem 2.1]{lovasz1999faster} and was later corrected by Montenegro (see  of~\cite[Theorem 3.1]{montenegro2002faster}):

\begin{theorem}\label{thm:upper_bound_mixing_lovasz_kannan}
For every connected graph $G=(V,E)$ let $\pi_0=\min_{x\in V}\pi(x)$ and $\pi_1$ denote the minimum of $\pi(X)$ where $X\subseteq V$ contains a vertex $x\in V$ such that $Q(\{x\},X)>1/2$ (that is, making a step from $x$ more than half the time we stay in $X$). Then, 
\[
\tau_{\textup{mix}}(G)=O\left( \int_{\pi_0}^{1/2} \frac{1}{t \varphi^2(t)}dt + \frac{1}{\varphi(1/2)} \right).
\]
and
\[
\tau_{\textup{mix}}(G)=O\left( \int_{\pi_1}^{1/2} \frac{1}{t \varphi^2(t)}dt + \frac{1}{\varphi(1/2)} \right).
\]
\end{theorem}
Although the above result was originally stated for the discrete time simple random walk, it also holds for the continuous time random walk (see for example Theorem 3.2 of \cite{montenegro2002faster}).

Finally, we state the following useful large deviation estimate for sums of independent identically distributed Bernoulli random variables.
\begin{lemma}\label{lem:chernoff_big_constant_exponent}
    Let $X\sim \textup{Bin}(n,p)$. For any $0<\alpha<1$ and $\alpha'>0$ there is a $p<1$ such that
    \[
    \mathbb{P}(X<\alpha\cdot n) < \exp(-\alpha' n).
    \]   
\end{lemma}
\begin{proof}
    The proof uses the following Chernoff type bound (see Theorem 1 of \cite{arratia1989binomial})
    \[
    \mathbb{P}(X < \alpha \cdot n) < \exp\left(-n \cdot D(\alpha \,\|\, p)\right),
\]
    where $D(x \,\|\, y)=x\log (x/y) + (1-x)\log((1-x)/(1-y))$. To conclude, observe that for any $0<\alpha<1$ 
    one can choose $p$ close enough to 1 so that $D(\alpha\,\|\, p)\ge \alpha'$.
\end{proof}


\subsection{Percolation, connectivity and isoperimetric inequalities}\label{sec:percolation}
This section contains topological and probabilistic results mainly in percolation theory which will be useful in proving statements about the conductance function of RGGs. When referring to percolation below, we always mean \emph{Bernoulli site percolation} which we describe next:  
For a positive integer $m$, let~$\Lambda_m$ 
be the $d$-dimensional (continuous) axis-parallel cube in $\mathbb{R}^d$ of volume $m$, that is,
$\Lambda_m=\left[-\frac{m^{1/d}}{2},\frac{m^{1/d}}{2}\right)$.
Let~$\Phi_m$ be the discrete counterpart of $\Lambda_m$, specifically, let $\Phi_m=\Lambda_m\cap\mathbb{Z}^d$. 
Throughout, we abuse notation and consider $\Phi_m$ as the graph with vertex set $\Lambda_m\cap\mathbb{Z}^d$ where two vertices are adjacent
if and only if they differ in a single coordinate by exactly~$1$.
Note that the order of $\Phi_m$ is approximately~$m$. In fact, 
  $|\Phi_m|/m\to 1$ as~$m\to\infty$. Vertices of~$\Phi_m$ are called \emph{sites}
and we are given $p\in [0,1]$. Each site is \emph{open} with probability~$p$ and \emph{closed} otherwise, independently of all other sites. We let~$\Phi_m^p$ be the (random) set of open sites in $\Phi_m$ and refer to it as percolation (or percolation process) in $\Phi_m$ with parameter $p$. We abuse notation and identify~$\Phi_m^p$ with the subgraph of $\Phi_m$ induced by the set of open sites.
Furthermore, we denote by~$F_m$ the largest component of $\Phi_m^p$. 

It is known that there exists a critical parameter $p_c(\mathbb{Z}^d)$ in the sense that if $p>p_c(\mathbb{Z}^d)$, then there exists a unique infinite component with probability 1~in percolation in $\mathbb{Z}^d$, and we will denote percolation with $p > p_c(\mathbb{Z}^d)$ as supercritical percolation. We denote by $\theta(p)$ the \emph{percolation probability}, that is, the probability that the origin belongs to the infinite component. The following result is well known (for the first part see Theorem 1.2 of~\cite{pisztora1996surface}): 
\begin{theorem}\label{thm:giant_Bernoulli_site_percolation}
There exists $p_c(d)$ such that w.h.p.~if $p>p_c(d)$ and $\eps>0$, then the largest component of $\Phi_m^p$ is of size at least $(\theta(p)-\eps) m$,
and the second-largest component of $\Phi_m^p$ is of order~$(\log m)^{d/(d-1)}$.
\end{theorem}
Next, we consider $\Phi_m^p$ and establish that w.h.p.~the number of open sites in any given  sufficiently large $k$-lattice animal in $\Phi_m$ is at least a constant fraction of its size.
To show this, as well as for other results we shall establish afterwards, we recall the following well known fact:
the number of~$k$-lattice animals of size $\ell$ in $\Phi_m$ is bounded above by $|\Phi_m|(c_k(d))^\ell$ for some positive constant $c_k(d)$ that depends on both the dimension $d$ and $k$ (this is an immediate consequence of Lemma 9.3 of~\cite{Pen03}). 
We now formalize the claim made at the beginning of this paragraph. Its proof is a straightforward adaptation of Proposition 2.10 of~\cite{benjamini2003mixing}.
\begin{lemma}\label{lem:percolation_connected_animals}
Fix a positive integer $k$. For any $0<\eps<1$, there exists $p^*=p^{*}(\eps)<1$ large enough and $\beta=\beta(\eps)$, such such that for percolation in $\Phi_m$ with $p > p^*$, 
with probability going to $1$ as $m\to\infty$, every $k$-lattice animal of size $\ell\ge\beta\log m$ contains at least $(1-\varepsilon)\ell$ open sites.
\end{lemma}
\begin{proof}
    Observe that the number of open sites in a fixed $k$-lattice animal $L\subseteq\Phi_m$ is a binomial random variable of parameters $|L|$ and $p$. Taking $p$  close enough to 1 so that Lemma~\ref{lem:chernoff_big_constant_exponent} is applicable for~$\alpha=1-\varepsilon$ and $\alpha'=\log(c_{k}(d)+1)$, we deduce that 
    \[
    \mathbb{P}(\mbox{Bin}(|L|,p)<(1-\varepsilon)|L|) <\left(\frac{1}{c_k(d)+1}\right)^{|L|}.
    \]
    Hence, since the number of $k$-lattice animals of size~$\ell$ in $\Phi_m$ is at most $|\Phi_m|(c_k(d))^{\ell}$, by a union bound, the probability that some $k$-lattice animal $L$ of length at least $\beta\log m$ has less than $(1-\varepsilon)|L|$ open sites is
    \[
    \sum_{\ell\ge\beta\log m} |\Phi_m|\left( \frac{c_k(d)}{c_k(d)+1}\right)^{\ell} 
    = O\Big(|\Phi_m|\left(\frac{c_k(d)}{c_k(d)+1}\right)^{\beta\log m}\Big),
    \]
    which for $\beta$ large enough goes to $0$  as $m\to \infty$.
\end{proof}
Given a set $K\subseteq \Phi_m$, we denote by $\partial^{+}K$ and $\partial^{-}K$ the \emph{external} and \emph{internal vertex-boundaries} of~$K$, respectively, that is,
\[
\begin{aligned}
\partial^{+}K 
&= \{\,x \in \Phi_m \setminus K \colon \text{$\exists y\in K$, $x$ and $y$ are adjacent in $\Phi_m$}\,\},\\
\partial^{-}K 
&= \{\,x \in K \colon \text{$\exists y\in \Phi_m \setminus K$,  $x$ and $y$ are adjacent in $\Phi_m$}\,\}.
\end{aligned}
\]
Also, we denote by $E(K,K^c)$ the \emph{edge-boundary} of $K$:
\[
E(K,K^c) = \big\{\{x,y\}\in E(\Phi_m) \colon \text{$x\in K$ and $y\in K^c$}\big\}.
\]
Note that the notion of connectedness in $\Phi_m$ is determined by the $\ell_1$ distance. 
Specifically, $x,y\in \Phi_m$ are adjacent if and only if the $\ell_1$ distance between $x$ and $y$ equals $1$ (that is,~$\|x-y\|_1=1$). 
The notions of external and internal vertex-boundaries, as well as edge-boundaries can also be characterized through this lens, for example, 
$x\in\partial^+K$ if and only if there is a $y\in K$ such that $\|x-y\|_1=1$.
For technical reasons, we will need to work with the weaker notion of connectedness that arises when replacing $\ell_1$ by the $\ell_\infty$ distance.  
In the latter case, we say that $K$ is \emph{*-connected} if for every $x,y\in K$ there is a sequence of vertices $v_0,...,v_\ell\in K$ such that $u=v_0$, $v=v_\ell$ and $\|v_{i-1}-v_i\|_\infty=1$ for all $i\in [\ell]$. Clearly, if $K$ is connected, then it is also *-connected.
Also, we let
$\partial_*^+ K$ and $\partial_*^- K$ denote the \emph{external} and \emph{internal *-vertex-boundaries} of $K$, respectively, that is,
\begin{align*}
  \partial_*^+ K & = \big\{ x\in\Phi_m\setminus K \colon \exists y\in K, \|x-y\|_\infty=1\big\}, \\
  \partial_*^- K & = \big\{ x\in\Phi_m \colon \exists y\in\Phi_m\setminus K, \|x-y\|_\infty=1\big\},
\end{align*}
and define the \emph{*-edge-boundary} of $K$ as
\[
E_*(K,K^c) = \big\{\{x,y\} \colon \text{$x\in K$, $y\in K^c$ and $\|x-y\|_\infty=1$} \big\}.
\]

We will use known isoperimetric inequalities relating the size of $K\subseteq\Phi_m$ and the sizes of its internal, external and edge-boundary.
The next statement follows directly from the proof of the isoperimetric inequality in Lemma 9.9 in Penrose~\cite{Pen03} which states lower bounds only on the sizes of the internal and external boundaries. However, the latter are derived from a lower bound on the size of the edge-boundary. Since we will need the latter, we explicitly state it adapted to our notation:%
\footnote{The asymptotic (in $m$) nature of our result, in contrast to Lemma 9.9 in Penrose~\cite{Pen03}, is due to the fact that, in our context, the order of $\Phi_m$ is only approximately, but not, equal to $m$ except when $m^{1/d}/2$ is an integer.} 
\begin{lemma}\label{lem:isoperimetric_in_lattice_z_d}
Assume $m$ is large enough.  For every $0<\varepsilon<1$ there is a positive constant $\gamma$ (depending on $\varepsilon$ and $d$ but independent of $m$), such that the following holds: 
for every $K\subseteq\Phi_m$ (not necessarily connected), of size at most  
  $\frac23(1-\varepsilon)m$, we have
  \[
  |E(K,K^c)| \geq \gamma |K|^{(d-1)/d}.
  \]
  Moreover,
  \[
  \min\big\{|\partial^- K|,|\partial^+ K|\big\} \geq (\gamma/2d)|K|^{(d-1)/d}.
  \]
\end{lemma}


We will also need the following fact concerning *-components whose complement is a $k$-lattice animal.

\begin{lemma}\label{lem:complement_star_connected_is_k_lattice_animal}
For a fixed positive integer $k$ and every *-connected set $K\subseteq\Phi_m$, if $K^c$ is a $k$-lattice animal, then $\partial^{+}_* K$ and $\partial^{+}K$ are $k'$-lattice animals for some $k'=k'(k)$.  
\end{lemma}
\begin{proof}
Let $\mathfrak{S}$ be the collection of *-components of $K^c$. 
Since~$K^c$ is a $k$-lattice animal, for any $S\in\mathfrak{S}$, it holds that 
\[
\mbox{d}_{\Phi_m} (S, K^c\setminus S)=\min\{\mbox{d}_{\Phi_m}(x,y) \colon x\in S, y\in K^c\setminus S\} \leq k.
\] 
We claim that for any $S\in\mathfrak{S}$, there is an $S'\in\mathfrak{S}$ and $x\in \partial^{-}_* S$ and $y \in \partial^{-}_* S'$ for which $\mbox{d}_{\Phi_m}(x,y)=\mbox{d}_{\Phi_m}(S,K^c\setminus S)\le k$. Indeed, if the distance between sets $S$ and $K^c\setminus S$ is attained at~$(x,y)$ with $x\in S \setminus \partial^{-}_*S$ and $y\in K^c\setminus S$, all neighbors of $x$ would belong to $S$ and necessarily one of them must be at distance $\mbox{d}_{\Phi_m}(x,y)-1$ from $y$, contradicting the minimality of $\mbox{d}_{\Phi_m}(x,y)$. 
Thus, it must hold that~$x\in\partial_*^- S$.
Similarly, one can show that $y\in\partial_*^-(K^c\setminus S)$. This establishes the claim.

Now, observe that $\partial^{+}_*K= \bigcup_{S\in\mathfrak{S}} \partial^{-}_* S$.
By definition of $\mathfrak{S}$, we know that $S\in\mathfrak{S}$ is *-connected.
Also, $S^c=K\cup(\cup_{S'\in\mathfrak{S}\setminus\{S\}}S')$ is~*-connected (because $K\cup S'$ is *-connected for every $S'\in\mathfrak{S}$).
Thus, Lemma~9.6 of~\cite{Pen03} implies that~$\partial^-_* S$ is *-connected.
(Note: Although the cited lemma is stated for connected sets, the proof is valid if the connectedness assumption is replaced by *-connectedness). 
Then, by recalling that a *-connected set is also a $2d$-lattice animal, we see that $\partial^+_*K$ is a~$k'$-lattice animal for some suitable $k'(k)>k$, because it is the union of *-connected sets that are at distance at most $k$ in $\Phi_m$ from each other. The assertion regarding $\partial^{+}K$ follows by the same idea. 
\end{proof}

Next, we establish an analog of the previous result. 
It provides a lower bound, in terms of the size of a connected set $K\subseteq\Phi_m$,  for the number of disjoint edges on the edge-boundary of $K$ all of whose end-vertices are open sites.

\begin{lemma}\label{lem:pair_of_open_sites_boundary}
    Fix a positive integer $k$. For percolation in $\Phi_m$ with parameter $p>p^*$ for some~$p^*<1$ large enough,  there exist $\delta, \delta'>0$ 
    such that with probability going to $1$ as $m\to\infty$, the following holds: for all *-connected sets $K\subseteq \Phi_m$ for which $K^c$ is a $k$-lattice animal and $\min\{|K|,|K^c|\}>\delta (\log m)^{d((d-1)}$,
    there are at least~$\delta'\min\{|K|^{(d-1)/d},|K^c|^{(d-1)/d}\}$ disjoint edges in~$E(K,K^c)$ each of whose endvertices are open sites. 
\end{lemma}
\begin{proof}
We say that $K$ is \emph{dangerous} if it satisfies the following conditions: $K$ is *-connected,  $K^c$ is a $k$-lattice animal, and $\min\{|K|,|K^c|\}>\delta (\log m)^{d/(d-1)}$. We also say that $(F_-,F_+)$ is a \emph{feasible} pair if there is a dangerous~$K$ such that $\partial^-K=F_-$ and $\partial^+K=F_+$. If in addition, the size of $F_+$ is $\ell$, we say that $(F_-,F_+)$ is an~$\ell$-feasible pair.

Let $\gamma$ be as in Lemma~\ref{lem:isoperimetric_in_lattice_z_d}.
Not that if $K$ is dangerous, applying Lemma~\ref{lem:isoperimetric_in_lattice_z_d} to the smaller set among $K$ and $K^c$ 
and using the fact that 
$\partial^+K=\partial^-K^c$ and $\partial^-K=\partial^+K^c$, we get
\begin{equation}\label{eqn1:lem:pair_of_open_sites_boundary}
\min\{|\partial^{-}K|,|\partial^{+}K|\}\geq (\gamma/2d)\min\{|K|^{(d-1)/d},|K^c|^{(d-1)/d}\}.
\end{equation}
Also, observe that there must be at least one edge from $E(K,K^c)$ incident to each vertex in $\partial^+K$.
Thus,
\begin{equation}\label{eqn2:lem:pair_of_open_sites_boundary}
|E(K,K^c)| 
\geq |\partial^+K|.
\end{equation}

Now, fix $\delta'=\gamma/(16d^3)$.
Let $\mathcal{B}_{(F_-,F_+)}$ be the event ``there is a dangerous $K$ such that $(\partial^-K,\partial^+K)=(F_-,F_+)$ and there are less than $\delta'\min\{|K|^{(d-1)/d},|K^c|^{(d-1)/d}\}$ disjoint edges in $E(K,K^c)$ each of whose endvertices are open sites''.
Define $\mathcal{B}$ (for ``bad'') as the union over all feasible pairs $(F_-,F_+)$ of the events $\mathcal{B}_{(F_-,F_+)}$.
Our goal is to show that the probability that $\mathcal{B}$ occurs goes to $0$ as $m\to\infty$.

We next establish four separate claims which together will imply that the probability of the event~$\mathcal{B}$ tends to $0$ as $ m \to \infty$.

The first claim is that the number of $\ell$-feasible pairs is at most $|\Phi_m|(4^dc_{k'}(d))^\ell$. Indeed, let $(F_-,F_+)$ be an $\ell$-feasible pair. 
Observe that $F_+$ is a $k'$-lattice animal for a suitable~$k'$ (since by definition of feasible pair there is a *-connected set $K$ 
such that $K^c$ is a $k$-lattice animal, so we may apply Lemma~\ref{lem:complement_star_connected_is_k_lattice_animal} to~$K$ and conclude that~$F_+=\partial^+K$ is a $k'$-lattice animal for some $k'=k'(k)$).
Now, recall that the number of $k'$-lattice animals of size $\ell$ in~$\Phi_m$ is at most $|\Phi_m|(c_{k'}(d))^\ell$ so there are at most this many possible choices for~$F_+$.
Moreover, there are at most $2^{2d|F_+|}=4^{d\ell}$ choices for~$F_-$ given $F_+$ of size $\ell$ (because at most $2d$ edges are incident to any vertex of~$\Phi_m$ and any vertex in~$F_-$ must be adjacent to a vertex in $F_+$ by definition of feasible pair).
Hence, the number of~$\ell$-feasible pairs is at most~$|\Phi_m|(4^dc_{k'}(d))^\ell$.

The second claim is that if $(F_-,F_+)$ is an $\ell$-feasible pair, then $\ell\geq\beta\log m$ for 
$\beta=\delta^{(d-1)/d}\gamma/(2d)$.
Indeed, if $(F_-,F_+)$ is an $\ell$-feasible pair, then there is a dangerous $K$ such that $\partial^+K=F_+$.
Since~$K$ is dangerous, we have $\min\{|K|,|K^c|\}\ge\delta(\log m)^{d/(d-1)}$, so by~\eqref{eqn1:lem:pair_of_open_sites_boundary}, we get
\[
\ell=|F_+|=|\partial^+K|\geq\min\{|\partial^{-}K|,|\partial^{+}K|\}\geq (\gamma/2d)\min\{|K|^{(d-1)/d},|K^c|^{(d-1)/d}\}
\ge \beta\log m.
\]
In order to state the next claim, we need to introduce additional notation. 
Assume the edges of~$\Phi_m$ follow some fixed total order. For $K\subseteq\Phi_m$, let $D(K,K^c)$ be the set of edges in $E(K,K^c)$ chosen sequentially according to the aforementioned ordering and, after selecting one edge $\{x,y\}$, removing from $E(K,K^c)$ any other edge incident to either~$x$ or~$y$.
Observe that at most $4d$ edges are deleted for each edge added to $D(K,K^c)$
(because at most~$2d$ edges are incident to any one vertex of $\Phi_m$). 
Hence, by construction, $D(K,K^c)$ is a disjoint collection of edges contained in the edge-boundary of~$K$ and $|D(K,K^c)|\geq\frac{1}{4d}|E(K,K^c)|$.
The third claim is that for any fixed $\ell$-feasible pair $(F_-,F_+)$,
\[
\mathbb{P}(\mathcal{B}_{(F_-,F_+)})\leq \mathbb{P}\big(\mbox{Bin}(|D(F_-,F_+)|,p^2)<(1/(2d))|D(F_-,F_+)|\big).
\]
By definition of $\mathcal{B}_{(F_-,F_+)}$, if $\mathcal{B}_{(F_-,F_+)}$ holds, then there is a dangerous $K$ such that $(\partial^-K,\partial^+K)=(F_-,F_+)$ and, moreover, the number of edges in 
$D(K,K^c)$ whose endvertices are open sites is less than
$\delta'\min\{|K|^{(d-1)/d},|K^c|^{(d-1)/d}\}$.
By disjointness, the latter number of edges is distributed like $\mathrm{Bin}(|D(F_-,F_+)|,p^2)$.
Moreover, by monotonicity of the probability, 
recalling that $\delta'=\gamma/(16d^3)$, and 
using~\eqref{eqn1:lem:pair_of_open_sites_boundary}
and~\eqref{eqn2:lem:pair_of_open_sites_boundary}, we obtain
\[
\delta'\min\{|K|^{(d-1)/d},|K^c|^{(d-1)/d}\}
\leq \frac{1}{8d^2}\min\{|\partial^-K|,|\partial^+K|\}
\leq \frac{1}{8d^2}|E(K,K^c)|
\leq \frac{1}{2d}|D(K,K^c)|.
\]
The claim follows observing that $D(K,K^c)=D(\partial^-K,\partial^+K)=D(F_-,F_+)$.

The fourth and last claim is that for any fixed $\ell$-feasible pair $(F_-,F_+)$
there is a $p<1$ for which
\[
\mathbb{P}\Big(\mbox{Bin}(|D(F_-,F_+)|,p^2)<(1/(2d))|D(F_-,F_+)|\Big)
\le \Big(\frac{1}{4^dc_{k'}(d)+1}\Big)^\ell.
\]
This follows directly from  Lemma~\ref{lem:chernoff_big_constant_exponent} setting $\alpha=1/(2d)$ and $\alpha'=\log(4^dc_{k'}(d)+1)$ and since 
$|D(F_-,F_+)|= |F_+|=\ell$  (since $D(F_-,F_+)$ is a set of disjoint edges so no more than one edge can be incident to a vertex in $F_+$ and, since $(F_-,F_+)$ is feasible, there is a nonempty $K$ such that $\partial^-K=F_-$ and $\partial^+K=F_+$, so by definition of internal/external-boundary at least one edge must be incident to every vertex of $F_+$).

\smallskip
We now have all the necessary ingredients to conclude: by a union bound,
\[
\mathbb{P}(\mathcal{B})\leq\sum_{(F_-,F_+)}\mathbb{P}(\mathcal{B}_{(F_-,F_+)}),
\]
where the summation is over all $\ell$-feasible pairs. 
The four claims proved above together imply that
\[
\mathbb{P}(\mathcal{B})
\le 
\sum_{\ell\ge\beta\log m}|\Phi_m|\Big(\frac{4^dc_{k'}(d)}{4^dc_{k'}(d)+1}\Big)^{\ell}
\le (4^dc_{k'}(d)+1)|\Phi_m|\Big(\frac{4^dc_{k'}(d)}{4^dc_{k'}(d)+1}\Big)^{\beta\log m},
\]
which for $\beta$ large enough goes to $0$ as $m\to\infty$.
\end{proof}

Our next result is similar in spirit to  Lemma~\ref{lem:pair_of_open_sites_boundary}. 
However, it does not impose connectivity conditions on the complement of the components to which it applies. 
\begin{lemma}\label{lem:pairs_of_open_sites_no_boundary_condition}
    For $0<\varepsilon<1$ and a percolation process in $\Phi_m$ with parameter $p>p^*$ for some $p^*<1$ large enough, there exist $\delta,\delta''>0$ such that, with probability going to 1 as $m\to \infty$, the following holds: for any *-connected set $K\subseteq \Phi_m$ with  $|F_m|>|K\cap\Phi_m^p|$ and $|K|>\delta(\log m)^{d/(d-1)}$, 
   there are at least $\delta''|K|^{(d-1)/d}$ disjoint edges in $E(K,K^c)$ each of whose endvertices are open sites.
\end{lemma}
\begin{proof} 
Consider all components of $K^c$ (note that they might contain open and closed sites), and denote by $\mathfrak{M}(K)$ the subset of these that contain at least one vertex of $F_m$.

For $X, Y \subseteq \Phi_m$, $X\cap Y=\emptyset$, denote by $D_{\mathrm{o}}(X,Y)$ a maximum cardinality subset of $E(X,Y)$ of edges whose endvertices are both open sites.

We claim that $|D_o(K,K^c)|\ge |\mathfrak{M}(K)|$ if $|\mathfrak{M}(K)|>1$. Indeed, consider a component~$M\in\mathfrak{M}(K)$ and 
an open site $y\in M\cap F_m$, and let $x\in F_m$ belong to another component $M'\in \mathfrak{M}(K)$. 
Since $x, y\in F_m$ there is a path in $F_m$ connecting 
$x$ and $y$. 
Such a path must contain an edge belonging to  $E(K,M)\subseteq E(K,K^c)$. Since the edge is also an edge of $F_m$ both of its endvertices are open sites.
Since $E(K,M)$ and~$E(K,M')$ are disjoint for distinct $M,M'\in\mathfrak{M}(K)$, the claim follows.

If $|\mathfrak{M}(K)|\ge \delta'|K|^{(d-1)/d}$ for $\delta'>0$ as in Lemma~\ref{lem:pair_of_open_sites_boundary}, the lemma follows taking $\delta''=\delta'$.
Henceforth, we may thus assume $|\mathfrak{M}(K)|< \delta' |K|^{(d-1)/d}$.
Denote by $\mathfrak{M}_{\smlSucceq}(K)$ the set of  components in $\mathfrak{M}(K)$ of size at least $\delta (\log m)^{d/(d-1)}$  for $\delta$ as in  Lemma~\ref{lem:pair_of_open_sites_boundary}. 
We will need to show that the union of the components in $\mathfrak{M}_{\smlSucceq}(K)$ cover a constant fraction of $\Phi_m$.
Specifically, we claim that 
\begin{equation}\label{eq:lem_4-5_first_eq}
\sum_{M\in 
\mathfrak{M}_{\smlSucceq}(K)}|M|\geq \varepsilon Cm/2.
\end{equation}
where $C>0$ is such that w.h.p.~$|F_m|\ge Cm$ (the existence of such $C$ follows from Theorem~\ref{thm:giant_Bernoulli_site_percolation}).
Indeed, if~\eqref{eq:lem_4-5_first_eq} does not hold, 
\[
\sum_{M \notin \mathfrak{M}_{\smlSucceq}(K)}|M| 
\ge |F_m \setminus K| 
-
\sum_{M\in\mathfrak{M}_{\smlSucceq}(K)}|M| 
> |F_m \setminus K| - \varepsilon Cm/2.
\]
Since $K\cap F_m\subseteq K\cap\Phi_m^p$, by the upper bound assumption on the cardinality of $K\cap\Phi_m^p$, we have that 
\[
|F_m\setminus K| 
= |F_m|-|K\cap F_m|\ge |F_m|-|K\cap\Phi_m^p|
\geq \varepsilon |F_m| \ge \varepsilon Cm,
\]
where the last inequality holds w.h.p. So, plugging this estimate into the above bound, w.h.p.,
$$
\sum_{M \notin \mathfrak{M}_{\smlSucceq}(K)}|M| > \varepsilon Cm/2.
$$
However, 
\[
\sum_{M \notin \mathfrak{M}_{\smlSucceq}(K)} |M| \le \delta|\mathfrak{M}(K)|(\log m)^{d/(d-1)} 
\le \delta\delta'|K|^{(d-1)/d}(\log m)^{d/(d-1)},
\]
which is $o(m)$ as $m\to\infty$ (because $|K|\leq m$)
contradicting the above. This concludes the proof of~\eqref{eq:lem_4-5_first_eq}.

\smallskip
Our goal is to show that $|D_o(K,K^c)|\ge \delta'' |K|^{(d-1)/d}$ for some $\delta''>0$. To do so, we leverage the fact that for any~$M \in \mathfrak{M}(K)$ we have that $M$ is connected and $M^c$ is *-connected (the latter follows since $K$ is *-connected, $M^c=K\cup\bigcup_{M'\neq M}M'$, each $M'\in\mathfrak{M}(K)$ is connected and $\partial^{+}M'\subseteq K$), so $M$ meets two of the hypothesis required in order to apply Lemma~\ref{lem:pair_of_open_sites_boundary}. 
We divide the rest of the analysis into two cases. 

\medskip\noindent
\textbf{Case 1 (there is an $M_*\in \mathfrak{M}_{\smlSucceq}(K)$ such that $|M^c_*|<|M_*|$):} 
Then, $|M^c_*|\leq |\Phi_m|/2$. Hence,
we can apply Lemma~\ref{lem:pair_of_open_sites_boundary}, since $|M^c_*|\ge |K|\ge \delta (\log m)^{d/(d-1)}$ (because $K\cap\Phi_m^p\subseteq K\subseteq M^c_*$ and $|K\cap\Phi_m^p|\ge\delta(\log m)^{d/(d-1)}$), and conclude that
\[
|D_o(K,K^c)|\ge |D_o(K,M_*)| =  |D_{\mathrm{o}}(M_*,M_*^c)| \geq \delta' |M^c_*|^{(d-1)/d} \geq \delta'|K|^{(d-1)/d}.
\]

\medskip\noindent
\textbf{Case 2
(for every $M \in \mathfrak{M}_{\smlSucceq}(K)$ we have that $|M|\leq |M^c|$):} Then, $|M|\le |\Phi_m|/2$ for every $M\in\mathfrak{M}_{\smlSucceq}(K)$. 
Since, by definition of $\mathfrak{M}_{\smlSucceq}(K)$, all elements of 
$\mathfrak{M}_{\smlSucceq}(K)$ have size large enough so we can apply Lemma~\ref{lem:pair_of_open_sites_boundary} to each of them, we get that
\[
 \sum_{M\in \mathfrak{M}_{\smlSucceq}(K)}|D_{\mathrm{o}}(M,M^c)| \geq \sum_{M\in \mathfrak{M}_{\smlSucceq}(K)}\delta'|M|^{(d-1)/d} 
 \geq\delta'\Big(\sum_{M\in\mathfrak{M}_{\smlSucceq}(K)}|M|\Big)^{(d-1)/d}
\geq \delta'(\eps Cm/2)^{(d-1)/d},
\]
where the last identity is by~\eqref{eq:lem_4-5_first_eq}.

Since any vertex appears at most $2d$ times as endvertex of an edge in $\bigcup_{M\in \mathfrak{M}(K)}D_{\mathrm{o}}(M,M^c)$, we have 
\[
|D_{\mathrm{o}}(K,K^c)| \geq \frac{1}{4}\sum_{M\in \mathfrak{M}(K)}|D_{\mathrm{o}}(M,M^c)| \geq \frac{\delta'}{4}(\eps C m/2)^{(d-1)/d}.
\]
Setting $\delta'' = \min\{ \frac{\delta'}{4}(\eps C/2)^{(d-1)/d} ,\delta'\}$ and observing that $|K|\le m$ (because $K\subseteq\Phi_m$), we arrive at the desired conclusion.


\end{proof}


\subsection{Preliminaries on random geometric graphs}\label{sec:RGG}
%

In this final subsection as well as throughout the entire following section we will always consider RGGs with $d \ge 2$,  assuming tacitly always that $r^d=O(n)$. While the latter assumption is justified by the fact that for $r^d= cn$ with $c$ large enough we already have a complete graph, the first assumption simplifies statements in the sequel, and we deal separately with $d=1$ in the final section.

We first collect the following two basic results on RGGs:


\begin{theorem}{(\cite[Theorem 10.9]{Pen03} and \cite[Lemma 2.10]{martinez2025jump})}\label{thm:concentration_vertices_edges_rgg}
Let $\eps > 0$ be arbitrarily small, and let $r\ge (1+\eps)r_g$. Then w.h.p. $|\mathcal{L}_n|=\Theta(n)$ and $|E(\mathcal{L}_n)|=\Theta(n r^d)$.
\end{theorem}


We also know that w.h.p.\ all other  components of $\mathcal{G}^r_n$ distinct from $\mathcal{L}_n$ are small, as the following well-known result shows:
\begin{theorem}{(\cite[Theorem 10.18]{Pen03})}\label{thm:size_second_component_rgg} 
    Let $\eps > 0$ be arbitrarily small, and let $r\ge (1+\eps)r_g$. Then, w.h.p.\ every  component of  $\mathcal{G}_n$ of size $\omega((\log n)^{d/(d-1)})$ belongs to $\mathcal{L}_n$.
\end{theorem}

We next describe the concept of a \emph{tiling} that will be important throughout.
For $\rho>0$ (which throughout will always be of the order of $r$) consider the tiling of $\mathbb{R}^d$ by (hyper)cubes of side length $\rho$, that is, consider $\{\tau_{\bfi} \colon \bfi\in\mathbb{Z}^d\}$ where 
\[
\tau_{\bfi} = \rho\cdot\bfi + \Big[\,-\frac{\rho}{2}, \frac{\rho}{2}\,\Big)^d.
\]
We refer to $\bfi$ as the \emph{center} of tile $\tau_{\bfi}$.
Also, we let~$\Lambda^{\rho}_n$ be the collection of all such tiles that are contained in~$\Lambda_n$.
Thus, all tiles in $\Lambda_n^\rho$ have the same volume which we henceforth
denote by $\mathrm{vol}_\rho$.  
Observe that~$\mathrm{vol}_\rho$ is proportional to~$\rho^d$ where the proportionality constant depends only on the dimension $d$.

In order to avoid confusions and clearly distinguish between situations where we count vertices or points in $\mathbb{R}^d$ from those where we count tiles in a given set, in what follows, for a collection of tiles~$S\subseteq\Lambda_n^{\rho}$ we denote its cardinality by $|S|_t$.

Now, let $m=|\Lambda^\rho_n|_t$.
Observe that $(n^{1/d}-2\rho)^d\le m\cdot\mathrm{vol}_\rho\le n$. 
Thus, for $\rho=\frac12n^{1/d}-\Omega(1)$, we have~$m=\Theta(n/\rho^d)$ and also $m=\Theta(n/r^d)$ since $\rho$ is always of order $r$.
Throughout, we identify~$\Lambda^{\rho}_n$ with $\Phi_m$ in the natural way by associating each tile in $\Lambda_n^{\rho}$ to the vertex $\Phi_m$ corresponding to its center. Via this identification the concepts and notation concerning $\Phi_m$ (e.g., edges, adjacency, order,  components, etc.) can and will, be transferred and used when dealing with $\Lambda_n^\rho$.

One particular collection of tiles we will often consider is  $\Lambda^{r_d}_n$, 
where $r_d=r_d(n)$ is the maximum real value such that $n^{1/d}/r_d$ is an odd integer and such that $r_d\leq r/(2\sqrt{d})$. 
Concretely, 
\[
r_d=n^{1/d}/\Big(2\Big\lceil\frac{\sqrt{d}n^{1/d}}{r}-\frac12\Big\rceil+1\Big).
\]
In particular, $r_d/r$ tends to $1/(2\sqrt{d})$ from below as $n\to\infty$.
There are two key reasons for this choice of side length of tiles. First, $\Lambda_n^{r_d}$ partitions $\Lambda_n$ into tiles of equal volume.
The second reason is that the choice of $r_d$ guarantees that the largest distance between any two points in a tile is at most $r/2$, which also implies that two vertices of $\mathcal{G}_n$ inside adjacent tiles (whose topological closure intersect in a~($d-1$)-dimensional face) of $\Lambda_n^{r_d}$ are neighbors in $\mathcal{G}_n$ (equivalently, the distance between any two points of adjacent tiles is at most $r$). The latter will turn out to be important further on. 

In what follows, we will often deal simultaneously with  collections of tiles, say $T\subseteq\Lambda_n^\rho$, and 
subsets of vertices, say $A\subseteq\mathcal{G}_n$.
We will abuse notation and view $T$ both as a collection of tiles and the subset of $\Lambda_n$ comprised by the union of all tiles in $T$ and denote the latter subset also by~$T$.
Moreover, we will also abuse language and talk about the vertices of $A$ in $T$ in order to refer to the vertices of $A$ identified with points that belong to some tile in $T$ and, making use of our abuse of notation, denote the collection of such vertices by $T\cap A$.

For a set $A \subseteq \mathcal{G}_n$ we let $L_A\subseteq\Lambda^{\rho}_n$ be the collection of tiles of $\Lambda_n^\rho$ that contain at least one vertex of~$A$, that is
\[
L_A = \big\{ \tau\in\Lambda^{\rho}_n \colon A\cap\tau\neq\emptyset\big\}.
\]
If $A$ is connected, then $L_A$ is a $k$-lattice animal for $k=\lceil \frac{r}{\rho\sqrt{d}}\rceil d$. 
So, when $\rho$ is proportional to $r$, as is always the case in this manuscript, the value of $k$ depends solely on the ratio $r/\rho$ and the dimension~$d$.
From now on, throughout all of what remains of this section, whenever we refer to a $k$-lattice animal it is to be understood that we have fixed $k$ equal to $\lceil \frac{r}{\rho\sqrt{d}}\rceil d$ for whatever~$\rho$ we are working with (typically,~$\rho=r_d$).

\begin{lemma}\label{lem:bounds_vertices_big_animals}
    There exists $\gamma>0$ and $r'>r_g$ being a sufficiently large  constant,
    so that for $r > r'$ the following holds w.h.p.: 
    every connected set $A\subseteq \mathcal{G}_n$ with $|A|>2\lceil\gamma\log n\rceil\mathrm{vol}_{r_d}$  satisfies, for $L_A\subseteq\Lambda_n^{r_d}$, that $|L_A|_t> |A|/(2\mathrm{vol}_{r_d})$.
\end{lemma}
\begin{proof}
   Let $L\subseteq\Lambda_n^{r_d}$ be a $k$-lattice animal of size $\ell$. Since $|L\cap\mathcal{G}_n|$ is a Poisson random variable with mean $\ell\mathrm{vol}_{r_d}$, by Chernoff bounds for Poisson random variables, the probability that  $|L\cap\mathcal{G}_n|> 2\ell\mathrm{vol}_{r_d}$ is at most $\exp(-\beta\ell\mathrm{vol}_{r_d})$ where $\beta$ is a positive constant. 

   We claim that there exists $\gamma>0$ such that, w.h.p., every $k$-lattice animal in $\Lambda_n^{r_d}$ of size $\ell \ge \gamma\log n$ contains at most $2\ell\mathrm{vol}_{r_d}$ vertices of $\mathcal{G}_n$.
   Indeed, recall that the number of $k$-lattice animals in $\Lambda_n^{r_d}$ of size $\ell$ is bounded above by $|\Lambda_n^{r_d}|_t(c_k(d))^\ell$.
   Then, by a union bound over all possible $k$-lattice animals of size at least $\gamma \log n$, the probability that some $k$-lattice animal in $\Lambda_n^{r_d}$, say $L$, is such that  $|L|_t\ge\gamma\log n$ and $|L\cap\mathcal{G}_n|>2|L|_t\mathrm{vol}_{r_d}$ is at most
    \[
    |\Lambda_n^{r_d}|_t\sum_{\ell\ge\gamma \log n} c_k(d)^\ell \exp(-\beta\ell\mathrm{vol}_{r_d})= |\Lambda_n^{r_d}|_t\sum_{\ell\ge\gamma \log n} \left[\exp\big(\log(c_k(d))-\beta\mathrm{vol}_{r_d}\big)\right]^{\ell}.
    \]
    Since $\mathrm{vol}_{r_d}=\Theta(r^d)$, there exists $r'>r_g$
    such that for all $r\geq r'$ 
    it holds that $\log(c_k (d))< \beta\mathrm{vol}_{r_d}$. Thus, because $|\Lambda_n^{r_d}|_t=O(n)$ for $r>r_g$,
    for a sufficiently large $\gamma>0$ the right-hand side of the last display tends to $0$ as $n$ goes to infinity.
    This concludes the claim's proof.

    Next, we prove the lemma's statement.
    By hypothesis $A$ is connected in~$\mathcal{G}_n$, so~$L_A$ is a $k$-lattice animal in $\Lambda^{r_d}_n$. 
    Thus, if $r>r'$ and $|L_A|_t\ge\gamma\log n$, by the discussion above, w.h.p., $|L_A\cap\mathcal{G}_n|\leq 2|L_A|_t\mathrm{vol}_{r_d}$. 
    But trivially $A\subseteq L_A\cap\mathcal{G}_n$, so  we get $2|L_A|_t\mathrm{vol}_{r_d} \ge |A|$, or equivalently, 
    $|L_A|_t\ge |A|/(2\mathrm{vol}_{r_d})$ w.h.p.
    On the other hand, if $|L_A|_t\leq \gamma \log n$, then there exists a $k$-lattice animal $L'_A$ in $\Lambda^{r_d}_n$
    of size~$\lceil\gamma \log n\rceil$ that contains~$L_A$ and, again by our claim above, the cardinality of $L'_A\cap\mathcal{G}_n\supseteq L_A\cap\mathcal{G}_n\supseteq A$ is, w.h.p., at most~$2\lceil\gamma \log n\rceil\mathrm{vol}_{r_d}$ which is in contradiction with the statement's hypothesis concerning the lower bound on the size of $A$.
    \end{proof}

    \begin{remark}\label{rem:bounds_vertices_big_animals}
        Note that in the proof of Lemma~\ref{lem:bounds_vertices_big_animals}, the connectedness of $A$ was used solely to ensure that $L_A$ is a $k$-lattice animal. Therefore, the conclusion of the lemma remains valid for any set $A\subseteq \mathcal{G}_n$ such that $L_A$ is a $k$-lattice animal, even if $A$ is not connected. In particular, w.h.p., given any $k$-lattice animal $L$ such that $|L\cap\mathcal{G}_n|> 2\lceil \gamma \log n\rceil \mbox{vol}_{r_d}$ it holds that $|L|_t=\Omega(|L\cap \mathcal{G}_n|/r^d)$.
    \end{remark}
    \begin{remark}\label{rem:bound_vertices_big_animals_2}
        For a fixed value of $k$ (we will choose $k=d$), the statement of Remark~\ref{rem:bounds_vertices_big_animals} can be extended to tessellations of $\Lambda_n$ with side length different from $r$. We claim that we can obtain the following statement for the tessellation $\Lambda_n^{\rho}$, if $\rho$ is chosen large enough: w.h.p., given any $d$-lattice animal $L\subseteq \Lambda_n^{\rho}$ satisfying $|L\cap \mathcal{G}_n|>2\lceil \gamma\log n\rceil \mathrm{vol}_{\rho}$, it holds that $|L|_t=\Omega(|L\cap \mathcal{G}_n|/\mathrm{vol}_{\rho})$: indeed, the proof of this statement is a simple adaptation of the proof of Lemma~\ref{lem:bounds_vertices_big_animals}: choosing $\rho$ large enough ensures that $\log(c_d(d))<\beta \mathrm{vol}_{\rho}$ and the rest of the proof is analogous (just setting $k=d$ and $\rho$ large enough).
    %
    \end{remark}
    
    In the next fundamental theorem  we establish a relation between $\pi(A)$ and $|A|$.  
 
    \begin{theorem}\label{thm:relation_StationaryMeasure_number_of_vertices_rgg}
   Let $d \ge 2$, let $\eps > 0$ arbitrarily small with $r \ge (1+\eps)r_g$. There exists $\eta > 0$ such that w.h.p.\ the following holds: for every connected set $A$ of vertices of $\mathcal{L}_n$ such that $\pi(A)>(\log n)^\zeta/n$ for some $\zeta > 6$, we have
    \[
     \sum_{x\in A} \deg(x) < \eta |A|r^d.
    \]
    
\end{theorem}
\begin{proof}
First observe that for $r^d=\omega(\log n)$, w.h.p., all degrees are $\Theta(r^d)$, and so the statement follows easily in this case. Assume thus from now on $r^d=O(\log n)$.
Observe that the minimum degree of vertices in $A$ is $1$ (because $\pi(A)>0$ implies that $A\neq\emptyset$ and $A\subseteq\mathcal{L}_n$).
Also, w.h.p.~the maximum degree in $\mathcal{G}_n$ is $O(\log n)$ and the total degree is $\Omega(n)$ by our assumption on $r$. 
Thus, since $\pi(A) > (\log n)^\zeta/n$ with $\zeta>6$, it follows that
$|A| > \log^{5} n$. 

Consider the tessellation $\Lambda^{\rho}_n$ of $\Lambda_n$ where $\rho$ is the smallest real value for which $n^{1/d}/\rho$ is an integer and $\rho\ge 2r$. In other words, let $\rho=n^{1/d}/\lfloor n^{1/d}/(2r)\rfloor$. 
Note that $\rho\le 2r/(1-2r/n^{1/d})=(1+o(1))2r$.
Since w.h.p.~every tile contains $O(\log n)$ vertices by our assumption of $r^d = O(\log n)$, the set $A$ can be partitioned into $k$-lattice animals of size between $(\log n)^4$ and $2(\log n)^4$. 
In fact, we will consider~$2^d$ tessellations of the same side length, where tessellation $\bfi=(i_1,...,i_d)\in\{0,1\}^d$ is the same as 
tessellation $\mathbf{0}=(0,...,0)$, with all tiles shifted to the right by $\rho$ in those dimensions where~$i_j=1$. In this way we can ensure that every edge of the graph is contained in at least one tile of one tessellation. 
By a union bound over all tessellations, 
it therefore suffices to show that, w.h.p., for every such lattice animal $L$ of one fixed tessellation corresponding to tiles that contain at least one vertex of~$A$, with~$\tau_x$ denoting the tile containing $x$ and $B_x(r)$ the Euclidean ball of radius $r$ centered at $x$, we have~$\sum_{x \in A \cap L} |\tau_x\cap B_x(r)| < (\eta/2^d)|A \cap L| r^d$; this will clearly imply the statement.

Let $\ell$ be a positive integer such that  $\log^4 n\le \ell \le 2\log^4 n$. Call a tile $j$-bad, if the number of vertices in it is in $[2^j r^d, 2^{j+1}r^d)$, for some $j_0=j_0(d) \le j < j_1$ (where $j_0$ is a large enough constant, depending on $d$, and where $j_1$ is the smallest integer such that 
$\ell/(2^{2j}j^2r^{d}) \le 1$. 
Observe that, since $2^{2j_1}j_1^2r^d=\Theta(\ell)=\Theta(\log^4 n)$, $r^d=O(\log n)$ and $r>r_g$, we have $j_1\ge(\frac32-o(1))\log\log n$ and $j_1\le (2-o(1))\log\log n$. 
Also, note that there are, in expectation, $O(r^d)=O(\log n)$ vertices in each tile, and by Chernoff bounds, w.h.p.~all tiles in all tessellations contain $O(\log n)$ vertices. Hence, for $n$ large we always may assume $j \le j_1$, as w.h.p.~in all tessellations all tiles contain at most $2^{j_1}r^d$ vertices, since $2^{j_1}r^d=\omega(\log n)$.
The probability that there exists a tessellation, and a  $k$-lattice animal of size $\ell$ of that tessellation with  more than $\lceil \ell/(2^{2j}j^2 r^d) \rceil$ tiles that are $j$-bad for some $j_0 \le j < j_1$, 
is at most 
\begin{align*}
& \sum_{j_0 \le j < j_1}  O\big(n (10^{dj})^{\lceil \ell / 10^j\rceil}\big)
\binom{\ell (2k+1)^d 10^{dj-1}}{\lceil \frac{\ell}{2^{2j}j^2r^{d}} \rceil} 2^jr^d \Big(e^{-\rho^d} \frac{(\rho^d)^{2^j r^d}}{(2^j r^d)!}\Big)^{\lceil \frac{\ell}{2^{2j}j^2r^{d}} \rceil} 
\\
& \qquad =O\big(n \log \log n\big) \max_{j_0 \le j < j_1 } O\big((10^{dj})^{\lceil \ell / 10^j\rceil}\big)\binom{\ell (2k+1)^d 10^{dj-1}}{\lceil \frac{\ell}{2^{2j}j^2r^{d}}\rceil} 2^jr^d \Big(e^{-\rho^d} \frac{(\rho^d)^{2^j r^d}}{(2^j r^d)!}\Big)^{\lceil\frac{\ell}{2^{2j}j^2r^{d}}\rceil}
\end{align*}
where the first factor of the general term of the summation in the first line is due to the fact that there are $O(n)$ choices of starting tiles for a $k$-lattice animal in all tessellations and that we can cover a~$k$-lattice animal of size~$\ell$ by $\lceil \ell/10^j \rceil$ containers, which are axis-parallel (hyper)cubes of side length~$(2k+1)10^{j}(2r)$, centered around the current tile and completely containing the next $10^j$ tiles of the~$k$-lattice animal, the second factor accounts for the $(2k+1)^d10^{dj}$ tiles that can be chosen as the next center of the next container (already counted tiles should be ignored, of course); then we bound the number of ways in which~$\lceil \ell/(2^{2j}j^2r^{d}) \rceil$ tiles that are $j$-bad can be chosen among the tiles in $\lceil\ell/10^j\rceil$ containers each comprised of  $(2k+1)^d10^{dj}$ tiles, the third term is because there are at most~$2^jr^d$ values among~$[2^jr^d,2^{j+1}r^d)$ for being $j$-bad, 
and the last term is the maximal probability for a Poisson random variable with parameter $\rho^d$ among these $2^j r^d$ values.
in the second line we used that the sum in the first line has $O(\log \log n)$ terms.
Taking logarithms, we see that the contribution of the  term in the second displayed line above is at most 
\[
\begin{split}
(1+o(1))\log n+\max_{j_0 \le j < j_1} \Big( (1+o(1))\frac{\ell}{10^j}dj\log(10)+(1+o(1))\frac{\ell}{2^{2j}j^2r^{d}}\log(10^{dj-1}2^{2j}j^2r^{d}(2k+1)^d) \\
-(1-o(1))\frac{\ell}{ 2^{2j}j^2r^{d}}2^{j}r^d\log (2^j/(e2^d))\Big). 
\end{split}
\]
Observe that for the expression inside the maximum, for $j_0$ large enough, the first and the second term are in absolute value at most one third of the third term, and plugging in these bounds for the first and second term, the expression inside the maximum is then increasing in $j$. Hence, the maximum is attained for $j=j_1$, and by definition of $j_1$, is at most $-(1-o(1)) 2^{j_1}r^d\log(2^{j_1}/(e2^d))/15$ (we used that~$\ell/(2^{2j_1}j_1^2r^d) > 1/5$). This expression is smaller than~$-c\log^2 n$ for sufficiently small $c > 0$ 
(because~$2^{j_1}j_1r^d =r^{d/2}\sqrt{2^{2j_1}j_1^2r^d}=\Omega(\log^2 n)$). 
Hence, the whole expression is negative, and so the original expression is $o(1)$. 

Since we aim for a statement that holds w.h.p., we may condition under this. For $ j_0 \le j < j_1$, let $A_j$ be the set of elements of $A$ inside $j$-bad tiles of a $k$-lattice animal $L$ of size~$\ell$ (recall that $\log^4 n \le \ell \le 2\log^4 n$). We get
\begin{align*}
\sum_{x \in A \cap L} |\tau_x\cap B_x(r)| 
\le 2^d\Big(2^{j_0}|A \cap L| r^d + \sum_{j_0 \le j < j_1} 2^{j+1}r^d |A_j|\Big) 
 =
O( |A \cap L|r^d)+2^{d+1}\sum_{j_0 \le j < j_1} 2^j r^d |A_j|.
\end{align*}
Note that as there are at most ${\lceil 
\ell/(2^{2j}j^2r^d)\rceil}$ tiles that are $j$-bad, for $j_0 \le j < j_1$, the contribution of each summand $2^{j} r^d|A_j|$ is at most 
\[
 2^{j} r^d \Big\lceil \frac{\ell}{2^{2j}j^2r^d} \Big\rceil 2^{j+1} r^d \le 2\Big(\frac{\ell}{j^2} + 2^{2j}r^d\Big)r^d 
 \le \frac{12}{j^2}\ell r^d,
\]
where for the last inequality we used that $j<j_1$ and $5\ell/(2^{2j_1}j_1^2r^d)>1$. 
Hence, 
$$
2^{d+1}\sum_{j_0 \le j < j_1}2^j r^d |A_j|\le 2^{d+1}\sum_{j_0 \le j < j_1}\frac{12}{j^2} \ell r^d.
$$
Observe now that $|A \cap L|\ge \ell$, as by definition at least one element of each tile of $L$ has to belong to~$A$. Thus,
$$
2^{d+1}\sum_{j_0 \le j < j_1}\frac{12}{j^2}\ell r^d=O(\ell r^d)=O(|A \cap L|r^d),
$$
and the theorem follows.
\end{proof}
Finally, when degrees are not necessarily concentrated and under some conditions, we show that a positive fraction of vertices of $\mathcal{L}_n$ have degrees close to the average degree.
\begin{proposition}\label{prop:vertices_with_degree_r_d}
    Let $\eps > 0$, and let $r \ge (1+\varepsilon)r_g$, and assume also $r^d=O(\log n)$. There exist positive constants $\underline{c}, \overline{c} > 0$ and $0 < \xi<1$ such that, if $\mathcal{X}_n$ denotes the set of vertices of $\mathcal{L}_n$ with degree between $\underline{c}\,r^d$ and $\overline{c}\,r^d$, then w.h.p.
    \[
    |\mathcal{X}_n| \ge \xi \,|\mathcal{L}_n| \, .
    \]
\end{proposition}
\begin{proof}
We condition on the existence of the giant, all the calculations below are done considering this conditioning (this does not affect the result as the giant exists w.h.p). We denote by $\mathbb{P}_{\mathbf{0}}$ the Palm probability, that is, the probability obtained by adding one virtual vertex at the origin $\mathbf{0}\in\mathbb{R}^d$. By the Mecke formula we obtain
\begin{equation}\label{eq:mecke_formula_first_moment_X_n}
\mathbb{E}(|\mathcal{X}_n|) =  n\mathbb{P}_{\mathbf{0}}(\mathbf{0}\in \mathcal{L}_n,\deg(\mathbf{0})\in [\underline{c}r^d,\overline{c}r^d]) +  o(n),
\end{equation}
where the term $o(n)$ results from boundary effects. First, recall that $\mathbb{P}_{\mathbf{0}} (\mathbf{0}\in \mathcal{L}_n)\ge \gamma$ for some positive $\gamma$ that does not depend on $n,r$. We claim that $\mathbb{P}_{\mathbf{0}}(\deg (\mathbf{0})\in [\underline{c}r^d,\overline{c}r^d] \mid \mathbf{0}\in \mathcal{L}_n)$ is uniformly bounded from below by a constant independent of $r,n$. We know that
\begin{equation}
    \mathbb{P}_{\mathbf{0}}(\deg(\mathbf{0})\in[\underline{c}r^d,\overline{c}r^d]) \ge 1-\delta,
\end{equation}
where $\delta > 0$ can be made arbitrarily small by setting $\underline{c},\overline{c}$ small and large enough, respectively. By the law of total probability we clearly have
\begin{align}
    \mathbb{P}_{\mathbf{0}}(\deg (\mathbf{0})\in [\underline{c}r^d,\overline{c}r^d] \mid\mathbf{0}\in \mathcal{L}_n)\mathbb{P}_{\mathbf{0}}(\mathbf{0} \in \mathcal{L}_n) + \mathbb{P}_{\mathbf{0}}(\deg (\mathbf{0})\in [\underline{c}r^d,\overline{c}r^d] \mid \mathbf{0}\notin \mathcal{L}_n)\mathbb{P}_{\mathbf{0}}(\mathbf{0} \notin \mathcal{L}_n)
    \\ 
    = \mathbb{P}_{\mathbf{0}}(\deg (\mathbf{0})\in [\underline{c}r^d,\overline{c}r^d]).
\end{align}
Thus, if $1-\gamma$ is smaller than $1-\delta$, the left hand side term of the sum of the first row is positive, and hence for $\delta > 0$, the claim is proven.

Substituting this into~\eqref{eq:mecke_formula_first_moment_X_n}, we obtain $\mathbb{E}(|\mathcal{X}_n|)\ge C'n$ for some positive $C'$ that does not depend on $n,r$. Now, we calculate the second moment of $|\mathcal{X}_n|$:
\begin{align}
\mathbb{E}(|\mathcal{X}_n|^2) 
& = \mathbb{E}\Big(\sum_{x,y \in \mathcal{L}_n} \mathbf{1}_{\{\deg(x) \in [\underline{c}r^d,\overline{c}r^d]\}} \mathbf{1}_{\{\deg(y)\in [\underline{c}r^d,\overline{c}r^d]\}} \Big) 
\\
& = \mathbb{E}\Big(\!\sum_{\substack{x,y\in \mathcal{L}_n\\ |x-y|>2r}}\!\!  \mathbf{1}_{\{\deg(x) \in [\underline{c}r^d,\overline{c}r^d]\}} \mathbf{1}_{\{\deg(y)\in [\underline{c}r^d,\overline{c}r^d]\}} + \!\!\sum_{\substack{x,y\in \mathcal{L}_n\\ |x-y|\le 2r}}\!\! \mathbf{1}_{\{\deg(x) \in [\underline{c}r^d,\overline{c}r^d]\}} \mathbf{1}_{\{\deg(y)\in [\underline{c}r^d,\overline{c}r^d]\}} \Big)
\\ 
& \le \left(\mathbb{E} (|\mathcal{X}_n|)\right)^2 + \mathbb{E}(|\mathcal{X}_n|)O(\log n),
\end{align}
where the first term of the last row follows from  the fact that the degrees of vertices at distance at least $2r$ are independent variables, and the second term comes from the fact that as $r^d=O(\log n)$, the maximum degree is $O(\log n)$ w.h.p. By Chebyshev's inequality,
$$
\mathbb{P} \left(\big||\mathcal{X}_n|-\mathbb{E}(|\mathcal{X}_n|)\big|\ge \lambda  \sqrt{\mathbb{E}(|\mathcal{X}_n| )O(\log n)}\right) \le \frac{1}{\lambda^2}.
$$
Choosing say $\lambda = \log n$ we see that w.h.p. $|\mathcal{X}_n|\ge \xi n$ for some $\xi>0$.
\end{proof}

\section{Upper Bound on the mixing time}\label{sec:Upper}
The main ingredient of the proof of the upper bound of our main result is an isoperimetric inequality which, for large enough connected sets of vertices $A\subseteq\mathcal{G}_n$  (see next result for a precise statement), gives a lower bound on the number of edges between $A$ and $\mathcal{L}_n\setminus A$.

Throughout this whole section, we use asymptotic notation whenever no ambiguity can arise and use explicit constants in the more delicate estimates, so that dependencies are clear.

Before proving the preceding result, we will show next that the upper bound on the mixing time in our main result follows from Theorem~\ref{thm:isoperimetric_inequality_rgg}.
\begin{lemma}\label{lem:lower_bound_varphi}
Let $\varepsilon > 0$, $r\ge (1+\varepsilon)r_g$, and $r^d=O(\log n)$.
Then, for 
$\pi_0\leq  t\leq 1/2$ where $f(n)=(\log n)^{5d}r^d/n$,
w.h.p., 
    \[
    \varphi(t) = \Omega\Big(\min\Big\{\frac{1}{nf(n)r^d}, \frac{r}{(nt)^{1/d}}\Big\}\Big).
    \] 
\end{lemma}
\begin{proof}
    For $A\subseteq\mathcal{L}_n$ such that $0<\pi(A)<1$, 
      let $\varphi_A=Q(A,A^c)/(\pi(A)\pi(A^c))$.
    Since $\pi(A^c)\leq 1$ for all~$A\subseteq\mathcal{L}_n$, 
    we have 
    \begin{equation}\label{eqn:lowBnd_varphiA}
    \varphi_A \geq \frac{Q(A,A^c)}{\pi(A)}.
    \end{equation}
    First, consider the case where $A\subseteq\mathcal{L}_n$ is such that $0<\pi(A)\leq f(n)$.
    Since $A\neq\emptyset$ (because $\pi(A)> 0$), $A^c\neq\emptyset$ (because $\pi(A^c)\ge 1-f(n)\ge 1/2$), and $\mathcal{L}_n$ is connected, it follows that $E(A,A^c)\neq \emptyset$.
    Thus, $Q(A,A^c)\geq 1/(2
    |E(\mathcal{L}_n)|)$ which, by~\eqref{eqn:lowBnd_varphiA}, together with our case
    assumption yields 
    \begin{equation}\label{eqn:varphiA1}
    \varphi_A \geq \frac{Q(A,A^c)}{\pi(A)} 
    \geq \frac{1}{2f(n)|E(\mathcal{L}_n)|}.
    \end{equation}
    
    Next, consider the case where $f(n)<\pi(A)\leq t\le t_0$, for some constant $t_0>0$ to be established later. 
    Then, by Theorem~\ref{thm:relation_StationaryMeasure_number_of_vertices_rgg}, for some $\eta > 0$, w.h.p., 
    \begin{equation}\label{eqn:lowerBnd_piA}
    f(n)< \pi(A) 
    < \frac{\eta |A|r^d}{2|E(\mathcal{L}_n)|}. 
    \end{equation}
    Hence, we get $|A|=\Omega(f(n)|E(\mathcal{L}_n)|/r^d)$. Thus, recalling that $d\geq 2$ and $r^d=O(\log n)$,
    since w.h.p.~$|E(\mathcal{L}_n)|=\Theta(nr^d)$
    (by Theorem~\ref{thm:concentration_vertices_edges_rgg})
    and 
    \[
    (\log n)^{5d}
    =\omega\big((r^{2d}(\log n)^{\frac{d}{d-1}}\big)^d\big),
    \]
    by definition of $f(n)$, we get that 
    $|A|=\omega((r^{2d+1}(\log n)^{d/(d-1)})^d)$.
    So, the only hypothesis remaining to apply Theorem~\ref{thm:isoperimetric_inequality_rgg} is $|\mathcal{L}_n\setminus A|=\Omega(|\mathcal{L}_n|)$. This can be achieved by making $t_0$ small enough: indeed, denote by $\mathcal{X}_n$ the set of vertices of $\mathcal{L}_n$ with degree between $\underline{c}r^d$ and $\overline{c}r^d$, with $\underline{c},\overline{c}$ as in Proposition~\ref{prop:vertices_with_degree_r_d}. By the hypothesis over $\pi(A)$,
    $$
    \sum_{x\in A} \deg(x) \le 2t_0 |E(\mathcal{L}_n)| \le 2C't_0|\mathcal{L}_n|r^d\le \eps'|\mathcal{L}_n|r^d
    $$
    where we used that $|E(\mathcal{L}_n)|\le C'|\mathcal{L}_n|r^d$ for some constant $C'$ (see Theorem~\ref{thm:concentration_vertices_edges_rgg}). Furthermore, $\eps'$ can be made small enough by making $t_0$ small. Next, we derive a lower bound on $|\mathcal{X}_n\setminus A|$. 
    First, note that 
    $$
    \sum_{x\in\mathcal{X}_n\setminus A}\deg(x) = \sum_{x\in \mathcal{X}_n} \deg(x) - \sum_{x\in A\setminus \mathcal{X}_n} \deg(x) \ge \underline{c}|\mathcal{X}_n|r^d - \eps'|\mathcal{L}_n|r^d = \Omega(|\mathcal{L}_n|r^d),
    $$
    where we used Proposition~\ref{prop:vertices_with_degree_r_d} to bound $|\mathcal{X}_n|$ and made $\eps'$ small enough such that the last equality holds. As the maximum degree of a vertex in $\mathcal{X}_n$ is $\overline{c}r^d$ by definition, we deduce that $|\mathcal{L}_n\setminus A|\ge |\mathcal{X}_n\setminus A|=\Omega(|\mathcal{L}_n|)$. 
        So, 
    we can apply Theorem~\ref{thm:isoperimetric_inequality_rgg} to such $A$ with $\pi(A)\le t_0$  and obtain that 
    \[
    \varphi_A  \geq \frac{|E(A,A^c)|}{2\pi(A)|E(\mathcal{L}_n)|} =\Omega\left( \frac{|A|^{(d-1)/d}\cdot r^{d+1}}{\pi(A)|E(\mathcal{L}_n)|}\right)
    =\Omega\left(\frac{r}{(\pi(A)|E(\mathcal{L}_n)|/r^d)^{1/d}} \right),
    \]
    where the  last equality follows from~\eqref{eqn:lowerBnd_piA}.
    Hence, by our case assumption,
    \begin{equation}\label{eqn:varphiA2}
    \varphi_A =\Omega\left(\frac{r}{(t|E(\mathcal{L}_n)|/r^d)^{1/d}} \right),
    \end{equation}
    By Theorem~\ref{thm:concentration_vertices_edges_rgg}, using once more that $|E(\mathcal{L}_n)|=\Theta(nr^d)$ w.h.p., from~\eqref{eqn:varphiA1} and~\eqref{eqn:varphiA2}, we conclude that for all $A\subseteq\mathcal{L}_n$ such that 
    $0<\pi(A)\le t_0$, w.h.p.,
    \[
    \varphi_A = \Omega\Big(\min\Big\{\frac{1}{nf(n)r^d}, \frac{r}{(nt)^{1/d}}\Big\}\Big).
    \]
    The claimed result follows by definition of $\varphi(t)$. To finish the proof of the lemma, if $t_0\leq t\leq\frac12$, it suffices to prove that $\varphi(t)=\Omega(r/n^{1/d})$. To see that the latter is true, note that the value of $\varphi(1/2)$ is attained when both $A$ and $\mathcal{L}_n\setminus A$ are connected (see Lemma~2.5 of~\cite{benjamini2003mixing}). An argument similar to the one used to handle the case $f(n)\le t\le t_0$ implies the claimed bound on $\varphi(1/2)$: indeed, as both $A$ and $A^c$ are connected, we may apply Theorem~\ref{thm:isoperimetric_inequality_rgg} to the smaller of them, as it must have at most half of the vertices of $\mathcal{L}_n$. This yields a lower bound of the correct order for $E(A,A^c)$, which in turn provides the desired bound $\varphi(1/2)=\Omega(r/n^{1/d})$.
    As $\varphi$ is a monotone non-increasing function on $t$, we use the trivial bound $\varphi(t)\ge \varphi(1/2)=\Omega(r/n^{1/d})$ for $t\in[t_0,\frac{1}{2}]$. 
\end{proof}

\begin{lemma}\label{lem:lower_bound_varphi2}
There exists $C' > 0$ such that for every $c>0$, if $r^d\ge C'\log n$ and $cr^d/n\le t\le 1/2$, then w.h.p.,
\[
\varphi(t) = \Omega\Big(\frac{r}{(nt)^{1/d}}\Big).
\]
\end{lemma}
\begin{proof}
The proof is similar to that of Lemma~\ref{lem:lower_bound_varphi}, so we only give a rough sketch. 
Set~$f(n)=2c r^d/n$ and proceed as in the mentioned proof in the case where $f(n)< \pi(A)\le t<\frac12$.
Then, since standard concentration arguments imply that w.h.p.~$\pi(A)\le 2|A|/n$, for $\pi(A)>f(n)$ we get that 
$|A|\geq c r^d$ so we can apply Theorem~\ref{thm:isoperimetric_inequality_rgg} and obtain that, w.h.p., 
\[
\varphi_A \geq \frac{|E(A,A^c)|}{2\pi(A)|E(\mathcal{L}_n)|}
= \Omega\Big(\frac{|A|^{(d-1)/d}r^{d+1}}{\pi(A)|E(\mathcal{L}_n)|}\Big)
= \Omega\Big(\frac{r}{(nt)^{1/d}}\Big),
\]
where for the last equality we used Theorem~\ref{thm:concentration_vertices_edges_rgg}.
\end{proof}

We can now give the proof of the upper bound of our main theorem, assuming for now that Theorem~\ref{thm:isoperimetric_inequality_rgg} holds:
\begin{proof}[Proof (of upper bound on the mixing time in Theorem~\ref{thm:main_thm_bounds_over_mixing_rgg})]
Assume that $r^d=O(\log n)$. 
Let~$f(n)$ and~$\varphi(t)$ be as in the statement of Lemma~\ref{lem:lower_bound_varphi}.
Define~$g(n)=(nf(n)r^{d+1})^d/n$ and observe that by Lemma~\ref{lem:lower_bound_varphi}, $\varphi(t)=\Omega(r/(nt)^{1/d})$ if $t\ge g(n)$ and~$\varphi(t)=\Omega(1/(nf(n)r^d))$ if $0<t< g(n)$.

Recall that $\pi(x) \ge 1/2(|E(\mathcal{L}_n)|)$ for all $x\in\mathcal{G}_n$. Thus, partitioning the range of integration of the integral 
in the statement of Theorem~\ref{thm:upper_bound_mixing_lovasz_kannan} we get  
\[
\tau_{\text{mix}}(\mathcal{L}_n) = O\left( \int_{1/(2|E(\mathcal{L}_n)|)}^{g(n)}\frac{dt}{t\varphi^{2}(t)} + \int_{g(n)}^{1/2} \frac{dt}{t\varphi^{2}(t)} +\frac{1}{\varphi(1/2)}\right). 
\]
By our observation from above,  w.h.p.,~we have 
\[
\int_{g(n)}^{1/2} \frac{dt}{t\varphi^{2}(t)} = O\left( \int_{g(n)}^{1/2} \frac{(nt)^{2/d}}{r^{2}t} dt \right) = O\left( \int_{g(n)}^{1/2} \frac{n^{2/d}}{r^{2}t^{(d-2)/d}}dt \right) = O\left( \frac{n^{2/d}}{r^2} \right)
\]
and 
\[
\frac{1}{\varphi(1/2)} = O\left(\frac{n^{1/d}}{r} \right) =O\left( \frac{n^{2/d}}{r^2} \right).
\]
On the other hand, again by the above observation and by Theorem~\ref{thm:concentration_vertices_edges_rgg}, w.h.p., 
\begin{align*}
\int_{1/(2|E(\mathcal{L}_n)|)}^{g(n)} \frac{dt}{t\varphi^{2}(t)} 
& = O\Big(\int_{1/(2|E(\mathcal{L}_n)|}^{g(n)} \frac{(nf(n)r^d)^2}{t} dt\Big) \\
& = O\Big(\big(nf(n)r^d\big)^2\log\Big((nf(n)r^{d+1})^d\frac{|E(\mathcal{L}_n)|}{n}\Big)\Big) \\
& = O\left( \frac{n^{2/d}}{r^2} \right),
\end{align*}
where the last identity holds because 
$nf(n) r^d=(\log n)^{5d}r^{2d}$ and 
$|E(\mathcal{L}_n)|=\Theta(nr^d)$  w.h.p., with room to spare since $r^d=O(\log n)$. 

To extend the upper bound on the mixing time to the regime $r^{d} \ge C'\log n$ for $C'$ large, recall that in this case the degree is highly concentrated around $r^d$, this is, w.h.p. $\deg(x) =\Theta(r^d)$ for every vertex in $\mathcal{G}_n$, and as a consequence there is a $c>0$ such that $\pi_1\ge cr^d/n$. The result follows directly by application of Theorem~\ref{thm:upper_bound_mixing_lovasz_kannan} together with Theorem~\ref{thm:isoperimetric_inequality_rgg}. 
\end{proof}

We divide the proof of Theorem~\ref{thm:isoperimetric_inequality_rgg} into three parts according to the value of $r$. First, we address the \emph{large radii regime} where $r^d\ge C'\log n$ for $C'$ large enough, then the \emph{intermediate radii regime} where $r'<r$ for a large enough but fixed $r'$ and $r^d=O(\log n)$ 
(the exact condition on $r'$ might depend on the particular statement where it appears, but we choose $r'$ large enough so that it applies everywhere). Finally, via a renormalization argument we consider the \emph{small radii regime} where $(1+\varepsilon)r_g\le r\le r'$ with $\varepsilon>0$. 

\subsection{Large radii regime}
This regime, that is, the case when $r^d\ge C'\log n$ for $C'>0$ large, is relatively simple to handle.
In this case, standard large deviation bounds together with a straightforward union bound suffice to show that every tile in~$\Lambda_n^{r_d}$ contains, up to a lower order factor, 
the same number of vertices, namely the expected number of vertices of $\mathcal{G}_n$  in a region of volume $\mathrm{vol}_{r_d}$.
Then, the analysis is divided into two cases depending on the set $A \subseteq \mathcal{L}_n$.
In the first one, we assume the vertices of $A$ are, mostly, distributed in tiles that have an abundance of vertices not belonging 
to $A$. The choice of $r_d$ then guarantees that most vertices $v \in A$ are incident to $\Omega(r^d)$ vertices not belonging to $A$ but in the same tile as $v$, therefore significantly contributing to $E(A, A^c)$. On the other hand, if vertices in $A$ are packed into tiles each containing few vertices not belonging to $A$, 
we rely on isoperimetric bounds 
from Section~\ref{sec:percolation} to show that sufficiently many tiles containing many vertices of $A$ are adjacent to tiles containing many vertices not in~$A$ so that each such pair contributes many edges to $E(A,A^c)$.

\begin{proposition}\label{prop:smal_radii_regime}
If $r^d\ge C'\log n$ for $C'>0$ large enough, then for any $C > 0$
there is a constant $c>0$ so that w.h.p.~the following holds: 
For every $A\subseteq\mathcal{L}_n$ (not necessarily connected) such that 
$|A|\ge Cr^d$ and $\pi(A)\le \frac{1}{2}$, the number of edges between $A$ and $A^c$ is at least $c|A|^{(d-1)/d}r^{d+1}$.
\end{proposition}
\begin{proof}
Say that a tile $\tau\in\Lambda_n^{r_d}$ is \emph{normal} if it contains at least $\frac{19}{20}\mbox{vol}_{r_d}$ and at most $\frac{21}{20}\mbox{vol}_{r_d}$ vertices of~$\mathcal{G}_n$, that is, $\frac{19}{20}\mbox{vol}_{r_d}\leq |\tau\cap\mathcal{G}_n|\leq \frac{21}{20}\mbox{vol}_{r_d}$. 
By choosing $C'$ large enough and by the hypothesis $r^d\ge C'\log n$, by Chernoff bounds and a union bound over the $|\Lambda_n^{r_d}|_t=O(n/r^d)=o(n)$ tiles in $\Lambda_n^{r_d}$ we may assume throughout the proof that w.h.p.~every tile in~$\Lambda_n^{r_d}$ is normal.

Again, because $r^d\ge C'\log n$ with $C'$ large enough, by Chernoff bounds and a union bound, we may also assume throughout the proof that,  w.h.p., the degree of every vertex of $\mathcal{G}_n$ is sufficiently close to its expected value $\mbox{vol}_{r_d}$ so that for every $A\subseteq\mathcal{G}_n$ it holds that $\pi(A)\ge \frac{10}{11}(|A|/n)$.

Now, note that
\begin{align}\label{eqn:partition}
\big|E(A,A^c)\big| & 
\ge \sum_{\tau\in\Lambda_n^{r_d}}\big|E(\tau\cap A,\tau\cap A^c)\big| 
+ 
\sum_{(\tau,\widehat{\tau})\in (\Lambda_n^{r_d})^2
\colon \{\tau,\widehat{\tau}\}\in E(\Lambda_n^{r_d}) }
\big|E(\tau\cap A,\widehat{\tau}\cap A^c)\big|.
\end{align}
Say that a tile $\tau\in\Lambda_n^{r_d}$ is \emph{$A$-sparse} if it contains at least one and at most $\frac{9}{10}\mbox{vol}_{r_d}$ elements of $A$, that is,~$0<|\tau\cap A|\leq \frac{9}{10}\mbox{vol}_{r_d}$. 
Let $S_A\subseteq\Lambda_n^{r_d}$ be the collection of $A$-sparse tiles.
By our assumption about all tiles being normal and the definition of $A$-sparse, every $\tau\in S_A$ contains at least $\mbox{vol}_{r_d}/20$ vertices in~$\tau\cap(\mathcal{G}_n\setminus A)$.

Thus,  since any two vertices belonging to the same tile are at distance at most $\sqrt{d}r_d\leq r/2$ (by definition of $r_d$) from each other, for every $A$-sparse tile $\tau$, 
\[
\big|E(\tau\cap A,\tau\cap A^c)\big| 
= |\tau\cap A|\cdot |\tau\cap (\mathcal{G}_n\setminus A)| \geq |\tau\cap A|\cdot\mbox{vol}_{r_d}/20.
\]
Hence, the first term in the RHS of~\eqref{eqn:partition} is at least
\begin{align*}
\sum_{\tau\in S_A}\big|E(\tau\cap A,\tau\cap A^c)\big|  
& 
\geq (\mbox{vol}_{r_d}/20)\sum_{\tau\in S_A} |\tau\cap A| 
= (\mbox{vol}_{r_d}/20)\cdot| A\cap S_A|
= \Omega(r^d|A\cap S_A|).
\end{align*}
So, if $|A\cap S_A|\ge \frac12 C^{1/d}|A|^{(d-1)/d}r$, then the statement of the proposition follows. 
Thus, assume otherwise, that is, 
$|A\cap S_A|<\frac12 C^{1/d}|A|^{(d-1)/d}r
\le \frac12 |A|$ where the last inequality is because by hypothesis $|A|\ge Cr^d$. 
Now, say that tile $\tau\in\Lambda_n^{r_d}$ is \emph{$A$-dense} if it contains more than $\frac{9}{10}\mbox{vol}_{r_d}$ vertices of $A$. Let~$D_A\subseteq\Lambda_n^{r_d}$ be the collection of $A$-dense tiles.
Since $S_A$ and $D_A$ are disjoint and $A=(A\cap S_A)\cup (A\cap D_A)$, it follows that~$|A\cap D_A|\ge \frac12 |A|$.
Recalling that, by assumption, each tile contains at most~$\frac{21}{20}\mbox{vol}_{r_d}=\Theta(r^d)$ vertices of $\mathcal{G}_n$, we get that
\[
\big|D_A\big|_t
\geq \frac{|A|}{2\frac{21}{20}\mbox{vol}_{r_d}}
= \Omega(|A|/r^d).
\]
By definition of $A$-dense, if $|D_A|_t>\frac{11}{18}|\Lambda_n^{r_d}|_t$,  then 
\[
\pi(A) \geq \frac{10}{11}\cdot\frac{|A|}{n} 
\geq \frac{10}{11}\cdot \frac{9}{10}\mbox{vol}_{r_d}\frac{|D_A|_t}{n}> \frac{1}{2},
\]
contradicting the upper bound hypothesis on $\pi(A)$.
Thus, it must hold that $|D_A|_t\leq \frac{11}{18}|\Lambda_n^{r_d}|_t$. Then, we can apply Lemma~\ref{lem:isoperimetric_in_lattice_z_d}, and obtain that

\[
\big|E(D_A,D^c_A)\big|=\Omega(|A|^{(d-1)/d}/r^{d-1}).
\]
Now, let $\{\tau,\widehat{\tau}\}\in E(D_A,D^c_A)$ where, without loss of generality, 
$\tau\in D_A$ (thus, $\widehat{\tau}\in D^c_A$).
Since $\widehat{\tau}$ is not $A$-dense, it contains at most $\frac{9}{10}\mbox{vol}_{r_d}$ elements of $A$ (it might not contain any), so by our assumption of all tiles being normal, $\widehat{\tau}$ must contain at least
$\mbox{vol}_{r_d}/20$ elements not in $A$. Hence,  
\[
\big|E(\tau\cap A,\widehat{\tau}\cap A^c)\big|
= |\tau\cap A|\cdot |\widehat{\tau}\cap A^c|
\geq (9\mbox{vol}_{r_d}/10)\cdot(\mbox{vol}_{r_d}/20\big)^2
= \Omega(r^{2d}).
\]
Summing over all  $\{\tau,\widehat{\tau}\}\in E(D_A,D^c_A)$ we get  the desired conclusion since 
the second term in the RHS of~\eqref{eqn:partition}
must be at least
\[
|D_A|_t\cdot \Omega(r^{2d}) 
=\Omega(|A|r^d) 
=\Omega\big(|A|^{(d-1)/d}r^{d+1}),
\]
where the last equality follows from our assumption of $|A|=\Omega(r^d)$. 
The proposition follows.
\end{proof}

\subsection{Intermediate radii regime}
We next deal with smaller radii, that is, we assume throughout this section that $r^d \le C'\log n$ for some $C'>0$, but we still assume $r > r'$ for $r'$ being a sufficiently large constant. The argument of the previous subsection does not work anymore, since not all tiles contain the expected number of vertices, and a lot of extra work is needed to obtain the applicable isoperimetric inequality:
\begin{proposition}\label{prop:isoperimetric_inequality_rgg_large_r}
Let $0<\delta<1$ be a constant and let $r\ge r'$ for $r'$ being a large enough constant and $r^d \le C'\log n$ for $C'>0$ as in Proposition~\ref{prop:smal_radii_regime}. 
There exist constants $c_1, C_1>0$ such that w.h.p.~the following is satisfied: for every connected set $A\subseteq \mathcal{L}_n$ such that $C_1(r^{2d+1}(\log n)^{\frac{d}{d-1}})^d < |A| \le (1-\delta)|\mathcal{L}_n|$ it holds that $|E(A,A^c)\ge c_1|A|^{(d-1)/d}r^{d+1}$.

\end{proposition}
The proof of Proposition~\ref{prop:isoperimetric_inequality_rgg_large_r} is rather long and somewhat subtle. So, we will break it into more manageable pieces and try to motivate each one before stating it. 
We start by observing that, if there are at least 
$|A|^{(d-1)/d}r^{d+1}$ vertices in $A$ with each one having at least one neighbor not in $A$, 
then the size of~$E(A,A^c)$ is as claimed in Proposition~\ref{prop:isoperimetric_inequality_rgg_large_r} (by setting $c_1=1$). 
This motivates us to consider the set $A'\subseteq A$ of ``problematic''
vertices in the sense that all of their neighbors are also in $A$: denoting by~$B_x(r)\subseteq\mathbb{R}^d$ the Euclidean ball of radius $r$ centered at $x\in\mathbb{R}^d$, we define (see Figure~\ref{fig:first}.a)
\[
A'=\{x \in A \colon B_x(r)\cap \mathcal{G}_n \subseteq A\},
\]
and focus on the case where $|A\setminus A'|<|A|^{(d-1)/d}r^{d+1}$.
We first show that the latter assumption already implies that the size of the $k$-lattice animal in $\Lambda_n^{r_d}$ of tiles that contain vertices of~$A$ -- which we define as $L_A$ (see Figure~\ref{fig:first}.b) -- is not so large, reflecting that the vertices in most of the tiles in $L_A$ must all be in~$A$ and that typical tiles contain $\Theta(r^d)$ vertices of $\mathcal{G}_n$.

\begin{lemma}\label{lem:upperBndSizeOfLA}
Let $r\geq r'$ for $r'>0$ large enough be such that $r^d \le C'\log n$ with $C'>0$ as in Proposition~\ref{prop:smal_radii_regime}. For $C_1>0$ large enough, there exists a positive constant $c_2$, such that the following holds w.h.p.:
For every $L_A\subseteq\Lambda_n^{r_d}$ where $A\subseteq\mathcal{G}_n$ is connected, $|A|>C_1\cdot \max\{r^{(2d+1)d},r^d\cdot \log n\}$, and $|A\setminus A'|<|A|^{(d-1)/d}r^{d+1}$, we have
\[
|L_A|_t< \frac{c_2|A|}{r^d}.
\]
\end{lemma}
\begin{proof}
Observe that for every set $A$ that is connected, $L_A$ is a $k$-lattice animal. We may consider an arbitrary such $A$ below meeting the conditions in the lemma's statement.
Let $p^*$ be as in Lemma~\ref{lem:percolation_connected_animals}.
For any tile $\lambda$ of $\Lambda_n^{r_d}$, since $|\lambda\cap\mathcal{G}_n|$ is a Poisson random variable with mean~$\mathrm{vol}_{r_d}$, by Chernoff bounds for Poisson random variables, since $r \ge r'$ and $r'$ is large enough, there exists $\xi>0$ such that~$\mathbb{P}(|\lambda\cap\mathcal{G}_n|\ge\xi\mathrm{vol}_{r_d}) > p^*$.
In particular, since $|A|>C_1 r^d\log n$ we may apply Lemma~\ref{lem:bounds_vertices_big_animals} by making $C_1$ large enough 
and obtain $|L_A|_t \ge |A|/(2\mathrm{vol}_{r_d})>c_3\log n$ w.h.p. for some constant $c_3$ that is large if $C_1$ is large. Hence, by Lemma~\ref{lem:percolation_connected_animals}, at least $1/2$ of the tiles of~$L_A$ contain at least $\xi \mathrm{vol}_{r_d}$ vertices of $\mathcal{G}_n$.

Since by hypothesis $|A\setminus A'|<|A|^{(d-1)/d}r^{d+1}$, by definition of $L_A$, at most $|A|^{(d-1)/d}r^{d+1}$ tiles of~$L_A$ may contain vertices in~$A\setminus A'$. However, every tile $\lambda$ in~$L_A$ contains elements of $A$, so if it does not contain vertices in~$A\setminus A'$, by definition of~$A'$, it must be the case that $\lambda\cap\mathcal{G}_n\subseteq A$.
It follows that at least
\begin{equation}
    \frac12 |L_A|_t -|A|^{(d-1)/d}r^{d+1} 
\end{equation}
tiles of~$L_A$ must contain at least $\xi\mathrm{vol}_{r_d}$ vertices of~$A$. If we had $|L_A|_t \ge 3|A|/(\xi \mathrm{vol}_{r_d})$, then $L_A$ would contain at least (recall the hypothesis $|A|>C_1 r^{(2d+1)d}$ and make $C_1$ large) 
\begin{equation}
 \left(\frac{3}{2}\frac{|A|}{\xi \mathrm{vol}_{r_d}}  - |A|^{(d-1)/d}r^{d+1}\right)\cdot \xi \mathrm{vol}_{r_d} >|A|
\end{equation}
vertices of $A$ which is clearly not possible, since trivially the number of vertices of $A$ inside $L_A$ equals~$|A|$. We conclude by setting $c_2=\frac{3r^d}{\xi\mathrm{vol}_{r^d}}$, which is a constant for any $r>0$.
\end{proof}

For $A\subseteq\mathcal{G}_n$ satisfying our running assumption, that is, $|A\setminus A'|<|A|^{(d-1)/d}r^{d+1}$, 
not only does~$A'$ contain almost all vertices of $A$, but we will establish that in fact almost all vertices of $A'$ belong to relatively large  components induced by $A'$ in $\mathcal{G}_n$.
Specifically, let $\mathcal{A}'$ be the  components induced by $A'$ in~$\mathcal{G}_n$ (see Figure~\ref{fig:second}.a), 
and define (see Figure~\ref{fig:second}.b) 
\[
\mathcal{A'}_{\smlSucceq} = \big\{S\in \mathcal{A}' \colon |S|\ge \mu\cdot(\log n)^{d/(d-1)}\cdot \mathrm{vol}_{r_d}\big\}.
\]
where $\mu$ is a large enough constant such that all calculations that follow apply. Define $A''$ to be the collection of vertices that belong to some $S\in\mathcal{A}'_{\smlSucceq}$, that is, $A''=\bigcup_{S\in\mathcal{A}'_{\smlSucceq}} S\subseteq A'$ (see Figure~\ref{fig:second}.c).
We next show that, under our running assumptions and mild constraints on the size of $A$, almost all vertices of $A$ belong to $A''$.

\begin{figure}
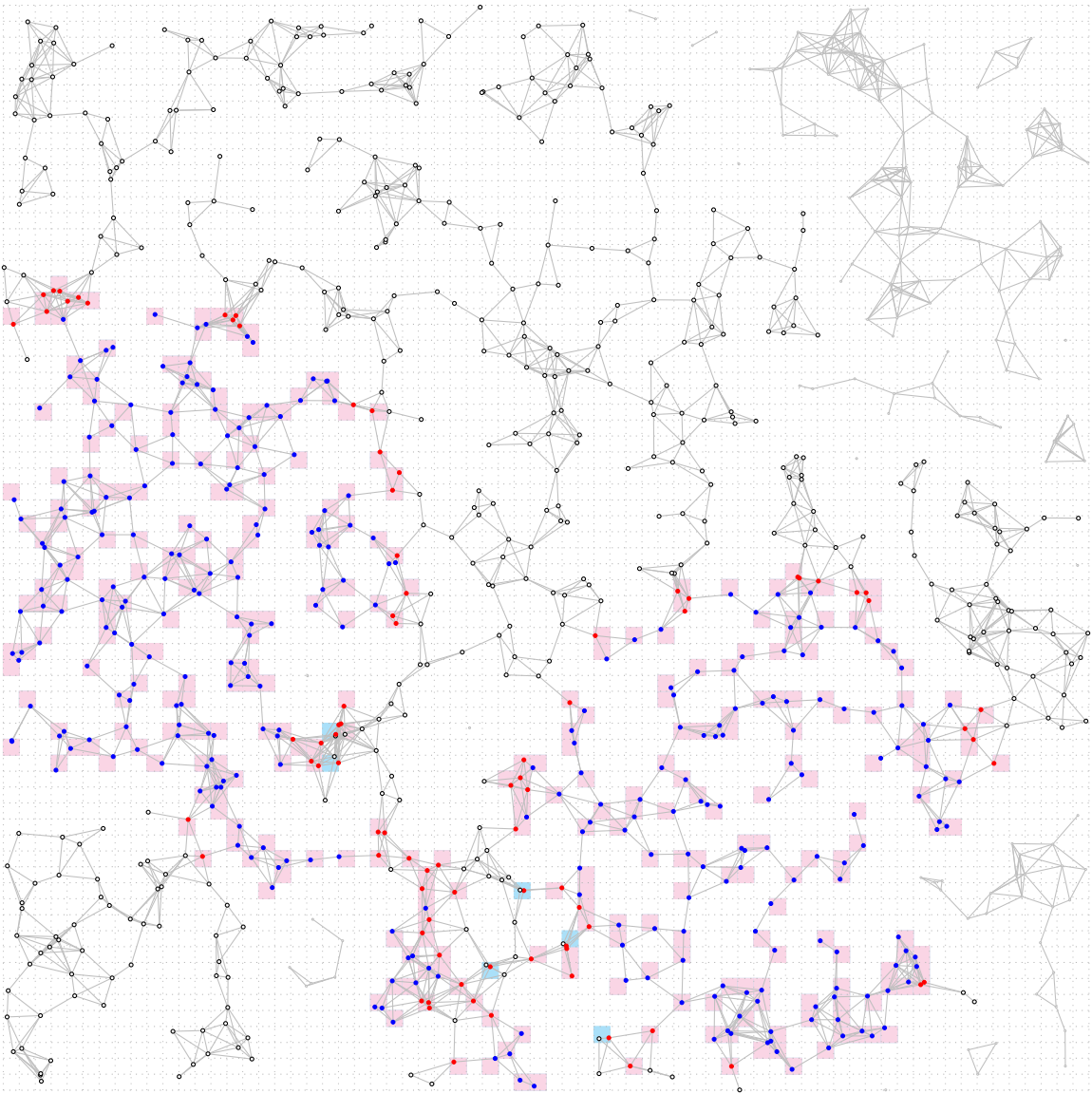

\begin{center}
  \tikzset{internal/.style={inner sep=0, outer sep=0, line cap=rect}}

\end{center}
\caption{Vertices of $\mathcal{G}_n$ are represented as circles. Small circles in gray represent vertices that do not belong to $\mathcal{L}_n$.
(a) Elements of $A'$ are shown as blue circles while elements of $A\setminus A'$ are shown in red. (b)
The tiles that are colored (cyan or magenta) are the tiles in $L_A$. 
(c) Tiles in $T_A$ are colored in magenta.}\label{fig:first}
\end{figure}

\begin{lemma}\label{lem:AprimeComponents}
Let $A \subseteq \mathcal{G}_n$ be arbitrary. For any $0<\eps<1$, there exist positive constants $C_1$ and $n_0$ such that if $n>n_0$, $A\subseteq\mathcal{G}_n$ 
    with $|A|>C_1 r^{(2d+1)d}(\log n)^{d^2/(d-1)}$ 
    and
    $|A\setminus A'|<|A|^{(d-1)/d}  r^{d+1}$, then $|A''|\ge(1-\eps)|A|$.
\end{lemma}
\begin{proof}
Let $\mathcal{A}'_{\smlPrec}=\mathcal{A}'\setminus\mathcal{A}'_{\smlSucceq}$.
By our hypotheses regarding the sizes of $A$ and $A\setminus A'$, we have 
\begin{equation}
|A'| =|A|-|A\setminus A'|=|A|(1-|A\setminus A'|/|A|)>|A|(1-r^{d+1}/|A|^{1/d}) \ge|A|(1-O(1/(\log n)^{d/(d-1)}).
\end{equation}
Since each vertex in $A\setminus A'$ is adjacent to a vertex of at most a constant number of different components in $\mathcal{A}$ (the constant depends only on the dimension $d$), we deduce that $|\mathcal{A}'_{\smlPrec}|\leq |\mathcal{A}'|<c'|A|^{(d-1)/d}r^{d+1}$ for some constant $c'$ that only depends on $d$.

For the sake of contradiction suppose that $|A''|\leq (1-\eps)|A|$. Then, 
\[
|A|(\eps+O(1/(\log n)^{d/(d-1)}))
\,\le\, |A'|-\big|A''|
\,=\, \sum_{S\in \mathcal{A}'_{\smlPrec}}|S|
\,\le\, |\mathcal{A}'_{\smlPrec}|\max_{S\in\mathcal{A}'_{\smlPrec}} |S|
\,\le\, |\mathcal{A}'_{\smlPrec}|\cdot \mu (\log n)^{d/(d-1)} \mathrm{vol}_{r_d},
\]
where the last inequality follows because $|S|\le \mu (\log n)^{d/(d-1)} \mathrm{vol}_{r_d}$ for all $S\in\mathcal{A}'_{\smlPrec}$.
Since $|\mathcal{A}'_{\smlPrec}| < c'|A|^{(d-1)/d}r^{d+1}$
we get that
\begin{equation}
    |\mathcal{A}'_{\smlPrec}|\cdot \mu (\log n)^{d/(d-1)} \mathrm{vol}_{r_d}\le c'' |A|^{(d-1)/d}r^{2d+1} (\log n)^{d/(d-1)}
\end{equation}
where $c''$ is a constant depending on $c',\mu$. Furthermore, by multiplying and dividing by $|A|^{1/d}$ the term to the right and using the hypothesis over the size of $A$, we obtain
\begin{equation}
    |A|(\eps+O(1/(\log n)^{d/(d-1)}))\le c''|A| \left(\frac{r^{2d+1}\cdot(\log n)^{d/(d-1)}}{|A|^{1/d}}\right) \le c''|A|\frac{1}{(C_1)^{1/d}}
\end{equation}
by setting constants $C_1,n_0$ large enough, and for $n>n_0$, we reach a contradiction, thus finishing the proof.
\end{proof}
To recap, unless Proposition~\ref{prop:isoperimetric_inequality_rgg_large_r} trivially holds, if $A$ is large enough, then most vertices in $A$ are in~$A'$, that is, most vertices in $A$ have all their neighbors also in $A$. In turn, most of the vertices in $A$ belong to relatively large components induced by $A'$ in $\mathcal{G}_n$.
However, if a vertex is in $A'$ and belongs to tile $\tau\in\Lambda_n^{r_d}$, since~$\sqrt{d}r_d\le r/2< r$, then all vertices in $\tau\cap\mathcal{G}_n$ also belong to $A$.
Lemma~\ref{lem:AprimeComponents} establishes that, under our running assumption, a significant proportion of the vertices in $A$ belongs to relatively large $k$-lattice animals that contain vertices only in $A$. 
The isoperimetric inequalities of Lemma~\ref{lem:isoperimetric_in_lattice_z_d} 
imply that there are many tiles adjacent to the aforementioned lattice animals. 
Each of these tiles must have vertices not in $A$ (otherwise all vertices in the tile would belong to $A$ and thus, typically, would belong to one of the mentioned $k$-lattice animals). 
If many of these tiles have many vertices not in $A$, say $\Omega(r^d)$, and are also adjacent to a tile~$\tau$ having many vertices in $A$, say $\Omega(r^d)$, and belonging to one of the mentioned $k$-lattice animals,  then we would have a significant number of pairs $(\tau,\widehat{\tau})$ contributing $\Omega(r^{2d})$ edges each to $E(A,A^c)$ (because any two points in adjacent tiles of $\Lambda_n^{r_d}$ are at distance at most $r$) and  we might leverage this fact to complete the proof of Proposition~\ref{prop:isoperimetric_inequality_rgg_large_r}. To succeed, the just discussed approach needs to guarantee that many tiles on the external boundary of large $k$-lattice animals whose tiles contains vertices exclusively in $A$, must not consist of vertices mostly in~$A$. 
Unfortunately, this may actually occur. Instead, we look at ``augmented'' $k$-lattice animals that contain the ``original'' $k$-lattice animals we started with. These ``augmented'' structures consist of tiles that have more than half their vertices in~$A$ and each of these tiles is reachable (in a sense that will be made clear) in $\Lambda_n^{r_d}$ from one of the ``original'' $k$-lattice animals. 
It is on the edge-boundary of these ``augmented'' $k$-lattice animals that we will later show there are many edges contributing to~$E(A,A^c)$. Next, we formalize the notion of augmented $k$-lattice animal. 

Let $T_A$ be the collection of non-empty tiles in $\Lambda_n^{r_d}$ such that at least half of the vertices in~$\mathcal{G}_n$ in the tile belong to $A$ (see Figure~\ref{fig:first}.c). More precisely, 
\[
T_A = \big\{ \tau\in\Lambda_n^{r_d} \colon |\tau\cap A|\geq |\tau\cap\mathcal{G}_n|/2 > 0\big\}.
\]
We remark that all tiles of $T_A$ contain at least one vertex of $\mathcal{G}_n$. Also, observe that a tile of $\Lambda_n^{r_d}$ that contains a vertex of $A'$ must belong to $T_A$.

Let $\mathfrak{L}(A)$ be the collection of components induced by $T_A$ in the $k$-transitive closure of~$\Lambda_n^{r_d}$ that contain some vertex of~$A''$ (see Figure~\ref{fig:second}.d). 
Formally,
\[
\mathfrak{L}(A) = \big\{ L \subseteq\Lambda_n^{r_d}{[ T_A ]} \colon 
\text{$L$ is a maximal $k$-lattice animal and $L\cap A''\neq\emptyset$} \big\}.
\]
Elements of $\mathfrak{L}(A)$ correspond to what we refereed to as ``augmented'' lattice animals in our informal discussion above.

\begin{figure}
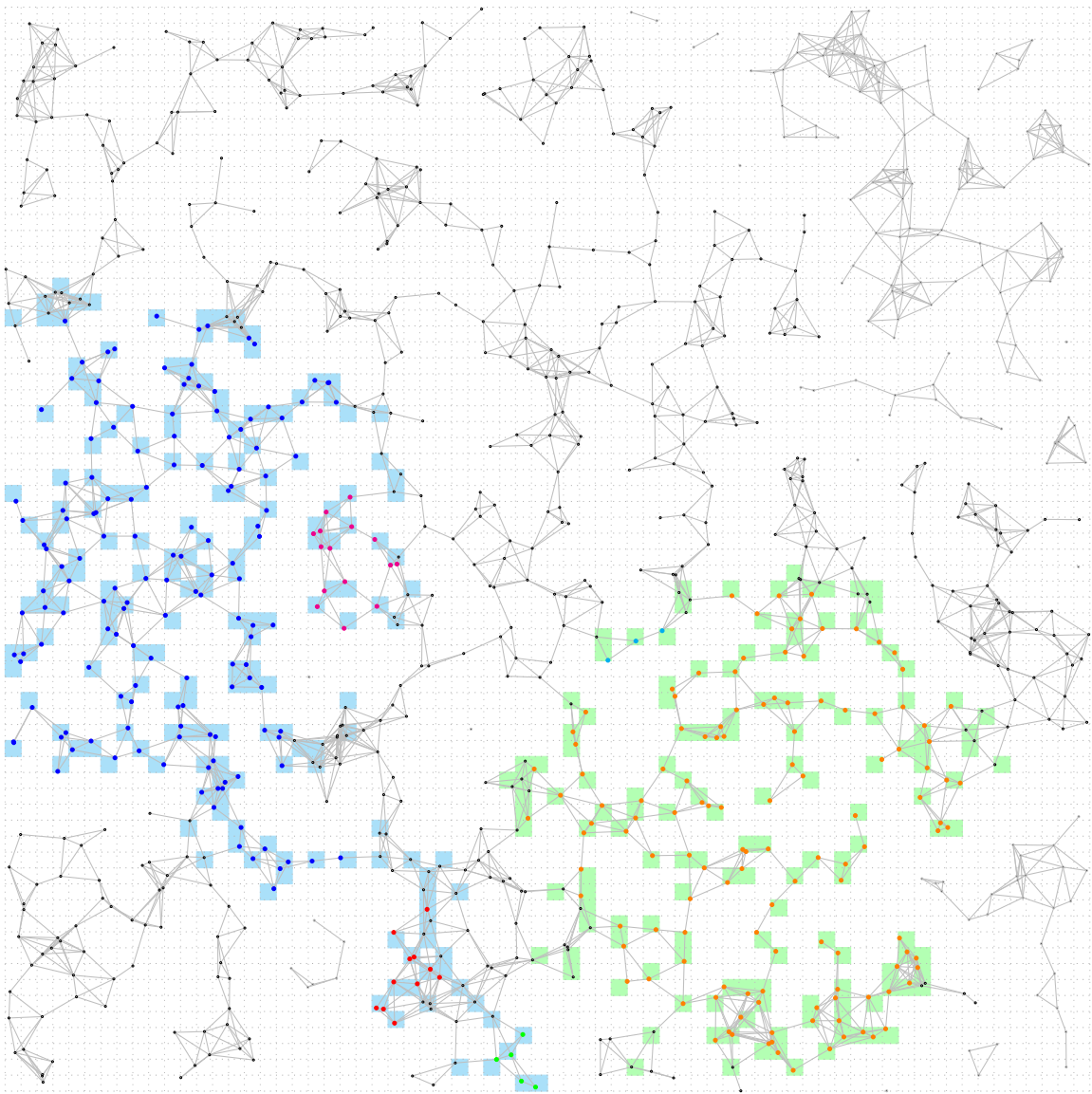

\begin{center}
  \tikzset{internal/.style={inner sep=0, outer sep=0, line cap=rect}}

\end{center}
\caption{
Vertices of $\mathcal{G}_n$ are represented as circles.
Small circles in gray represent vertices that do not belong to $\mathcal{L}_n$.
  Solid vertices shown in colors correspond to elements of $A'$. (a) Distinct elements of $\mathcal{A}'$ (components induced by $A'$ in $\mathcal{L}_n$) are shown in different colors.
  (b) For illustration purposes, we consider as elements of~$\mathcal{A}'_{\smlSucceq}$ components of~$\mathcal{A}'$ of size at least~$20$ which are illustrated as blue and orange colored vertices.
  (c) Vertices in colored tiles correspond to $A''$.
  (d) Again, for the sake of illustration, we consider $k=3$ and color the tiles belonging to each of the two $3$-lattice animals in~$\mathfrak{L}(A)$
  in green and cyan, respectively.
  (e) Uncolored tiles are *-connected and contain elements of $\mathcal{L}_n\setminus A$ (see Figure~\ref{fig:first} for a depiction of the set $A$). Hence, in this example, $\mathfrak{K}(A)$ contains a unique element, consisting of the collection of uncolored tiles.
(f) For the illustrated instance, for each $L\in\mathfrak{L}(A)$, it holds that $\mathfrak{M}(L)=\{L^c\}$.}\label{fig:second}
\end{figure}

The next result establishes a lower bound on the number of tiles of each element of $\mathfrak{L}(A)$ and shows that collectively  all elements of $\mathfrak{L}(A)$ contain $\Omega(|A|/r^d)$ tiles of $\Lambda_n^{r_d}$.
\begin{lemma}\label{lem:propHats}
 Let $r\geq r'$ for $r'$ large enough. For any large enough positive constant 
 $C_3>0$, there exist constants $C_1,c_2>0$ such that the following holds w.h.p.: For every  connected set~$A\subseteq\mathcal{G}_n$ of size at least 
  $C_1 (r^{2d+1}(\log n)^{\frac{d}{d-1}})^{d}$
  satisfying ~$|A\setminus A'|<|A|^{(d-1)/d} r^{d+1}$, the following holds: 
  \begin{enumerate}
  \item\label{lem:propHats:itm2} $|\bigcup_{L\in\mathfrak{L}(A)}L|_t \ge c_2|A|/r^d$.
  \medskip
  \item\label{lem:propHats:itm3}  $|L|_t\ge C_3(\log n)^{d/(d-1)}$ for all $L\in\mathfrak{L}(A)$.
  \end{enumerate}
\end{lemma}
\begin{proof}
Let $A$ be a (fixed) set satisfying the hypothesis of the lemma.
First, we show that 
\begin{equation}\label{eqn:propHat1}
A''\subseteq \mathcal{G}_n\cap \bigcup_{L\in\mathfrak{L}(A)}L.
\end{equation}

Consider $v\in A''$.
Let $\tau$ be the tile of $\Lambda_n^{r_d}$ containing~$v$.
Since the side length of tile $\tau$ is $r_d$, by definition of $r_d$, any two points in $\tau$ are at distance at most $r/2$ of each other.
Hence, $\tau\subseteq B_v(r)$.
Since $v\in A''\subseteq A'$, by definition of $A'$, we have that 
$\tau\cap\mathcal{G}_n \subseteq B_v(r)\cap\mathcal{G}_n\subseteq A$.
Moreover, $\tau\cap\mathcal{G}_n$ is non-empty, since $v$ belongs to $\tau$.
Thus, $\tau$ is a non-empty tile all of whose vertices belong to $A$, so by definition of $T_A$, it follows that $\tau\in T_A$.
Hence, given that $\tau$ contains~$v$ which is a vertex in $A''$,  by definition of $\mathfrak{L}(A)$, tile $\tau$ must belong to some $L\in\mathfrak{L}(A)$ and $v\in\mathcal{G}_n\cap L$.
This concludes the proof of~\eqref{eqn:propHat1}.

Next, observe that we can apply Lemma~\ref{lem:AprimeComponents}.
By~\eqref{eqn:propHat1}, since all $L$ in $\mathfrak{L}(A)$ are disjoint and by Lemma~\ref{lem:AprimeComponents} (this works for any $C_3$ large enough) 
, we get that 
\begin{equation}\label{eqn:propHat2}
\sum_{L\in\mathfrak{L}(A)} |\mathcal{G}_n \cap L|
\ge |A''| 
= (1-\eps)|A|.
\end{equation}
Consider $L\in\mathfrak{L}(A)$.
By definition of $\mathfrak{L}(A)$, there must exist $v\in A''$ such that $v\in L$. 
Let~$S\in\mathcal{A}'_{\smlSucceq}$ be such that $v\in S$.
Since $S$ is a connected set in $\mathcal{G}_n$ (because $S\in\mathcal{A}'$ and the elements of $\mathcal{A}'$ are, by definition,  components in $\mathcal{G}_n[A']$) 
and $L$ is a $k$-lattice animal in~$\Lambda_n^{r_d}$, it follows that $S\subseteq \mathcal{G}_n\cap L$.
Recalling that $|S|\geq \mu (\log n)^{(d/(d-1)}\mathrm{vol}_{r_d}$ (because $S\in\mathcal{A}'_{\smlSucceq}$), we get that $|\mathcal{G}_n\cap L|\ge \mu (\log n)^{d/(d-1)} \mathrm{vol}_{r_d}$. Hence, by making $\mu$ large enough, we can apply Remark~\ref{rem:bounds_vertices_big_animals} and deduce that, w.h.p., for any connected set $A$ as in the statement, $|L|_t>\delta |\mathcal{G}_n\cap L|/r^d$ for some positive $\delta$.
Summing over $L\in\mathfrak{L}(A)$ and using~\eqref{eqn:propHat2} yields~\ref{lem:propHats:itm2}. 
Moreover, since $S\subseteq \mathcal{G}_n\cap L$, using again that~$|S|\geq \mu(\log n)^{d/(d-1)} \mathrm{vol}_{r_d}$, we get that Part~\ref{lem:propHats:itm3} also holds by making $\mu$ large enough.
\end{proof}

Now, let $\mathfrak{K}(A)$ be the collection of *-components 
of tiles of $\Lambda_n^{r_d}\setminus \bigcup_{L\in\mathfrak{L}(A)} L$ that contain some vertex of~$\mathcal{L}_n$ not in $A$. In other words, we look at such components before intersecting with $\mathcal{L}_n$ and then discard those that have no vertex in $\mathcal{L}_n$ (see Figure~\ref{fig:second}.e). 
Formally, 
\[
\mathfrak{K}(A) =
\Big\{K\subseteq \Lambda_n^{r_d} \colon \text{$K\subseteq \big(\bigcup_{L\in\mathfrak{L}(A)} L\big)^c$ is a *-component and $K\cap(\mathcal{L}_n\setminus A)\neq\emptyset$}\Big\}.
\]
We remark that in general $K\in \mathfrak{K}(A)$ contains tiles with vertices not in the giant component (it could even contain tiles without any vertices of $\mathcal{G}_n$). 


For the sake of argument, consider $A\subseteq\mathcal{L}_n$ and take a component $K\in\mathfrak{K}(A)$.
By definition, $K$ contains a vertex $v_K$ in $\mathcal{L}_n\setminus A$. Moreover, tiles in $\partial_*^+K$ must contain vertices in $A\subseteq\mathcal{L}_n$. 
Since $\mathcal{L}_n$ is connected, there must be a path from $v_K$ to elements of $A$ in the external *-vertex-boundary of $K$. Thus, intuitively, each $K\in\mathfrak{K}(A)$ must contribute at least one edge to $E(A,A^c)$. Thus, if $\mathfrak{K}(A)$ is large enough, then $E(A,A^c)$ must be large and Proposition~\ref{prop:isoperimetric_inequality_rgg_large_r} should follow.
The following two results formalize and make precise this paragraph's discussion.
\begin{lemma}\label{lem:starComp}
If $A \subseteq \mathcal{L}_n$, then
for each $K\in\mathfrak{K}(A)$ there is at least one edge between $A$ and $\mathcal{L}_n\setminus A$ which is incident to a vertex in $(K\cup\partial_*^{+}K)\cap\mathcal{L}_n$.
\end{lemma}
\begin{proof}
    First, recall that, by definition of $\mathfrak{K}(A)$, for each $K \in \mathfrak{K}(A)$, there is a vertex $v\in K$ that belongs to $\mathcal{L}_n\setminus A$, and that each tile in~$\partial_*^{+}K$ has at least one vertex of $A$ (because $\partial_*^{+}K$ is contained in $\bigcup_{L\in\mathfrak{L}(A)}L$ and, by definition of $\mathfrak{L}(A)$, each tile in this union contains elements of $A$). Now pick $u \in A$ such that it belongs to some tile in $\partial_*^{+}K$.
    Since both $u$ and $v$ belong to $\mathcal{L}_n$, there is  a path in~$\mathcal{L}_n$ from $v$ to $u$. Let $v'$ be the first (starting from $v$) vertex such that  the line segment between~$v'$ and~$v''$, where~$v''$ is the vertex that follows $v'$ on the path, intersects~$\partial_*^{+}K$ (note that this does not mean that~$v' \in \partial_*^{+}K$, but $v' \in K \cup \partial_*^{+}K$).
    Observe that, for every path in $\mathcal{L}_n$ between $u$ and~$v$, this must happen at least once. 
    Now, starting from $v$ follow the path until we find a vertex of~$A$ (that belongs to the path). 
    In what follows we consider four cases.
    
    First, assume we find $u'\in A$ before $v'$ in the path, say $u'$ is the first such vertex. Since all vertices of the path before $v'$ are totally contained within $K$, we get that $u'\in K\cap A$ and the vertex $u''$ preceeding~$u'$ in the path must be in $K$ but not in $A$, so we have~$\{u',u''\}\in E(A,A^c)$ and~$u',u''\in K$. 
    
    Second, assume the first vertex on the path in $A$ is $v'$. Defining $u'$ as in the previous case, by the same argument as before, we conclude that~$\{u',v'\}\in E(A,A^c)$, $u'\in K$ and $v'\in\partial_*^+K$.  
    
    Third, assume $v'\not\in A$ and $v''\in A$, then $\{v',v''\}$ is an edge satisfying the desired condition. 
    
    The last case remaining is when 
    $v',v''\notin A$. We claim that either $v'$ or $v''$ is connected by an edge to a vertex of $A$ in some tile of $\partial_*^{+}K$. Let $\tau\in \partial_*^{+}K$ (recall that the side length of tiles is $r_d$) be a tile that intersects the line segment between $v'$ and $v''$ which we denote by $\overline{v'v''}$ (see Figure~\ref{fig:edge-through-square}). We prove that~$\tau\subseteq B_{v'}(r)\cup B_{v''}(r)$: 
    observe that any point inside $\tau$ is at distance at most $r_d\sqrt{d}$ (the distance between two antipodal vertices of $\tau$) from $\overline{v'v''}$.  On the other hand, the points on the boundary of $B_{v'}(r)\cup B_{v''}(r)$ that are closest to the segment $\overline{v'v''}$ are the ones on the intersection of the boundaries of $B_{v'}(r)$
    and $B_{v''}(r)$. All points on these boundary intersections are at distance exactly $\sqrt{r^2-(\|v'-v''\|_2/2)^2}\ge r\sqrt{3}/2\ge r_d$ from $\overline{v'v''}$ (because $\|v'-v''\|_2\le r$).
    However, by definition of $r_d$ we have that $r_d\sqrt{d}\leq r/2\leq r\sqrt{3}/2$, which proves that $\tau\subseteq B_{v'}(r)\cup B_{v''}(r)$.
    Since $\tau\in\partial^+_* K\subseteq\bigcup_{L\in\mathfrak{L}(A)}L$, by definition of $\mathfrak{L}(A)$, tile $\tau$ must contain an element of $A$, say $v'''$, and either $\{v',v'''\}$ or $\{v'',v'''\}$ is an edge in $E(A,A^c)$ satisfying the desired condition.
\end{proof}

    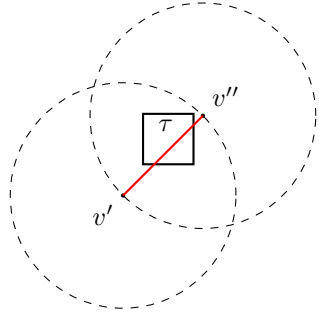
\begin{figure}    \centering
    \begin{tikzpicture}[scale=1.5]
    \def\r{1/sqrt(5)}
    \def\s{1/sqrt(2)}

    \draw[thick] ({\s/4},{\s/4+0.1}) rectangle ({\s/4+\r},{\s/4+\r+0.1});

    \node at ({\s/4+0.2},{\s/4+\r}) {$\tau$};

    \coordinate (v1) at (0,0);
    \coordinate (v2) at ({\s},{\s});
    \filldraw[black] (v1) circle (0.4pt) node[below left] {$v'$};
    \filldraw[black] (v2) circle (0.4pt) node[above right] {$v''$};

    \draw[red, thick] (v1) -- (v2);

    \draw[dashed] (v1) circle ({1});
    \draw[dashed] (v2) circle ({1});
    \end{tikzpicture}
    \caption{Vertices $v'$ and $v''$ are connected by an edge of length $r$ traversing the tile~$\tau$ of side length $r_d$, which is contained in either $B_{v'}(r)$ or $B_{v''}(r)$}\label{fig:edge-through-square}
    \end{figure}

\begin{corollary}\label{cor:starComp}
There exists a constant $\xi=\xi(d)$ such that if $A\subseteq\mathcal{L}_n$, then $|E(A,A^c)|\ge \xi|\mathfrak{K}(A)|$.
\end{corollary}
\begin{proof}
For every $K\in\mathfrak{K}(A)$ let $v_K$ be a vertex in $K\cup\partial_*^+K$ for which there is an edge between~$A$ and~$\mathcal{L}_n\setminus A$ incident to it (Lemma~\ref{lem:starComp} guarantees the existence of such a vertex). 
It may happen that~$v_K=v_{K'}$ for $K'\in\mathfrak{K}(A)\setminus\{K\}$. However, since every vertex is connected to at most constantly many (depending on the dimension $d$) distinct components in $\mathfrak{K}(A)$, there can be at most a constant number of $K'$'s for which $v_K=v_{K'}$. Thus, $\Omega(|\mathfrak{K}(A)|)$ 
elements of $\mathfrak{K}(A)$ can be associated to distinct edges between $A$ and $\mathcal{L}_n\setminus A$, and we are done.
\end{proof}

Thanks to the preceding result, we can incorporate  $|\mathfrak{K}(A)|<|A|^{(d-1)/d}r^{d+1}$ as part of our running assumption, since otherwise we are done ($|E(A,A^c)|\ge \xi|A|^{(d-1)/d}r^{d+1}$). Doing so amounts to saying that there are few *- components in $(\cup_{L\in\mathfrak{L}(A)}L)^c$ containing vertices of the giant but not of $A$. On the other hand, all tiles in $\cup_{L\in\mathfrak{L}(A)}L$ have at least half their vertices in $A$ and, 
since in expectation tiles have $\mbox{vol}_{r_d}=\Theta(r^d)$ vertices, one would expect that there are $O(|A|/r^d)$ tiles
in $\cup_{L\in\mathfrak{L}(A)}L$.
Thus, provided $|A|\le (1-\delta)|\mathcal{L}_n|$ for some constant $\delta>0$, we may presume that 
$(\cup_{L\in\mathfrak{L}(A)}L)^c$ contains at least a constant fraction of all tiles in $\Lambda_n^{r_d}$. 
But this suggests, assuming $\mathfrak{K}(A)<|A|^{(d-1)/d}r^{d+1}$, that it is enough to consider just large components of $\mathfrak{K}(A)$ to cover a significant proportion of the tiles in $\Lambda_n^{r_d}$ not belonging to $\cup_{L\in\mathfrak{L}(A)}L$.
Next, we define what we mean by being a not so small component of $\mathfrak{K}(A)$ as well as specify additional conditions for which the preceding discussion's conclusion remains valid. 

Denote by $\mathfrak{K}_{\smlSucceq}(A)$ the collection of $K$'s in $\mathfrak{K}(A)$ for which $|K|>\max\{(\log n)^{d/(d-1)},2\lceil\gamma \log n\rceil\mbox{vol}_{r_d}\}$, where $\gamma$ is as in the statement of Lemma~\ref{lem:bounds_vertices_big_animals}. 
\begin{lemma}\label{lem:size_union_large_componentes_D_i}
Consider $r^d\le C'\log n$ with $C'>0$ as in Proposition~\ref{prop:smal_radii_regime}
and $r\geq r'$ for $r'>0$ large enough. Let $0<\delta<1$, there exists $c>0$ such that w.h.p.~the following holds: For every connected set $A$ of vertices of $\mathcal{L}_n$ with 
$2\lceil\gamma\log n\rceil\textup{vol}_{r_d}\le |A|\le (1-\delta)|\mathcal{L}_n|$ such that $|A\setminus A'|<|A|^{(d-1)/d}r^{d+1}$ and $|\mathfrak{K}(A)|<|A|^{(d-1)/d}r^{d+1}$, we have
$$
\displaystyle 
\sum_{K\in\mathfrak{K}_{\smlSucceq}(A)} \big|K\big|_t \ge c|\mathcal{L}_n|/r^d.
$$
\end{lemma}
\begin{proof}
First, we assert that $|\mathcal{L}_n\cap (\cup_{K\in \mathfrak{K}(A)}K)|\ge (\delta-o(1))|\mathcal{L}_n|$. By the hypothesis $|A\setminus A'|<|A|^{(d-1)/d}r^{d+1}$, we see that $|L_{\mathcal{L}_n\setminus A}\cap L_{A}|_t<|A|^{(d-1)/d}r^{d+1}$. By the hypothesis on $r$, we have that w.h.p.~all the tiles of $\Lambda_n^\rho$ have $O(\log n)$ vertices of $\mathcal{G}_n$. Then, $|\mathcal{G}_n \cap (L_{\mathcal{L}_n\setminus  A}\cap L_A)|=|A|^{(d-1)/d}\cdot r^{d+1}\cdot O(\log n)=o(n)$ and this means that $|\mathcal{L}_n\cap T_A|=|A|+o(|\mathcal{L}_n|)$. Hence, almost all vertices of $\mathcal{L}_n\setminus A$ belong to a tile not in $L_A$, that is    

\begin{equation}\label{eq:lower_bound_vertices_giant_outside_L_A}
    |\mathcal{L}_n\cap (\Lambda_n^{r_d}\setminus L_A)|= |\mathcal{L}_n|- |\mathcal{L}_n\cap L_A| \ge (\delta-o(1))|\mathcal{L}_n|
\end{equation}
By definition of $\mathfrak{L}(A)$, every tile in some $L$ belonging to $\mathfrak{L}(A)$ contains vertices of $A$, so 
$\bigcup_{L\in \mathfrak{L}(A)} L\subseteq L_A$. Thus, 
\[
\mathcal{L}_n\cap (\Lambda_n^{r_d}\setminus L_A)\subseteq \mathcal{L}_n\cap \Big(\Lambda_n^{r_d} \setminus \bigcup_{L\in \mathfrak{L}(A)} L\Big)
= \mathcal{L}_n \cap \bigcup_{K\in \mathfrak{K}(A)}K,
\]
where the last equality follows by definition of $\mathfrak{K}(A)$, since the vertices of $\mathcal{L}_n$ that are not in some tile of $\bigcup_{L\in \mathfrak{L}(A)} L$ are in some tile of $\bigcup_{K\in \mathfrak{K}(A)}K$.
Hence,~\eqref{eq:lower_bound_vertices_giant_outside_L_A} implies the assertion.

Now, to prove the lemma, observe that the total number of vertices in $K$'s not in  $\mathfrak{K}_{\smlSucceq}(A)$ is at most $|\mathfrak{K}(A)|\cdot\max\{(\log n)^{d/(d-1)},2r^d \lceil \gamma \log n \rceil\}$, which by our assumption on $|\mathfrak{K}(A)|$ is at most
$$|A|^{(d-1)/d}r^{d+1}\cdot \max\{(\log n)^{d/(d-1)},2r^d \lceil \gamma \log n \rceil\}=o(|\mathcal{L}_n|),$$
and hence \[
\Big|\mathcal{L}_n\cap \bigcup_{K\in \mathfrak{K}_{\smlSucceq}(A)}K\Big|\geq (\delta-o(1))|\mathcal{L}_n|.
\]
Applying Remark~\ref{rem:bounds_vertices_big_animals} to each $K\in \mathfrak{K}_{\smlSucceq} (A)$ (the remark can be applied since by definition of $\mathfrak{K}_{\smlSucceq} (A)$, each such $K$ is large enough and *-connected, thence also a $k$-lattice animal),  there is $c_1>0$ such that
\[
  \sum_{K\in\mathfrak{K}_{\smlSucceq}(A)}\big|K\big|_t 
  \ge c_1\Big(\sum_{K\in \mathfrak{K}_{\smlSucceq}(A)} \big| K\cap\mathcal{L}_n\big|/r^d\Big) 
  = c_1 \big(\big|\mathcal{L}_n\cap\bigcup_{K\in\mathfrak{K}_{\smlSucceq}(A)} K\big|/r^d\big)\ge c_1(\delta-o(1))|\mathcal{L}_n|/r^d.    
\]
 We now pick $c=c_1\delta/2$ and the Lemma follows.
\end{proof}

Finally, we have all necessary ingredients to prove Proposition~\ref{prop:isoperimetric_inequality_rgg_large_r}.
To do so, similarly to the case of large radii,
we say that a tile $\tau\in\Lambda_n^{r_d}$ is \emph{normal} if it contains between $\frac12\mathrm{vol}_{r_d}$ and $2\mathrm{vol}_{r_d}$ vertices, that is, if $\frac12\mathrm{vol}_{r_d}\leq |\tau\cap\mathcal{G}_n|\le 2\mathrm{vol}_{r_d}$. 
Also, we say $(\tau,\tau')$ is a normal pair if $\tau$ and $\tau'$ are both normal tiles and adjacent in~$\Lambda_n^{r_d}$.
For~$L\in\mathfrak{L}(A)$, we denote by $\mathcal{N}_L$ the collection of normal pairs $(\tau,\tau')$
such that $\tau \in L$ and~$\tau'\not\in L$ (thus, by definition of $\mathfrak{L}(A)$,  since the components in $\mathfrak{L}(A)$ are inclusion-wise maximal, tile $\tau'$ does not belong to $\bigcup_{L\in\mathfrak{L}(A)}L$).
We claim that if $(\tau,\tau')\in\mathcal{N}_L$, then there are~$\Omega(r^{2d})$ edges between $A\cap\tau$ and $A^c\cap\tau'$: indeed, at least half of the vertices inside of $\tau$ belong to~$A$ and at least half the vertices inside of $\tau'$ belong to~$A^c$ (otherwise, $\tau'$ would belong to $L$).
The claim follows, since~$\tau$ and~$\tau'$ are both normal, and so 
\[
|E(A\cap\tau,A^c\cap\tau')| \ge
    \big(\tfrac12\mathrm{vol}_{r_d}\big)^2=\Omega(r^{2d}).
\]
Hence, to show Proposition~\ref{prop:isoperimetric_inequality_rgg_large_r} we only need to prove that $|\bigcup_{L\in\mathfrak{L}(A)}\mathcal{N}_L|= \Omega(|A|^{(d-1)/d}/r^{d-1})$, where the hidden constant is uniform over all $A$ meeting the required hypothesis.

\begin{proof}[Proof of Proposition~\ref{prop:isoperimetric_inequality_rgg_large_r}]
%
For every $k$-lattice animal $L\in\mathfrak{L}(A)$, 
let $\mathfrak{M}(L)$ be the collection of *-components induced in $\Lambda_n^{r_d}$ by the tiles not in $L$ (see Figure~\ref{fig:second}.f). 
Note that each component $M\in\mathfrak{M}(L)$ has the property that~$\partial_*^{+}M$ is a~$k'$-lattice animal: indeed, to see the latter, observe that $M^c = \Lambda_n^{r_d}\setminus M = L\cup (\cup_{M'\in\mathfrak{M}(L)\setminus \{M\}}M')$ is a~$k'$-lattice animal, which is the union of a $k$-lattice animal and *-components that share at least one pair of tiles; hence, by Lemma~\ref{lem:complement_star_connected_is_k_lattice_animal}, $\partial_*^-M^{c}=\partial_*^+M$ 
is a $k'$-lattice animal. 
We also note that~$\partial_*^{+}M\subseteq L$.

We remark that $K\in\mathfrak{K}_{\smlSucceq}(A)$ need not be an element of $\mathfrak{M}(L)$ when $L\in\mathfrak{L}(A)$. However, 
for every~$K\in\mathfrak{K}_{\smlSucceq}(A)$ and every $L\in\mathfrak{L}(A)$ there is an $M'\in \mathfrak{M}(L)$ such that $K\subseteq M'$: indeed, the components of $\mathfrak{L}$ might be larger as they are defined only on the complement of $L$. Next, we show that there are many normal pairs. 
Recall that we assume that $|A\setminus A'|<|A|^{(d-1)/d}r^{d+1}$ since otherwise we are done (see the discussion after the statement of Proposition~\ref{prop:isoperimetric_inequality_rgg_large_r}).

We divide the analysis into two cases. We first observe that there is an abundance of normal tiles. Indeed, considering normal and non-normal tiles as open and closed sites, respectively, defines a (site) percolation process. By concentration of Poisson random variables, the probability that a tile is normal tends to $1$ as $r\to \infty$. We henceforth fix $r'$ large enough so that for all $r>r'$ the probability that a given tile is normal is larger than $p^*$, where $p^*$ is as in Lemma~\ref{lem:pair_of_open_sites_boundary}.
We assume $r'$ large enough so Lemma~\ref{lem:size_union_large_componentes_D_i} holds.

\smallskip
\noindent\textbf{Case 1 (for every $L\in\mathfrak{L}(A)$ there is an $M\in\mathfrak{M}(L)$ such that $\big|M^c\big|_t\leq |M|_t$):} For every~$L\in\mathfrak{L}(A)$, let $M_L$ be an $M\in\mathfrak{M}(L)$ maximizing $|M|_t$ and satisfying $\big|M^c\big|_t\le |M|_t$ (ties broken arbitrarily). 
Recalling that $M^c=L\cup (\cup_{M'\in\mathfrak{M}(L)\setminus \{M\}}M')$, we have $M_L^c\supseteq L$, and since by Lemma~\ref{lem:propHats} Part~\ref{lem:propHats:itm3} we have $|L|_t>c_2(\log n)^{d/(d-1)}$ w.h.p. for $c_3$ large enough, we conclude that $|M_L|_t\ge |M_L^c|_t\ge c_3(\log n)^{d/(d-1)}$ w.h.p.
Then, as each $M_L$ is *-connected and each $M_L^c$ is a $k$-lattice animal, we can apply Lemma~\ref{lem:pair_of_open_sites_boundary} 
with $K=M_L^c$ 
and conclude that there are at least $\delta' |L|_t^{(d-1)/d}$ normal pairs $(\tau,\tau')$ such that $\tau\in M^c_L$ and $\tau'\in M_L$ (indeed, the open sites to which Lemma~\ref{lem:pair_of_open_sites_boundary} conclusion's refers to, by this proof's  first paragraph's discussion, corresponds to tiles being normal). 
Since $\tau\in\partial_*^{+}M_L \subseteq L$
and $\tau'\not\in L$ (because $\tau'\in\partial^-_*M_L$), we get that $(\tau,\tau')\in\mathcal{N}_L$. 
Clearly, $\mathcal{N}_L\cap\mathcal{N}_{L'}=\emptyset$ for all $L\neq L'$. Thus, 
\begin{equation}
    \big|\!\!\bigcup_{L\in\mathfrak{L}(A)}\!\!\mathcal{N}_L\big|
    =\sum_{L\in\mathfrak{L}(A)}\!|\mathcal{N}_L|
    \ge \delta'\Big(\sum_{L\in\mathfrak{L}(A)}\!\!|L|_t^{(d-1)/d} \Big) \ge \delta'\Big(\big|\!\!\bigcup_{L\in\mathfrak{L}(A)}\!\! L\big|_t^{(d-1)/d}\Big) \ge \delta'\big(c_2|A|\big)^{(d-1)/d}/r^{d-1},
\end{equation}
where the last equality follows from Lemma~\ref{lem:propHats} Part~\ref{lem:propHats:itm2}.

    \smallskip\noindent
    \textbf{Case 2 (Otherwise):} In this case, we assume there is an $L_*\in\mathfrak{L}(A)$ such that $\big|M^c\big|_t>|M|_t$ for all~$M\in\mathfrak{M}(L_*)$.

    First, observe that
    \[
    \big|\mathcal{N}_{L_*}\big|
    \ge \Big|\bigcup_{M\in\mathfrak{M}(L_*)}\big\{(\tau,\tau')\in\mathcal{N}_{L_*} \colon \tau'\in M\big\}\Big| = \sum_{M\in\mathfrak{M}(L_*)}\big|\{(\tau,\tau')\in\mathcal{N}_{L_*} \colon \tau'\in M\}\big|. 
    \]
    
    Recall that for every $K\in\mathfrak{K}_{\smlSucceq}(A)$ there exists an $M\in\mathfrak{M}(L_*)$ such that $K\subseteq M$. We apply Lemma~\ref{lem:pair_of_open_sites_boundary} to each  $M\in\mathfrak{M}(L_*)$ that contains at least one $K\in\mathfrak{K}_{\smlSucceq}(A)$ (recall that both $M$ is *-connected and its complement is a $k$-lattice animal and obtain (even when $K,K' \subseteq M$, the edges counted below are disjoint)
\begin{align*}
\big|\mathcal{N}_{L_*}\big|
& \ge\delta'\Big(
  \sum_{\substack
    {M\in\mathfrak{M}(L_*) \colon \\ \exists K \in \mathfrak{K}_{\smlSucceq}(A),
     K \subseteq M}
  }
    \big|M\big|_t^{(d-1)/d}
\Big) 
\\ & 
\ge \delta'\Big(
  \Big(\sum_{\substack
    {M\in\mathfrak{M}(L_*) \colon \\ \exists K \in \mathfrak{K}_{\smlSucceq}(A),
     K \subseteq M}
  }
      \big|M\big|_t
  \Big)^{(d-1)/d}
\Big) 
\\ &  
\ge \delta'\Big(
  \Big(\,
    \sum_{K \in \mathfrak{K}_{\smlSucceq}(A)} \big|K\big|_t
  \Big)^{(d-1)/d}
\Big).
\end{align*}
To conclude, note that either 
$|\mathfrak{K}(A)|\le |A|^{(d-1)/d}r^{d+1}$ and 
Lemma~\ref{lem:size_union_large_componentes_D_i} implies that
$$|\mathcal{N}_{L_*}|=\delta'c^{(d-1)/d}|\mathcal{L}_n|^{(d-1)/d}/r^{d-1}\ge \delta' \Big(\frac{c|A|}{1-\delta}\Big)^{(d-1)/d}/r^{d-1},$$ or
$|\mathfrak{K}(A)|>|A|^{(d-1)/d}r^{d+1}$, and we are done by Corollary~\ref{cor:starComp} (in this case we obtain $|E(A,A^c)|\ge \xi |A|^{(d-1)/d}r^{d+1}$).

The proof is finished by setting $c_1=\min\{1,\xi,\delta'c_2^{(d-1)/d},\delta'\big(\frac{c}{1-\delta}\big)^{(d-1)/d}\}$.
\end{proof}

\subsection{The small radii regime}
Here we extend the isoperimetric inequality (that is, the statement of Proposition~\ref{prop:isoperimetric_inequality_rgg_large_r}) to the RGG in the regime $(1+\varepsilon)r_g \le r<r'$ for arbitrarily small but fixed $\varepsilon > 0$. We use a renormalization argument to overcome the fact that, for $r<r'$, the proof used in the intermediate radii regime does not apply directly.
Our main goal is to establish the following result: 
\begin{proposition}\label{prop:iso}
    Let $\varepsilon>0$ and $ 0<\delta<1$. There exists $r'$ large enough so that for every $(1+\varepsilon)r_g\le r\le r'$ there exists $c_1,C_1>0$ such that w.h.p.~the following is satisfied : for every connected set~$A \subseteq \mathcal{L}_n$  
    such that $C_1(\log n)^{d/(d-1)}\le |A|\le (1-\delta)|\mathcal{L}_n|$ it holds that $|E(A,A^c)| \ge c_1|A|^{(d-1)/d}$. 
\end{proposition}
To prove the previous theorem, we use ideas developed in the renormalization section of~\cite{mathieu2004isoperimetry}. Before going into the actual proof, we introduce some necessary preliminaries. We will work with $\Lambda_n^{\rho}$ 
where $\rho=2(\frac{3M}{10}-1)r$  
and $M=M(\varepsilon)$ is a large positive constant.
Recall that if $\tau$ is a tile of $\Lambda_n^\rho$, say of center $\bfi\in \Lambda_n\cap\mathbb{Z}^d$, we have 
\[
\tau = \tau_{\bfi} = \rho\cdot\bfi + \Big[-\frac{\rho}{2},\frac{\rho}{2}\Big)^d.
\]
We associate to $\tau$ an \emph{enlarged} tile, denoted $\tau'$ of side length $Mr$, also centered at $\bfi$ and defined as
\[
\tau' = \rho\cdot\bfi + \left[-\frac{Mr}{2}, \frac{Mr}{2}\right)^d
= 2\Big(\frac{3Mr}{10}-1\Big)\cdot\bfi + \left[-\frac{Mr}{2}, \frac{Mr}{2}\right)^d.
\]
Note that the tiles $\tau$ do not overlap, whereas the tiles $\tau'$ do. 

Given a parallelepiped $U\subseteq \Lambda_n$ we say that it has a \emph{crossing component}
of $\mathcal{G}_n$ if there exists a component of the graph induced by $U\cap\mathcal{G}_n$ in $\mathcal{G}_n$ with the property that for each of the $2d$ faces of~$U$ there is at least one vertex of the component at distance at most $r$ from the face. The \emph{Euclidean diameter} of a  component is defined as the maximum Euclidean distance between any two vertices of the component.

Next we introduce a key concept. Given a tile $\lambda$, we say that it is \emph{good} if the following properties hold:
\begin{enumerate}
    \item When dividing the tile $\lambda$ into $2^d$ equal tiles of half the side length of $\lambda$, each of the resulting smaller tiles (intersected with $\Lambda_n$) has a crossing component.
    \item There is a unique giant component in $\lambda$ that contains any other component of Euclidean diameter at least $1/5$ of the side length of $\lambda$. 
\end{enumerate}
We observe that both, the choice of $\Lambda_n^{\rho}$ and the definition of goodness, are similar as the one introduced in~\cite{pisztorapenrose}.

We say that tile $\tau$ of $\Lambda_n^\rho$
is \emph{useful} if both $\tau$ and $\tau'$ are good tiles.

\begin{figure}[htbp]
\centering
\includegraphics[width=0.45\linewidth]{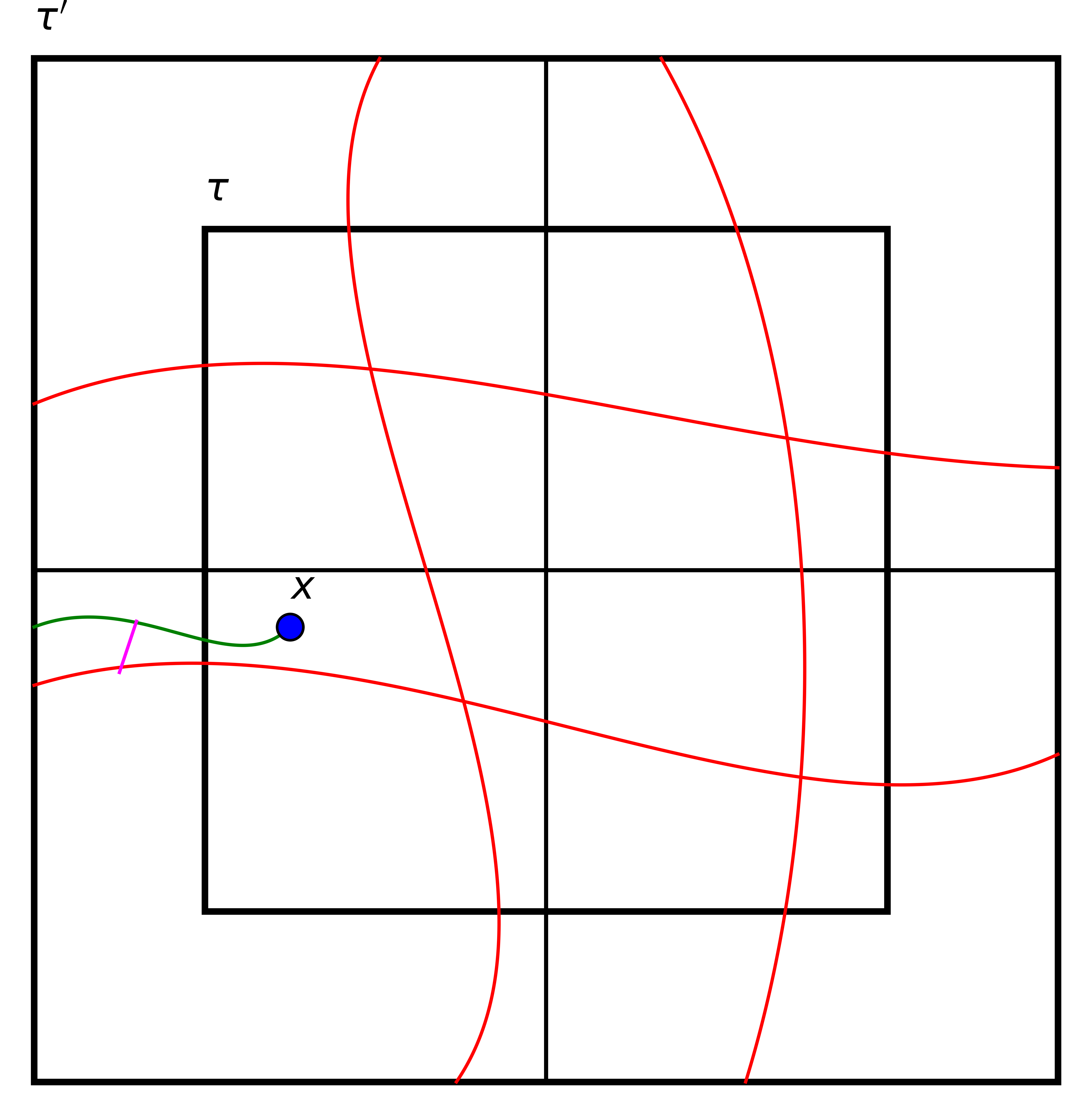}
\caption{Tile $\tau$ is useful, implying that both $\tau$ and $\tau'$ are good. The red lines represent the crossings of the smaller sub-tiles, and thus a part of the largest component inside $\tau'$. $x$ is a vertex of $\mathcal{L}_n$, the green line represents a path to the outside of the tile~$\tau'$. As the part of this path in $\tau'\setminus \tau$ is of length at least one fifth of the side length of $\tau'$, it must belong to the largest component of $\tau'$ (by definition of good tile): the connection to the largest component of $\tau'$ is represented by the magenta line.}
\label{fig:x_in_giant_inside_good_tile}
\end{figure}

\begin{lemma}\label{lem:renormalization_probability_good_cube}
Let $\varepsilon>0$ be constant, $M=M(\varepsilon)$ be sufficiently large,  and $(1+\varepsilon)r_g\le r<r'$ for $r'$ large enough. 
    If $\tau\in\Lambda_n^{\rho}$, then 
    \[
    \mathbb{P}(\text{$\tau$ is useful})\geq 1-\exp(-c'M), 
    \]
    for some constant $c'=c'(\eps,d,r')>0$.
\end{lemma}
\begin{proof}
This follows by a union bound and a straightforward application of Lemma~11 of~\cite{friedrich2013diameter} to both~$\tau$, $\tau'$ and each of their $2^d$ subdivisions, respectively.
\end{proof}

We define a random field with state space in $\{0,1\}^{\Lambda_n^{\rho}}$, where a value of 1 is associated with the sites whose tiles are useful, and a value of 0 with those that are not. We will refer to it as \emph{random field of useful tiles}. Note that this is a finite-range dependent percolation process on $\Lambda_n^{\rho}$,
in the sense that there exists a constant $u>0$ such that for tiles $\tau_{\bfi}, \tau_{\bfj}\in\Lambda_n^\rho$ centered at $\bfi,\bfj\in\mathbb{Z}^d$, respectively, with~$\|\bfi - \bfj\|_1>u$, the events of $\tau$ and $\tau'$ being both useful tiles are independent. Given two random fields~$\phi$ and $\phi'$ on $\{0,1\}^{\Lambda_n^\rho}$, we say that $\phi$ dominates $\phi'$ if $\phi\geq \phi'$ with probability 1. We have the following quantitative domination result.

\begin{lemma}\label{lem:lower_bound_parameter_dominated_bernoulliPercolation}
    Let $\varepsilon>0$ be a constant and $(1+\eps)r_g\leq r<r'$ for $r'$ large enough.  Then, for $M=M(\varepsilon)$ large enough, the random field of useful tiles dominates a percolation process with parameter 
    $$
    p = 1-\exp(-c''M)
    $$
    for some $c''=c''(\eps,d,r')>0$.
\end{lemma}
\begin{proof}
    This follows directly from the quantitative estimate in Theorem 1.3 of \cite{liggett1997domination} together with Lemma~\ref{lem:renormalization_probability_good_cube}.
\end{proof}
In the next lemma we provide several facts that will be useful in the proof of Proposition~\ref{prop:iso}. 
\begin{lemma}\label{lem:vertex_interior_good_cube_in_giant}  
    Let $\tau\in \Lambda_n^{\rho}$ be useful.
    \begin{enumerate}    \item\label{lem:vertex_interior_good_cube_in_giant:itm1} If $x\in\tau\cap\mathcal{L}_n$, then $x$ belongs to the largest component of $\tau'\cap\mathcal{G}_n$.
    \item\label{lem:vertex_interior_good_cube_in_giant:itm2} The largest component of $\tau\cap\mathcal{G}_n$ is contained in the largest component of $\tau'\cap\mathcal{G}_n$.
    \item\label{lem:vertex_interior_good_cube_in_giant:itm4} 
    For adjacent (in $\Lambda_n^\rho$) good tiles $\tau$ and $\widehat{\tau}$, the largest components of $\tau'\cap\mathcal{G}_n$ and $\widehat{\tau}'\cap\mathcal{G}_n$ belong to the same  component in $(\tau'\cup\widehat{\tau}')\cap\mathcal{G}_n$.
    \end{enumerate}
\end{lemma}
\begin{proof}
    We begin with~\ref{lem:vertex_interior_good_cube_in_giant:itm1}: Since $\tau$ is useful, we have that $\tau'$ is a good tile, which implies there is a unique largest component in $\tau'\cap\mathcal{G}_n$. Note that $x$ is at Euclidean distance at least $(Mr-\rho)/2=Mr/5+r$ from the boundary of $\tau'$. Since $x \in \mathcal{L}_n$, there is a path starting at $x$ and going out of $\tau'$. Hence, inside~$\tau'$, there exists $z \in \tau' \cap \mathcal{L}_n$, so that the path from $x$ to $z$ stays inside $\tau'$ and has Euclidean diameter at least $Mr/5$. Thus, because $\tau'$ is good, vertex~$x$ belongs to the largest component inside~$\tau'\cap\mathcal{G}_n$ (see Figure~\ref{fig:x_in_giant_inside_good_tile} for a visual explanation).
    
    To prove~\ref{lem:vertex_interior_good_cube_in_giant:itm2}, note that the largest component of $\tau\cap \mathcal{G}_n$ is a crossing component in $\tau$. Indeed, each crossing component of each subdivion of $\tau\cap \mathcal{G}_n$ belongs to the giant of $\tau\cap \mathcal{G}_n$ (because the Euclidean diameter of each subdivision is at least $\rho/2-2r=(\frac{3M}{10}-3)r\ge \rho/5$ for $M$ large enough). Then, for each face of the tile $\tau$, pick one subdivision of the tile with one face that is contained in the selected face of~$\tau$. Observe that the crossing component of the subdivion has one vertex at distance at most $r$ from this face. Such vertex also belongs to the largest component of $\tau\cap \mathcal{G}_n$ (by this paragraph's opening note). Then, $\tau\cap \mathcal{G}_n$ has a crossing component. 
    Thus, it has Euclidean diameter at least $\rho-2r\ge Mr/5$ for $M$ large enough. Since this is a component inside of $\tau'$ it must belong to the largest component of $\tau'$ by definition of good tile.

    
    To prove~\ref{lem:vertex_interior_good_cube_in_giant:itm4}, note that $(\tau'\cap \widehat{\tau}')\cap \mathcal{G}_n$ has one component of Euclidean diameter at least~$Mr/5$: to see this, consider the largest component of $\tau'\cap \mathcal{G}_n$. As it is a crossing component of $\tau'$, there is a path contained in~$\tau'\cap\widehat{\tau}'$ that crosses (there is one vertex of the path at distance at most $r$ from both faces) the shortest side of $\tau'\cap\widehat{\tau}'$:
    indeed, the latter path has length at least $\frac{2M}{5}r$. As this path is a component that belongs to $\tau'\cap \widehat{\tau}'$ and has Euclidean diameter at least $Mr/5$, it belongs to both largest components of $\tau'\cap \mathcal{G}_n$ and~$\widehat{\tau}'\cap \mathcal{G}_n$. Then both largest components have points in common, so they belong to the same component of $(\tau'\cup\widehat{\tau}')\cap \mathcal{G}_n$.  
\end{proof}
We denote by $F_m$  the largest component of open sites in $\Lambda_n^\rho$, $m=|\Lambda_n^\rho|_t$, of the dominated percolation process with parameter $p$, where $p$ is as in Lemma~\ref{lem:lower_bound_parameter_dominated_bernoulliPercolation}, and with abuse of notation, we will also use~$F_m$ to denote the actual tiles associated with the sites.
\begin{lemma}\label{lem:size_vertices_outside_giantof_usefultiles}
    For any $\gamma>0$, there is an $M'=M'(\gamma)$ large enough such that if $M>M'$, then $|\mathcal{G}_n\cap(\Lambda_n^\rho\setminus F_m)|\leq \gamma n$ w.h.p. 
\end{lemma}
\begin{proof}
    We assume that $\Lambda_n^\rho$ tessellates $\Lambda_n$. Recall that  we denote the volume of a tile in $\Lambda_n^\rho$ by $\mbox{vol}_{\rho}$.
    Clearly, $|\Lambda_n^\rho|_t=n/\mbox{vol}_{\rho}=m$. By Theorem~\ref{thm:giant_Bernoulli_site_percolation}, $|F_m|_t\geq (\theta(p)-\gamma_1)m$ w.h.p., where $p=1-\exp(-c''M)$ (see Lemma~\ref{lem:lower_bound_parameter_dominated_bernoulliPercolation}) and $\gamma_1>0$ can be made as small as desired. Recalling the known fact that $\theta(p)\to 1$ as $p\to 1$, we can make $p$ as close to 1 as desired by making $M>M'$ for some $M'$ very large, so
    \begin{equation}
        |F_m|_t \ge (1-\gamma_2)m 
    \end{equation}
    for $\gamma_2>0$  arbitrarily small. By concentration of Poisson variables, for any positive $p'<1$ there exists~$M'$ large enough such that if $M>M'$, then for $\delta>0$,
    \begin{equation}
        \mathbb{P}(|\mathcal{G}_n\cap \tau|>(1-\delta)\mbox{vol}_{\rho}) \ge p'.
    \end{equation}
    Then, Lemma~\ref{lem:percolation_connected_animals} is applicable to $F_m$: for any fixed $0<\gamma_3<1$, w.h.p.\ there are at least $(1-\gamma_3)|F_m|_t\ge (1-\gamma_3)(1-\gamma_2)m = (1-\gamma_4) m$ tiles of $F_m$ that have at least $(1-\delta)\mbox{vol}_{\rho}$ vertices of $\mathcal{G}_n$ each. Thus, 
    \begin{equation}
    |\mathcal{G}_n\cap F_m| \ge (1-\delta)\mbox{vol}_{\rho}\left((1-\gamma_4)m\right)=(1-\delta)(1-\gamma_4) n, 
    \end{equation}
    and the latter is at least $(1-\gamma)n$ for $\gamma_4$ small enough (which can be done by choosing $\gamma_2$ and $\gamma_3$ sufficiently small). This finishes the proof of the lemma. 
\end{proof}

\begin{proof}[Proof of Proposition~\ref{prop:iso}]


In what follows, we fix $M$ large enough and $n^{1/d}/M \in \mathbb{N}$ so that $\Lambda_n^\rho$ tessellates $\Lambda_n$. Denote by~$T_A$  the subset of tiles in $\Lambda_n^\rho$ that contain at least one vertex of $A$. We note that $T_A$ is a *-component in $\Lambda_n^{\rho}$ (to this end $M$ needs to be larger than $5$ so $\rho\geq r$). By Remark~\ref{rem:bound_vertices_big_animals_2} 
(which we may apply to the dominated percolation process since $M$ is large, and for $r$ constant, we have $(\log n)^{d/(d-1)} = \omega(r^d\log n)=\Omega(\log n)$), 
there is a constant $c_2$ such that w.h.p.
\begin{equation}\label{eq:renormalization_size_upsilon_A}
|T_A|_t\ge c_2|A|.    
\end{equation}
Henceforth, let $U\subseteq\Lambda_n^\rho$ be the collection of useful tiles.
Since the probability $p$ of a tile $\tau$ being useful can be made larger than $p^*$ as in Lemma~\ref{lem:percolation_connected_animals}  (by choosing $M$ large), the random field of useful tiles in~$\Lambda_n^{\rho}$ dominates a supercritical percolation process on~$\Lambda_n^\rho$. 

From Lemma~\ref{lem:vertex_interior_good_cube_in_giant} Part~\ref{lem:vertex_interior_good_cube_in_giant:itm1}, it is clear that if $\tau$ is useful, then any two vertices of $\mathcal{L}_n\cap \tau$ are connected by a path totally contained within $\tau'$. Moreover, if $\tau\in U\cap T_A$ and there exists $y\in \mathcal{L}_n\cap \tau$ such that~$y\notin A$, then there exists at least one edge of $E(A,A^c)$ with both endvertices belonging to~$\tau'$ (to see the latter, note that the path between $x\in A\cap \tau$ and $y\in (\mathcal{L}_n\setminus A)\cap \tau$  that is totally contained within~$\tau'$ must contain one such edge). Let $U_{\mathcal{L}_n\setminus A}$ be the collection of useful tiles that contain a vertex in~$\mathcal{L}_n\setminus A$. 
Then, each tile in $U_{\mathcal{L}_n\setminus A}\cap T_A$ contributes with at least one edge to~$E(A,A^c)$. 
Note that at least a positive  fraction, say $c_3$, of the previous edges are disjoint, as there is limited overlap  between enlarged tiles (any tile $\tau'$ overlaps with only a constant amount of other enlarged tiles, with that constant depending on the dimension only). Hence, if 
we have~$|U_{\mathcal{L}_n\setminus A}\cap T_A|_t\ge |A|^{(d-1)/d}$, it follows that $|E(A,A^c)|\ge c_3 |A|^{(d-1)/d}$.

Suppose otherwise, that is, assume $|U_{\mathcal{L}_n\setminus A}\cap T_A|_t<|A|^{(d-1)/d}$. 
We know that there is a $\underline{c}$ such that w.h.p.~$|\mathcal{L}_n|\geq \underline{c}n$ (by Theorem~\ref{thm:concentration_vertices_edges_rgg}). First, we observe that for any $0<\gamma<1$ we can make $M$ large enough so that w.h.p.
\begin{equation}\label{eq:lower_bound_verticesGiant_inside_giantDominaterBernoulli}
|\mathcal{L}_n\cap F_m|\geq (1-\gamma)|\mathcal{L}_n|.
\end{equation}
Indeed, by Lemma \ref{lem:size_vertices_outside_giantof_usefultiles} it is possible to choose $M$ large enough such that, w.h.p.,
$$|\mathcal{L}_n\cap (\Lambda_n^{\rho}\setminus F_m)|\leq|\mathcal{G}_n\cap (\Lambda_n^{\rho}\setminus F_m)|\leq (\gamma/\underline{c}) \cdot n\leq \gamma |\mathcal{L}_n|,
$$
which immediately yields the assertion.

Since we are working under the assumption that $|U_{\mathcal{L}_n\setminus A}\cap T_A|_t<|A|^{(d-1)/d}$, because the hypothesis $|A|\le (1-\delta)|\mathcal{L}_n|$, and the fact that w.h.p.\ every tile of $\Lambda_n^\rho$ contains $O(\log n)$ vertices of $\mathcal{G}_n$, w.h.p. 
\begin{equation}
|\mathcal{L}_n\cap (U_{\mathcal{L}_n\setminus A} \cap T_A)| = O(|A|^{(d-1)/d}\log n) =o(|\mathcal{L}_n|).  
\end{equation}
Together with the hypothesis $|A|\leq (1-\delta)|\mathcal{L}_n|$, it follows that
$$|\mathcal{L}_n\cap (F_m\cap T_A)| \le \underbrace{|\mathcal{L}_n\cap ((F_m\cap T_A)\setminus U_{\mathcal{L}_n\setminus A})|}_{\le |A|} + |\mathcal{L}_n\cap (U_{\mathcal{L}_n\setminus A}\cap T_A)| \leq (1-\delta+o(1))|\mathcal{L}_n|, 
$$
and by~\eqref{eq:lower_bound_verticesGiant_inside_giantDominaterBernoulli}, this means that, w.h.p., 
\begin{equation}\label{eq:lower_bound_vertices_outside_tiles_TA}
|\mathcal{L}_n\cap (F_m\setminus T_A)|= |\mathcal{L}_n \cap F_m|-|\mathcal{L}_n\cap (F_m\cap T_A)| \ge (\delta-\gamma - o(1))|\mathcal{L}_n|. 
\end{equation}
We may choose $\gamma$ small enough such that the last expression is larger than $c_4n$ for some positive constant~$c_4$ for $n$ large enough. 

Denote by $\mathfrak{D}$ the collection of components of $F_m\setminus T_A$. Now, consider a pair of adjacent (in $\Lambda_n^\rho$) useful tiles $(\tau,\widehat{\tau})$ such that $\tau\in T_A$ and $\widehat{\tau}\in T_A^c$. By Lemma~\ref{lem:vertex_interior_good_cube_in_giant} Part~\ref{lem:vertex_interior_good_cube_in_giant:itm4}, there is some component of $(\tau'\cup \widehat{\tau}')\cap \mathcal{G}_n$ that contains the largest components of both $\tau'\cap \mathcal{G}_n$ and $\widehat{\tau}'\cap \mathcal{G}_n$. Hence, any vertex belonging to the largest component of $\tau'\cap \mathcal{G}_n$ is connected to any vertex of the largest component of $\widehat{\tau}'\cap \mathcal{G}_n$ by a path totally contained within $\tau'\cup \widehat{\tau}'$. 
In order to derive the desired lower bound on the number of edges from $A$ to $\mathcal{L}_n\setminus A$, suppose then that tile $\tau$ contains at least one vertex of~$A$ and the other tile $\widehat{\tau}$ does not.
Then, pick $x\in A\cap \tau$ ($x$ belongs to the largest component of $\tau'$ by Lemma~\ref{lem:vertex_interior_good_cube_in_giant} Part~\ref{lem:vertex_interior_good_cube_in_giant:itm1}) and $y\in \widehat{\tau}\cap (\mathcal{L}_n\setminus A)$ that belongs to the largest component of $\widehat{\tau}'$ (there exists such~$y$ by definition of useful tile and because $\widehat{\tau}\in T^c_A$). 
We can associate to $(\tau,\widehat{\tau})$ one edge of $E(A,A^c)$ whose endvertices are in $\tau'\cup\widehat{\tau}'$ (since $x\in A$ and $y\notin A$, such an edge exists in the path between $x$ and $y$ contained in $\tau'\cup\widehat{\tau}'$).
We remark that the edge associated to~$(\tau,\widehat{\tau})$ could 
also be associated to some other pair of adjacent useful tiles as a consequence of non-empty overlap, but this happens at most a constant number of times (with a constant depending on the dimension only).

We claim that $|\mathfrak{D}|\ge |A|^{(d-1)/d}$ implies the stated proposition. Indeed, consider $D,D'\in\mathfrak{D}$, $D\neq D'$, and a path of tiles in $F_m$ between a tile in $D$ and another in $D'$ (such tiles exist by definition of $\mathfrak{D}$). The latter path must intersect $T_A$, as $D$ and $D'$ are components of $\Lambda_n^{\rho}\setminus T_A$. Then, there is one pair $(\tau,\widehat{\tau})$ of useful tiles with $\tau \in D$ and $\widehat{\tau} \in T_A$ for each $D\in \mathfrak{D}$, and by the argument of the previous paragraph this means that there is an edge of $E(A,A^c)$ in $\tau\cup \widehat{\tau}$, and thus $|E(A,A^c)|\ge c_5|\mathfrak{D}|$ (the constant $c_5>0$ depends only on $d$, here we considered the overlap of enlarged tiles in the last inequality). This finishes the proof of the claim. 

We now address the remaining case, this is $|\mathfrak{D}|<|A|^{(d-1)/d}$. 
Recall that the at least $c_4 n$ vertices of $|\mathcal{L}_n \cap (F_m\setminus T_A)|$ (see~\eqref{eq:lower_bound_vertices_outside_tiles_TA}) are distributed among the components of $\mathfrak{D}$. Denote by $\mathfrak{D}_{\smlSucceq}$ the elements of $\mathfrak{D}$ that have at least $c_6(\log n)^{d/(d-1)}$ vertices, for some constant $c_6>0$ large enough. We claim that there is a constant $c_7>0$ such that w.h.p.
\begin{equation}\label{eq:lower_bound_size_vertices_in_giant_bigcomponents_mathfrak_D}
    \sum_{D\in \mathfrak{D}_{\smlSucceq}}|\mathcal{L}_n\cap D| \ge c_7 n.
\end{equation}
To see this, recall that w.h.p.\ every tile contains  $O(\log n)$ vertices. Hence, by hypothesis on $|A|$ and Theorem~\ref{thm:concentration_vertices_edges_rgg}, the claim follows since, w.h.p.,
\begin{equation}
    \sum_{D\in \mathfrak{D}\setminus \mathfrak{D}_{\smlSucceq}} |\mathcal{L}_n\cap D|= O(|\mathfrak{D}|\log n) = O(|A|^{(d-1)/d}\log n) =o(n).
\end{equation}
We apply Remark \ref{rem:bound_vertices_big_animals_2} to each $D\in\mathfrak{D}$ (we set $c_6$ large enough) so that $|D|_t\ge c_2|\mathcal{L}_n\cap D|\ge c_2c_6(\log n)^{d/(d-1)}$ (we used the hypothesis $r=\Theta(1)$). Recalling that elements of $\mathfrak{D}$ belong to $F_m\setminus T_A$ and are pairwise disjoint, using~\eqref{eq:lower_bound_size_vertices_in_giant_bigcomponents_mathfrak_D}, we obtain 
\begin{equation}
    |F_m\setminus T_A|_t\ge \sum_{D\in\mathfrak{D}_{\smlSucceq}}|D|_t \ge \sum_{D\in \mathfrak{D}_{\smlSucceq}} c_2|\mathcal{L}_n \cap D| \ge c_2c_7 n.
\end{equation} 
Thus, we may apply Lemma~\ref{lem:pairs_of_open_sites_no_boundary_condition} to $T_A$: there is a constant $c_8$ such that w.h.p.~we have that there are at least $c_8|T_A|_t^{(d-1)/d}$ disjoint edges in $E(T_A,T_A^c)$ both of whose endvertices are open sites. Since the random field of useful tiles dominates the latter percolation process, we also obtain at least $c_8|T_A|_t^{(d-1)/d}$ disjoint edges in $E(T_A,T_A^c)$ both of whose endvertices (tiles) are useful. 
Hence, for some constant $c_9=c_9(d)$ that accounts for the overlap of enlarged tiles, 
\[
|E(A,A^c)| \ge c_9|E(T_A,T_A^c)| \ge c_9c_8 |T_A|_t^{(d-1)/d} \ge c_9c_8c_2|A|^{(d-1)/d},
\]
where the last equality is by~\eqref{eq:renormalization_size_upsilon_A}.

To finish the proof of the proposition we set $c_1=\min \{c_3,c_5,c_9c_8c_2\}$.
\end{proof}

\section{Lower Bound on the relaxation time}\label{sec:lowerBnd}
It remains to show the lower bound of Theorem~\ref{thm:main_thm_bounds_over_mixing_rgg}: the proof is considerably easier and similar to the proof already given in~\cite{benjamini2003mixing}. The approach is based on the fact that the mixing time is bounded from below by the relaxation time, a quantity that is generally more tractable to analyze.
The next stated result is key in our derivation of the lower bound for $\tau_{\textup{mix}}(\mathcal{L}_n)$.
\begin{lemma}{\cite[Lemma 2.2]{benjamini2003mixing}}
\label{lem:lowerBnd:formula}If $G$ is a connected graph, then 
\[
\tau_{\textup{rel}}(G) \geq \max_{v \in G} \left( \pi(D^2_v) - \pi^2(D_v) \right) = \max_{v\in G} \left( \pi(\pi(D_v)-D_v)^2 \right),
\]
where $D_v(x)=d_G(v,x)$ is the graph distance in $G$ between $v$ and $x$, and $\pi(f)$ is the expected value of~$f$ with respect to the stationary distribution $\pi$ of the random walk in $G$, that is, 
\[
\pi(f) = \sum_{v\in G}\pi(v)f(v).
\]
\end{lemma}
The previous result shows that one promising way to prove a matching lower bound for $\tau_{\textup{rel}}(\mathcal{L}_n)$ requires identifying a vertex $v \in \mathcal{L}_n$ whose graph distance to a vertex $x$, sampled from the stationary distribution of the random walk in $\mathcal{L}_n$, deviates substantially from its mean.
Since any vertex of $\mathcal{L}_n$ should yield, up to constant factors, comparable bounds in Lemma~\ref{lem:lowerBnd:formula}, we choose for convenience a vertex located close to the origin (the center of $\Lambda_n$).
To estimate its graph distances to all other vertices in the giant component, we invoke the following result, which links graph distance in $\mathcal{L}_n$ to the Euclidean distance between the corresponding points in $\Lambda_n$.
\begin{lemma}{\cite[Theorem 3]{friedrich2013diameter}}\label{lem:chemical_distance_to_origin_rgg}
If $d\ge 2$ and $r>r_g$ 
for any two vertices $x,y\in\mathcal{G}_n^r$ belonging to the same component and such that $\|x - y\|_2 = \omega(\log n/r^{d-1})$, the graph distance between~$x$ and $y$ is $O(\lceil\|x -y\|_2/r\rceil)$ with probability $1-O(n^{-1})$, .
\end{lemma}
We are now equipped with all the necessary ingredients to prove the lower bound in the main result of this manuscript.
For a suitably small $\varepsilon >0$, we consider the set of vertices in $\mathcal{G}_n$ lying within~$\ell_\infty$ distance at most $\varepsilon n$ from the boundary of $\Lambda_n$. 
We then show that, w.h.p., $\mathcal{L}_n$ contains~$\Theta(n)$ such vertices, each of which is at graph distance at least $\xi n^{1/d} / r$ (for some $\xi > 0$) from any vertex sufficiently close to the origin.
Next, we consider vertices whose $\ell_\infty$-distance from the origin is small — but not too small - and show that, w.h.p., $\mathcal{L}_n$ also contains~$\Theta(n)$ such vertices, each at graph distance at most~$\xi'n^{1/d} / r$ (for some $\xi' > 0$ significantly smaller than $\xi$) from any vertex near the origin.
Taken together, these two observations imply that the graph distance from a fixed vertex $v \in \mathcal{L}_n$ near the origin to a vertex~$x \in \mathcal{L}_n$ sampled from the stationary distribution must deviate substantially from its mean.
\begin{proof}[Proof (of lower bound on the relaxation time in
Theorem~\ref{thm:main_thm_bounds_over_mixing_rgg})]
Recall that for $x,y\in\mathbb{R}^d$ we have $\|x-y\|_2\leq d\|x-y\|_{\infty}$ and also that $r_g$ is strictly positive for all $d$.
Hence, by Lemma~\ref{lem:chemical_distance_to_origin_rgg}, there exists $\zeta\ge 1$ (depending on $d$ and~$r$ but independent of $n$)
such that, w.h.p.~(for sufficiently large $n$), for any two vertices $x,y\in\mathcal{L}_n$ for which  
$\|x-y\|_2=\omega(\log n)$, it holds that 
$d_G(x,y)\leq \zeta\|x-y\|_\infty/r$.

Now, with hindsight, let $s=\lceil 3\zeta\rceil+2$.
Thus, $s-1-3\zeta\ge 1$.
Next, let $\bfI=\{-s,..,s\}^d$ and fix~$\rho$ so that 
  $\rho=n^{1/d}/(2s+1)$.
Consider the tessellation $\Lambda_n^\rho=\{\tau^\rho_\bfi \colon \bfi\in\bfI\}$ of $\Lambda_n$ (each~$\tau^\rho_\bfi$ is an axis aligned (hyper)cube of side length $\rho$; in particular, tile $\tau_{\mathbf{0}}$ is centered at the origin of $\mathbb{R}^d$).
Since $s$ is independent of $n$, it follows that~$|\,\bfI\,|=\Theta(1)$.

Recall that for a graph $G$ and a subset $S$ of its vertices, we denote by $G[S]$ the subgraph of $G$ induced by~$S$.
We assert that, w.h.p., for all $\bfi\in\bfI$ we have~$|E(\mathcal{L}_n[\tau_\bfi^\rho])|=\Omega(|E(\mathcal{L}_n)|)$ (in particular, $\mathcal{L}_n[\tau_\bfi^\rho]$ is a non-empty graph):
by Theorem~\ref{thm:concentration_vertices_edges_rgg} together with a union bound over all $s^d$ values of $\bfi$, w.h.p., 
there is a component $\widetilde{\mathcal{L}}_n^{\bfi}$ in $\mathcal{G}_n[\tau_\bfi^\rho]$  such that $|E(\widetilde{\mathcal{L}}_n^{\bfi})|=\Theta((n/s^d)r^d)=\Theta(nr^d)=\Theta(|E(\mathcal{L}_n)|)$. 
The $\widetilde{\mathcal{L}}_n^{\bfi}$'s and $\mathcal{L}_n$ must be 
  connected in~$\mathcal{G}_n$ (otherwise we will get a contradiction with Theorem~\ref{thm:size_second_component_rgg}).
Hence, $\widetilde{\mathcal{L}}_n^{\bfi}\subseteq\mathcal{L}_n[\tau^\rho_{\bfi}]$, implying that 
  $|E(\mathcal{L}_n[\tau^\rho_{\bfi}])|\geq |E(\widetilde{\mathcal{L}}_n^{\bfi})|=\Theta(|E(\mathcal{L}_n)|)$, which proves our assertion.

Now, let $v$ be a vertex of $\mathcal{L}_n[\tau^\rho_{\mathbf{0}}]$ where $\mathbf{0}=(0,...,0)\in\bfI$ (such a $v$ exists w.h.p.~because $\mathcal{L}_n[\tau_{\mathbf{0}}^\rho]$ is, by the above discussion, w.h.p., a non-empty graph).
Since~$v$ is in $\tau^\rho_{\mathbf{0}}$, it holds that $\|v\|_\infty\le \rho/2=n^{1/d}/(4s+2)$.  By Lemma~\ref{lem:lowerBnd:formula}, in order to find a lower bound on $\tau_2(\mathcal{L}_n)$ (thence, also on $\tau_{\text{mix}}(\mathcal{L}_n)$), it is enough to find a lower bound on $\pi(\pi(D_v)-D_v)^2$.

Define $\overline{V}$ as the set of vertices $x\in\mathcal{L}_n[\tau^\rho_{\bfi}]$ for some $\bfi\in\bfI$ such that 
$\|\bfi\|_\infty=s$. 
Thus, if $x\in\overline{V}$, then $\|x\|_{\infty}\geq (2s-1)n^{1/d}/(4s+2)$.
Since the endvertices of an edge of~$\mathcal{G}_n$ are at Euclidean distance at most~$r$ from each other, for every $x\in\overline{V}$, we have
\[
D_{v}(x) \geq \|x-v\|_2/r \geq \|x-v\|_\infty/r 
\ge (\|x\|_\infty-\|v\|_\infty)/r
\ge 
\tfrac{s-1}{2s+1}n^{1/d}/r.
\]
Now we distinguish two cases. First, assume that $\pi(D_{v})<\frac{s-1+3\zeta}{4s+2}n^{1/d}/r$. Then, for all $x\in\overline{V}$, we get that
\[
(\pi(D_{v})-D_{v}(x))^2 \geq \big(\tfrac{s-1-3\zeta}{4s+2}\big)^2 n^{2/d}/r^2 
=\Omega(n^{2/d}/r^2),
\]
and
\begin{align*}
\pi\big(\pi(D_v)-D_v\big)^2
& = \Omega\Big(\frac{n^{2/d}}{r^2}\sum_{x\in\overline{V}}\frac{\deg(x)}{2|E(\mathcal{L}_n)|}\Big) \\
& = \Omega\Big(\frac{n^{2/d}}{r^2}\sum_{\bfi\in\bfI \colon \|\bfi\|_\infty=s}\sum_{x\in\overline{V}\cap\tau^\rho_{\bfi}}\frac{\deg(x)}{2|E(\mathcal{L}_n)|}\Big)
\end{align*}
which equals $\Omega(n^{2/d}/r^2)$  since $\sum_{x\in \overline{V}\cap\tau^\rho_{\bfi}}\deg(x)\geq |E(\mathcal{L}_n[\tau^\rho_{\bfi}])|$ for all $\bfi\in\bfI$, and recalling that, w.h.p., 
$|E(\mathcal{L}_n[\tau^\rho_{\bfi}])|=\Omega(|E(\mathcal{L}_n)|)$ for all $\bfi\in\bfI$. 

Now, assume $\pi(D_{v})\ge \frac{s-1+3\zeta}{4s+2}n^{1/d}/r$.
Define $\underline{V}$ as the set of vertices $x\in\mathcal{L}_n\cap\Lambda_n^{\bfi}$ for some $\bfi\in\bfI$ such that $\|\bfi\|_\infty=2$.
Thus, if $x\in\underline{V}$, then $3n^{1/d}/(4s+2)\le\|x\|_\infty\le 5n^{1/d}/(4s+2)$.
Since $\|v\|_\infty\leq n^{1/d}/(4s+2)$, for every $x\in\underline{V}$ it holds that $\|x-v\|_2\geq \|x-v\|_\infty\ge\|x\|_\infty-\|v\|_\infty\ge \frac{1}{2s+1}n^{1/d}=\omega(\log n)$, so by our choice of $\zeta$, w.h.p., for every $x\in \underline{V}$,
\[
D_{v}(x) 
\le \zeta\|x-v\|_\infty/r
\le \zeta(\|x\|_\infty+\|v\|_\infty)/r
\le 
\tfrac{3\zeta}{2s+1}n^{1/d}/r,
\]
and 
\[
(\pi(D_{v})-D_{v}(x))^2 \geq \big(\tfrac{s-1-3\zeta}{4s+2}\big)^2n^{2/d}/r^2
= \Omega(n^{2/d}/r^2).
\]
Hence, 
\begin{align*}
\pi\big(\pi(D_v)-D_v\big)^2
= \Omega\Big(\frac{n^{2/d}}{r^2}\sum_{x\in\underline{V}}\frac{\deg(x)}{2|E(\mathcal{L}_n)|}\Big) 
= \Omega\Big(\frac{n^{2/d}}{r^2}\sum_{\bfi\in\bfI : \|\bfi\|_\infty=2}\;\sum_{x\in\underline{V}\cap \tau^\rho_{\bfi}}\frac{\deg(x)}{2|E(\mathcal{L}_n)|}\Big),
\end{align*}
which equals $\Omega(n^{2/d}/r^2)$ since w.h.p.~$\underline{V}\cap\tau^\rho_{\bfi}\supseteq \mathcal{L}_n[\tau^\rho_\bfi]$, $\sum_{x\in \underline{V}\cap\tau^\rho_{\bfi}}\deg(x)\geq |E(\mathcal{L}_n[\tau^\rho_{\bfi}])|$ for all $\bfi\in\bfI$, and recalling that, w.h.p., 
$|E(\mathcal{L}_n[\tau^\rho_{\bfi}])|=\Omega(|E(\mathcal{L}_n)|)$ for all $\bfi\in\bfI$. 
\end{proof}

\section{The one-dimensional case}\label{sec:1D}
In this section we extend our main result (Theorem~\ref{thm:main_thm_bounds_over_mixing_rgg}) to the one-dimensional setting. 
The proof of the upper bound on the mixing time for $d=1$ is considerably simpler than in the general case~$d>1$. 
On the other hand, the proof of the lower bound carries over essentially unchanged from the higher-dimensional setting.
\begin{proof}[Proof (of Theorem~\ref{thm:main_thm_bounds_over_mixing_rgg} for $d=1$)]
 First, we establish the upper bound on the mixing time. It is well known that $r_c=(1+o(1))\log n$ (see for example Theorem 1 of~\cite{AppelRusso}). Next, for $d=1$, also $r_g=(1+o(1))r_c$: to see this, let $s \in \mathbb{N}$ be arbitrarily large and consider say $r=r_c-C \log \log n$ for $C$ large. Partition $[-\frac{n}{2},\frac{n}{2})$ into $s$ intervals of size $n/s$. By standard Chernoff bounds together with a union bound, w.h.p.\ the number of vertices in each such interval is at most $2n/s$. Also, by another union bound, w.h.p., in each such interval there will be at least one isolated vertex. Thus, w.h.p., the largest component will have at most $4n/s$ vertices, and since $s$ can be made arbitrarily large, for such $r$ there is no giant component, and thus $r_g=(1+o(1))r_c$.


Assume first that $A\subseteq\mathcal{L}_n$ is such that $|A|\le \varepsilon_1 r$ for small enough $\varepsilon_1 > 0$. Clearly, every vertex has at least a constant fraction of its $\Theta(r)$ neighbors outside $A$, and thus there are also $\Theta(|A|r)$ edges between $A$ and $A^c$. 

Next, assume that $A \subseteq \mathcal{L}_n$ is a connected set consisting of consecutive vertices on the line (that is, when considering vertices from left to right, first there is a set of vertices of $A^c$, then a set of vertices in~$A$, and then again in $A^c$ (the first or third set can be potentially empty, but not both simultaneously).
Then, consider overlapping intervals of size $r-C$ for some large constant $C > 0$, such that the $i$-th such interval is at the position of the $(i-1)$-th interval shifted by $C$ to the right (recall that $r\ge (1+\eps)\log n$).  Once again by standard Chernoff bounds and a union bound over all such intervals, w.h.p.~each such interval contains at least $\eps_1 r$ vertices for some $\eps_1=\eps_1(\eps) > 0$ small enough and condition on this. 
We claim that if $\eps_1 r \le |A| \le (1-\delta)|\mathcal{L}_n|$, then the rightmost $(\eps_1/2)r$ vertices of $A$ will all connect by an edge to each of the next $(\eps_1/2)r$ vertices to the right and outside $A$ (if there are only $o(r)$ vertices outside $A$, consider the leftmost $(\eps_1/2)r$ vertices of $A$, and the argument is analogous;  since $|A|\le(1-\delta)|\mathcal{L}_n|$, at least on one side, for $\eps_1$ sufficiently small, there must be at least~$(\eps_1/2)r$ vertices of $A^c$ remaining): 
to see this, consider the first, starting from the rightmost interval and sweeping towards the left, interval $I$ among the mentioned intervals of size $r-C$ containing more than $(\eps_1/2)r$ vertices from $A$, and let $I'$ be the interval $I$ shifted by $C$ to the right. 
In $I'$, by the above, there are at least $(\eps_1/2)r$ vertices from $A^c$. 
All the at least $(\eps_1/2)r$ vertices from $A$ inside $I' \cup I$ are all connected by an edge to all vertices inside $A^c\cap I'$ (because $I\cup I'$ has length $r$), and hence, the minimum number of edges between any such set $A$ to $A^c$ is thus w.h.p.~$\Omega(r^2)$.

Now, we consider general connected sets $A$ satisfying $|A|\ge \eps_1 r$. Similarly as before, sweeping from right to left, let again $I$ be the first interval containing a vertex from $A$ that is not among the $(\eps_1 /2)r$ rightmost ones of $A$. All intervals to the right of $I$, by definition, then contain at most $(\eps_1/2)r$ vertices of $A$ and at least $(\eps_1 /2)r$ vertices of $A^c$. In particular, all vertices in $I \cap A$ are connected to all vertices in $I' \cap A^c$, where $I'$ is the interval $I$ shifted to the right by $C$, and for all vertices in $A$ that are in an interval $J$ to the right of $I$, they all have at least $(\eps_1/2)r$ vertices in $A^c\cap J$ to which they are connected.\footnote{If we were in the case $r=\Theta(n)$ and $|A|=\Theta(n)$, and $A$ consists of at least $(\eps_1/2)r$ vertices both at the beginning and at the end of the line, then the desired $\Theta(n^2)$ edges between $A$ and $A^c$ will be found by considering the first $(\eps_1/2)r=\Theta(n)$ vertices belonging to $A^c$: all of them must have $\Theta(n)$ neighbors in $A$ to the left.}
 Hence, as before, each of the rightmost $(\varepsilon_1/2)r$ vertices from $A$ all have at least $(\eps_1/2)r$ neighbors in $A^c$. Hence, the minimum number of edges between $A$ and $A^c$ is w.h.p.\ also $\Omega(r^2)$.

In conclusion, for $t =O(r/n)$, we have w.h.p.\ $\varphi(t)=\Omega(1)$, whereas by the above, for $t =\Omega(r/n)$, w.h.p.\ $\varphi(t) = \Omega(r/tn)$. In both cases, we have by concentration of degrees, w.h.p.\ each set $A$ is such that $\pi(A)=\Theta(|A|/n)$). 
By Theorem~\ref{thm:upper_bound_mixing_lovasz_kannan}, thus w.h.p.\,
$$\tau_{\textup{mix}}(G)=O\Big(\int_{\pi_0}^{1/2} \frac{1}{t\varphi^2(t)}dt + \frac{1}{\varphi(1/2)}\Big)=O(n^2/r^2),$$
where we used that $1/\varphi(1/2)=O(n/r)=O(n^2/r^2)$, $\pi_0=\Theta(r/n)$, and that for $C$ large enough we have
$$
\int_{\pi_0}^{1/2} \frac{1}{t\varphi^2(t)}dt=O\Big(\int_{\pi_0}^{C r/n}\frac{1}{t}dt+ \int_{Cr/n}^{1/2} \frac{n^2t}{r^2}dt\Big)=O(n^2/r^2),
$$
finishing the proof.

Next, we prove the lower bound on the relaxation time: we claim that the proof of the lower bound of Theorem~\ref{thm:main_thm_bounds_over_mixing_rgg} given in Section~\ref{sec:lowerBnd} applies without modification to the $d=1$ case provided, as we assert below:
first, we note that Lemma~\ref{lem:lowerBnd:formula} holds for all $d\in\mathbb{N}$.
On the other hand, 
Lemma~\ref{lem:chemical_distance_to_origin_rgg} was stated in~\cite{friedrich2013diameter} for $d\geq 2$.
For $d=1$ an even stronger version holds: indeed, for any two vertices $x$ and $y$ in the same connected component of $\mathcal{G}_n^{r,1}$, the graph distance between $x$ and $y$  is at most $2|x-y|/r+1$ (since the first and last vertex among any three consecutive ones in a shortest path between $x$ and $y$ must be at distance at least $r$ between them). Theorems~\ref{thm:concentration_vertices_edges_rgg} and~\ref{thm:size_second_component_rgg} remain valid when $d=1$, because as already observed in this proof's first paragraph, in this setting, both $r_g$ and~$r_c$ equal $(1+o(1))\log n$, and so standard concentration arguments together with a union bound imply the statement of Theorem~\ref{thm:concentration_vertices_edges_rgg}, 
while Theorem~\ref{thm:size_second_component_rgg} is obviously satisfied given that there is only one component for $r \ge (1+\eps)r_c$.
\end{proof}

\bibliographystyle{abbrv}
\bibliography{ref}

\end{document}